\documentclass[12pt]{article}
\usepackage{amsmath,amssymb,latexsym, amsfonts, amscd, amsthm}
\usepackage{graphicx,epsfig}

\theoremstyle{plain}
\newtheorem{theorem}{Theorem}[section]

\newtheorem{proposition}[theorem]{Proposition}
\newtheorem{corollary}[theorem]{Corollary}

\newtheorem{lemma}[theorem]{Lemma}

\theoremstyle{definition}
\newtheorem{definition}[theorem]{Definition}
\newtheorem{remark}[theorem]{Remark}

\newcommand{\lra}{\longrightarrow}
\newcommand{\ra}{\rightarrow}

\newcommand{\Spec}{\mbox{Spec}}
\newcommand{\Spf}{\mbox{Spf}}
\newcommand{\Spm}{\mbox{Spm}}

\newcommand{\Iw}{\mbox{Iw}}

\newcommand{\GL}{{\rm \mathbf{GL}}}

\newcommand{\SL}{{\rm {\mathbf{SL}}}}

\newcommand{\Z}{{\mathbb Z}}
\newcommand{\Q}{{\mathbb Q}}
\newcommand{\C}{{\mathbb C}}

\newcommand{\F}{{\mathbb F}}
\newcommand{\N}{{\mathbb N}}

\newcommand{\WW}{{\mathbb W}}

\newcommand{\qbar}{{\bar\Q}}
\newcommand{\uS}{{\underline{S}}}

\newcommand{\Kbar}{{\overline{K}}}

\newcommand{\Rbar}{{\overline{R}}}
\newcommand{\Nbar}{{\overline{N}}}
\newcommand{\hR}{\widehat{{\overline{R}}}}

\newcommand{\uX}{{\underline{X}}}
\newcommand{\ucX}{{\underline{\cX}}}
\newcommand{\uY}{{\underline{Y}}}
\newcommand{\ucY}{{\underline{\cY}}}
\newcommand{\uZ}{{\underline{Z}}}
\newcommand{\ucZ}{{\underline{\cZ}}}
\newcommand{\ucU}{{\underline{\cU}}}
\newcommand{\uW}{{\underline{W}}}
\newcommand{\ucV}{{\underline{\cV}}}

\newcommand{\cM}{{\cal M}}
\newcommand{\cH}{{\cal H}}
\newcommand{\cC}{{\cal C}}

\newcommand{\cU}{{\cal U}}

\newcommand{\cT}{{\cal T}}
\newcommand{\cF}{{\cal F}}
\newcommand{\cG}{{\cal G}}
\newcommand{\cE}{{\cal E}}
\newcommand{\cV}{{\cal V}}
\newcommand{\cA}{{\cal A}}

\newcommand{\cD}{{\cal D}}
\newcommand{\cO}{{\cal O}}
\newcommand{\hO}{\widehat{\cO}}

\newcommand{\Sh}{{\mbox{Sh}}}

\newcommand{\cW}{{\cal W}}
\newcommand{\cZ}{{\cal Z}}

\newcommand{\fX}{{\mathfrak{X}}}

\newcommand{\fZ}{{\mathfrak{Z}}}

\newcommand{\Symm}{{\rm Sym}}

\newcommand{\rH}{{\rm H}}

\newcommand{\ho}{\hat{\otimes}}

\newcommand{\cY}{{\cal Y}}
\newcommand{\cX}{{\cal X}}

\begin{document}

\title{Overconvergent Eichler-Shimura isomorphisms}
\author{Fabrizio Andreatta \\ Adrian Iovita\\ Glenn Stevens}
\maketitle

\tableofcontents \pagebreak
\section{Introduction}
\label{sec:intro}

\bigskip
\noindent
We start by fixing a prime integer $p>2$, a complete discrete valuation field $K$ of characteristic $0$, ring of integers $\cO_K$ and residue field $\F$, a perfect field of characteristic $p$.

\bigskip
\noindent
Let us first recall the classical Eichler-Shimura isomorphism. We fix $N\ge 3$ an integer not divisible by $p$ and
let $\Gamma:=\Gamma_1(N)\cap \Gamma_0(p)\subset \SL_2(\Z)$. Let us denote by
$X:=X(N,p)$ the modular curve over $\Spec(\Z[1/(Np)])$ which classifies generalized elliptic curves with $\Gamma$-level
structure, $E\rightarrow X$ the universal semi-abelian scheme and $\omega:=\omega_{E/X}=e^\ast(\Omega^1_{E/X})$
the invertible sheaf on $X$ of invariant $1$-differentials, where
$e:X\rightarrow E$ denotes the zero-section. With these notations we have

\begin{theorem} [see \cite{deligne}] For every $k\in \Z$, $k\ge 0$ we have a natural isomorphism compatible with the
action of Hecke operators
$$
\rH^1\bigl(\Gamma, V_{k,\C}\bigr) \cong \rH^0\bigl(X_{\C}, \omega^{k+2}\bigr)\oplus \overline{\rH^0\bigl(X_{\C}, \omega^k\otimes\Omega^1_{X/\C}\bigr)},
$$
where $V_{k,\C}$ is the natural $\Gamma$-representation $V_{k,\C}:={\Symm}^k(\C^2)$ and the overline on the second
term on the right means "complex conjugation".
\end{theorem}

The elements of $\rH^1\bigl(\Gamma, V_{k,\C}\bigr)$ are called (classical) weight $k$  {\em modular symbols}
while the elements appearing on the right hand side of the Eichler-Shimura isomorphism
are (classical) {\em modular, respectively cusp forms} of weight $k+2$.

\bigskip
\noindent
There is a more arithmetic version of the above theorem, which we will also call a classical Eichler-Shimura isomorphism.
Namely let us consider now the modular curve $X$ over the $p$-adic field $K$ and for $k\ge 0$ an integer, we let
$V_k:={\Symm}^k(\Q_p^2)$ with its natural action of $\Gamma$. Then $\rH^1\bigl(\Gamma, V_k\bigr)$ can be seen as an
\'etale cohomology group over $Y_{\overline{K}}:=(X-\{\mbox{cusps}\})_{\Kbar}$ (see section \S 5 for more details), this $\Q_p$-vector space is endowed
both with a natural action of the Galois group $G_K:={\rm Gal}(\Kbar/K)$ and a commuting action of the Hecke operators.

\begin{theorem}[\cite{faltingsmodular}]
\label{thm:faltings}
 We have a natural, $G_K$ and Hecke equivariant isomorphism
$$
\rH^1\bigl(\Gamma, V_k\bigr)\otimes_K\C_p\cong  \Bigl(\rH^0\bigl(X, \omega^{k+2}\bigr)\otimes_K\C_p\Bigr)\oplus
\Bigl(\rH^1\bigl(X, \omega^{-k}\bigr)\otimes_K\C_p(k+1)\Bigr),
$$
where $\C_p$ is the $p$-adic completion of $\Kbar$ and $(k+1)$ referes to a Tate twist.
\end{theorem}

\bigskip
\noindent
In this article we are mainly concerned with the $p$-adic variation of modular forms and modular symbols,
and in fact with the relationship between these two variations.

The parameter space for the above mentioned variations, denoted  $\cW$ and
called {\em the weight space}, is the rigid analytic space associated to
the complete noetherian semilocal algebra $\Lambda:=\Z_p[\![\Z_p^\times]\!]$.  We set $T_0:=\Z_p^\times\times\Z_p$,
seen as a compact subset of $\Z_p^2$, endowed with a natural action of the compact group
$\Z_p^\times$ and of the Iwahori subgroup of $\GL_2(\Z_p)$. If $k\in \cW(K)$ is a weight, we denote by
$D_k$ the $K$-Banach space of analytic distributions on $T_0$, homogeneous of degree $k$ for the action of
of $\Z_p^\times$. Then $D_k$ is a $\Gamma$-representation. The same construction can be performed in a
slightly more
complicated situation: let $U\subset \cW$ be a wide open disk defined over
$K$, let $\Lambda_U$ denote the $\cO_K$-algebra of bounded rigid analytic functions on $U$,
let $B_U:=\Lambda_U\otimes_{\cO_K}K$  and we denote by $k_U:\Z_p^\times\lra \Lambda_U^\times$ the associated universal
character. We denote by $D_U$ the $B_U$-Banach module of $B_U$-valued  compact analytic distributions of $T_0$, homogeneous
of degree $k_U$ for the action of $\Z_p^\times$. Then $D_U$ is also a $\Gamma$-representation and we denote by $D^o_U$ the integral distributions, i.e. the ones with values in $\Lambda_U$. See section \S 3 for more details.

Of course if $k\in U(K)$ the two $\Gamma$-representations above are connected by a  $\Gamma$-equivariant specialization map
$D_U\rightarrow  D_k$.
We say that the classes in $\rH^1\bigl(\Gamma, D_k\bigr)$ are overconvergent modular symbols and the ones in
$\rH^1\bigl(\Gamma, D_U\bigr)$ are $p$-adic families of overconvergent modular symbols.

At this point we would like to point out that we have introduced a small modification to the usual way $p$-adic families of modular symbols are defined, namely we have used a wide open disk
instead of an affinoid as parameter space for the weights of our family. As a result the (integral) family of modular symbols  $\rH^1\bigl(\Gamma, D^o_U(1)\bigr)$ is a $\Lambda_U$-module. Without wishing to be pedantic, we'd like to stress that this small modification is essential for the following
interpretation: the integral distribution module $D^o_U$ has a natural filtration $\{{\rm Fil}^i(D^o_U)\}_{i\ge 0}$ which is $\Gamma$-invariant and which has  quotients which are
artinian $\cO_K$-modules  (see section \S 3 for more details).
This allows one to identify naturally the $p$-adic family of modular symbols over $U$ with
the continuous cohomology group
$$
\rH^1_{\rm cont}\Bigl(\Gamma, \bigl(D^o_U/{\rm Fil}^i(D^o_U)\bigr)_{i\ge 0}\otimes K\Bigr),
$$
which then can be identified by GAGA with the \'etale cohomology group on $Y_{\overline{\Q}}$
of the associated ind-continuous \'etale sheaf.

Due to this identification $\rH^1\bigl(\Gamma, D_U\bigr)$ has a natural $G_{\Q}$-action and at the same
time the completely continuous action of $U_p$ allows finite slope decompositions (to the expense of maybe shrinking $U$).

Let us say this again: {\em we only have Galois actions on \'etale cohomology groups and we only know how to perform  slope decompositions on arithmetic cohomology groups, therefore to have both
properties on our $p$-adic families of modular symbols we have to be able to make the above identification.} This is allowed by the small modification pointed out above, i.e. the use of the wide open disk $U$ and its ring of bounded rigid analytic functions $\Lambda_U$.

On the other hand for each $w\in \Q$  such that $0< w< p/(p+1)$ we denote by $X(w)$ the strict neighborhood of
the component containing the cusp $\infty$ of the ordinary locus of width $p^w$ in the rigid analytic curve $(X_{/K})^{\rm an}$
(see section \S2 for more details). For every $k\in \cW(K)$, in \cite{andreatta_iovita_stevens} we have shown that there exist a $w$ as above and an invertible, modular sheaf
$\omega^{\dagger,k}_w$ on $X(w)$ such that if $k\in \Z$ then  $\omega_w^{\dagger,k}\cong \omega^k|_{X(w)}$.
We call the elements of $\rH^0\bigl(X(w), \omega_w^{\dagger,k}\bigr)$ {\em overconvergent modular forms of weight} $k$ (and  radius of overconvergence $w$). In \cite{andreatta_iovita_stevens} it is shown that after taking the limit for $w\to 0$ we obtain precisely the Hecke module of overconvergent modular forms of weight $k$ introduced by Robert Coleman \cite{coleman}.
Similarly, if $U\subset \cW$ is a wide open disk and $k_U$ its universal weight, there is a $w$ and a modular sheaf
of $B_U$-Banach modules $\omega_w^{\dagger, k_U}$ such that the elements of $\rH^0\bigl(X(w), \omega_w^{\dagger,k_U}\bigr)$
are $p$-adic families of overconvergent modular forms over $U$.

\bigskip
\noindent
We can now display the main result of this article.

Fix $U\subset \cW^\ast$ a wide open disk defined over $K$ of so called {\em accessible weights}, namely of weights
$k$ such that $|k(t)^{p-1}-1|<p^{-1/(p-1}$. Let $w\in \Q$, $0<  w< p/(p+1)$ be such that $\omega_w^{\dagger, k_U+2}$ is defined over $X(w)$.
We define a {\bf geometric} $(B_U\ho\C_p)$-linear homomorphism
$$
\Psi_U\colon\rH^1\bigl(\Gamma, D_U\bigr)\ho_K\C_p(1)\lra \rH^0\bigl(X(w), \omega^{\dagger, k_U+2}_w\bigr)\ho_K\C_p,
$$
which is equivariant for the action of the Galois group $G_K={\rm Gal}(\Kbar/K)$ and the action of
the Hecke operators $T_\ell$, for $\ell$ not dividing $Np$ and $U_p$ and most importantly it is compatible with {\bf specializations}.
Let now $h\ge 0$ be an integer. We suppose that $U$ is such that both $\rH^1\bigl(\Gamma, D_U\bigr)$ and
$\rH^0\bigl(X(w), \omega^{\dagger, k_U}_w\bigr)$ have slope $\le h$ decompositions and that
there is an integer $k_0>h-1$ such that $k_0\in U(K)$.

If
$N$ is a $B_U[U_p]$-module which has a slope $\le h$-decomposition we denote by $N^{(h)}$
the slope $\le h$ submodule of $N$. All these being said, $\Psi_U$ induces a morphism on slope $\le h$ parts:
$$
\Psi_U^{(h)}\colon \rH^1\bigl(\Gamma, D_U\bigr)^{(h)}\ho_K\C_p(1)\lra \rH^0\bigl(X(w), \omega^{\dagger, k_U}_w\bigr)^{(h)}\ho_K\C_p.
$$

\noindent
We prove the following:

\begin{theorem}
\label{thm:muhintro}

There is a finite subset of weights $Z\subset U(\C_p)$ such that:

a) For each $k\in U(K)-Z$ there exists a finite dimensional $K$-vector space $S_{k+2}^{(h)}$ on which the Hecke operators
$T_{\ell}$ for $(\ell, Np)=1$ and $U_p$ act, $U_p$ acts with slope $\le h$ and such that we have natural, $G_K$ and
Hecke-equivariant isomorphisms
$$
\rH^1\bigl(\Gamma, D_k\bigr)^{(h)}\otimes_K\C_p(1)\cong \Bigl(\rH^0\bigl(X(w), \omega_w^{\dagger, k+2}\bigr)^{(h)}\otimes_K\C_p\Bigr)
\oplus \Bigl(S_{k+2}^{(h)}\otimes_K\C_p(k+2)\Bigr).
$$
Here the projection  of $\rH^1\bigl(\Gamma, D_k\bigr)^{(h)}\otimes_K\C_p(1)$ onto $\rH^0\bigl(X(w), \omega_w^{\dagger, k+2}\bigr)^{(h)}\otimes_K\C_p $ is
determined by the geometric morphism $\Psi_U^{(h)}$ above.

Moreover, the characteristic polynomial of $T_{\ell}$ acting on $S_{k+2}^{(h)}$ is equal to the characteristic
polynomial of $T_{\ell}$ acting on the space of
overconvergent cusp forms of weight $k+2$ and slope less or equal to $h$, $\rH^0\bigl(X(w), \omega_w^{\dagger, k}\otimes\Omega^1_{X(w)/K}\bigr)^{(h)}$.

b) We have a family version of a) above: for every wide open disk $V\subset U$ defined over $K$ such that
$V(\C_p)\cap Z=\phi$, there is a finite free $B_V$-module $S_V^{(h)}$ on which the Hecke operators $T_{\ell}, U_p$ act,
$U_p$ acts with slope less or eqaul to $h$, and we have a natural isomorphism $G_K$ and Hecke equivariant
$$
\rH^1\bigl(\Gamma, D_V^{(h)}\bigr)\ho_K\C_p(1)\cong \Bigl(\rH^0\bigl(X(w), \omega_w^{\dagger,k_V+2}\bigr)^{(h)}\ho_K\C_p\Bigr)\oplus
\Bigl(S_V\ho_K\C_p(\chi^{\rm univ}_V\cdot\chi)\Bigr),
$$
where $\chi^{\rm univ}_V:G_K\lra \Lambda_V^\times$ is the universal cyclotomic character of $V$. As at a), the first projection is determined by
the geometric map $\Psi_U^{(h)}$.

c) If $V$ is as at b) above let $k\in V(K)$ and let us denote by $t_k$ a uniformizer of $B_V$ at $k$.
Then we have natural isomorphisms as Hecke modules
$$
S_V^{(h)}/t_kS_V^{(h)}\cong S_{k+2}^{(h)},
$$
where $S_{k+2}^{(h)}$ is the Hecke module appearing at a).

\end{theorem}

The theorem above has as immediate consequence the following geometric interpretation of the global
Galois representations attached to  generic overconvergent cuspidal eigenforms of finite slope. Let $U,h,Z$ be as in theorem
\ref{thm:muhintro} and $k\in U(K)-Z$. Let $f$ be an overconvergent cuspidal eigenform of weight $k$ and slope $\le h$
and let us denote by $K_f$ the finite extension of $K$ generated by the $f$-eigenvalues of $T_\ell$, for all $\ell$
not dividing $Np$.

\begin{theorem}
The $G_\Q$-representation $\rH^1\bigl(\Gamma, D_k(1)\bigr)^{(h)}_f$ is a two dimensional $K_f$-vector space
and it is isomorphic to the $G_\Q$-representation attached to $f$ by the theory of pseudo-representations.
\end{theorem}

\bigskip
\noindent
The main difficulty in proving these theorems is the definition of the  geometric  map  $\Psi_U^{(h)}$
having all the required properties. We see it as a
map comparing a $p$-adic \'etale cohomology group, $\rH^1\bigl(\Gamma, D_U\bigr)(1)$ with a differential object,
namely $\rH^0\bigl(X(w), \omega_w^{\dagger, k_U+2}\bigr)$. We obtain it as a Hodge-Tate comparison map except that
on the one hand the \'etale cohomology group is global (on $X$) while the differential object only lives on the affinoid
$X(w)$. Moreover, to make things worse the \'etale sheaf associated to the $\Gamma$-representation
$D_U$ is {\bf not} a Hodge-Tate sheaf.

Let us explain the main new ideas in this article. We denote by $\fX(N,p)$ and $\fX(w)$ the logarithmic  Faltings' sites associated to
certain log formal models of $X$ and respectivel $X(w)$. There is a continuous functor $\nu: \fX(N,p)\lra \fX(w)$
which allows to move sheaves from one site to another.

Let $\cD_U$ denote the ind-continuous \'etale sheaf associated to $D_U$, it can be seen as sheaf on $\fX(N,p)$, then $\nu^\ast(\cD_U)$
is a sheaf on $\fX(w)$. At this point  something remarkable happens, namely using the Hodge-Tate sequence one can construct a natural
$\widehat{\cO}_{\fX(w)}[1/p]$-linear morphism of shaves on $\fX(w)$:
 $$
\delta^\vee_k(w)\colon \nu^\ast(\cD_U)\hat{\otimes}\widehat{\cO}_{\fX(w)}\lra  \omega_w^{\dagger, k_U+2}\ho\widehat{\cO}_{\fX(w)}.
$$

This fact allows us to define the map $\Psi_U^{(h)}$ as the composition
$$
\rH^1\bigl(\Gamma, D_U\bigr)\ho_K\C_p(1)\lra \rH^1\bigl(\fX(N,p), \cD_U\ho\widehat{\cO}_{\fX(N,p)}(1)\bigr)\lra
\rH^1\bigl(\fX(w), \nu^\ast(\cD_U)\ho\widehat{\cO}_{\fX(w)}(1)\bigr)\lra
$$
$$
\lra \rH^1\bigl(\fX(w), \omega_w^{\dagger, k_U}\ho\widehat{\cO}_{\fX(w)}(1)\bigr)\lra \rH^0\bigl(X(w), \omega^{\dagger, k_U+2}_w\bigr)\ho\C_p.
$$

 The theory developed in \cite{andreatta_iovita_stevens} plays a crucial role in the definition of the maps
 $\delta^\vee_k(w)$ . It was in fact the search for such  maps which lead us to discover  the modular sheaves $\omega_w^{\dagger, k}$.

Our finding was that  the theory of the canonical subgroup can be used in order to provide a new integral structure on the sheaf of invariant differentials of the universal generalized elliptic curve making the Hodge-Tate sequence exact integrally (namely without inverting $p$!). With this accomplished,  a definition \`a la Katz  provides the sought for sheaves $\omega_w^{\dagger, k}$ for any $k$ and the maps $\delta^\vee_k(w)$. We believe that this application to the problem of making the Hodge-Tate sequence integrally  exact constitutes the essence of the theory of the canonical subgroup.
 This intimate relation with the existence of the canonical subgroup  should justify the fact that $\delta^\vee_k(w)$ can be defined {\bf only} over $\fX(w)$ and not over the whole $\fX(N,p)$.

\

We believe that the ideas and techniques presented here could be applied without much change  in other settings in
order to prove overconvergent Eichler-Shimura isomorphisms (for example for Shimura curves, for Hilbert modular varieties etc.).
We realized that a very convenient concept to use in order to define
Hecke operators on Faltings' cohomology groups is that of localized (or induced) topos.
We made a careful study of various localized logarithmic Faltings' sites and showed that trace maps can be defined (see
section \S 2.4).
This allows us to
work in situations where we do not have explicit descriptions of good  integral models of the curves involved (for example $X_1(Np^r)$
for $r>1$) and would allow extensions of these results to higher dimensional Shimura varieties.

Let us finally remark that it would be possible to give geometric interpretations both of the Hecke modules $S_V^{(h)}$ and of the maps
$\rH^1\bigl(\Gamma, D_V\bigr)^{(h)}\ho\C_p(1)\lra S_V^{(h)}\ho\C_p(\chi_V^{\rm univ}\cdot\chi)$ appearing in theorem
\ref{thm:muhintro} and we propose to write these in a future article.

\bigskip
\noindent
{\bf Notations} In what follows we will denote by calligraphic letters $\cX,\cY,\cZ,\ldots$ log formal schemes over $\cO_K$ and by $\ucX$, $\ucY$, $\ucZ$
respectively the formal schemes underlying $\cX$ respectively $\cY$, respectively $\cZ$. We will denote by $X$, $Y$, $Z$, $\ldots$ respectively the log rigid
analytic generic fibers of $\cX$, $\cY$, $\cZ$, $\ldots$ and by $\uX$, $\uY$, $\uZ$, $\ldots$ respectively the underlying rigid spaces.

\section{Faltings' topoi}

\subsection{The geometric set-up.}
\label{sec:setup}

Let $p\ge 3$ be a prime integer, $K$ a complete discrete valuation field of characteristic $0$ and perfect residue field $\F$ of characteristic $p$ and
$N\ge 3$ a positive integer not divisible by $p$. We fix once for all
an algebraic closure $\Kbar$ of $K$ and an embedding $\overline{\Q}\hookrightarrow
\Kbar$, where $\overline{\Q}$ is the algebraic closure of $\Q$ in $\C$.
We denote by $\C_p$ the completion of $\Kbar$ and by $G_K$ the Galois group
of $\Kbar$ over $K$. We denote by $v$ the valuation on $\C_p$, normalized such that
$v(p)=1$.

Now we'd like to recall the basic geometric set-up from \cite{andreatta_iovita_stevens}.
Let $w\in \Q$ be such that $0\le w\le p/(p+1)$ and let us suppose that there is an
element (which will be denoted $p^v$) in $K$ whose valuation is $v:=w/(p-1)$.
We fix an integer $r\ge 1$ and we suppose that $w<2/(p^r-1)$ if $p>3$
 and $w<1/3^r$ if $p=3$.

We consider the following tower of rigid analytic modular curves over $K$
(in this section there are no log structures):
$$
X_1(Np^r)\lra X(N,p^r)\lra X_1(N),
$$
where $X_1(Np^r)$, respectively $X_1(N)$ classify generalized elliptic curves with $\Gamma_1(Np^r)$
respectively $\Gamma_1(N)$)-level
structure, while $X(N,p^r)$ classifies
generalized elliptic curves with $\Gamma_1(N)\cap\Gamma_0(p^r)$-level structure.
The morphism $X(N,p^r)\lra X_1(N)$ is the one which forgets the $\Gamma_0(p^r)$-level
structure.

We denote by ${\rm Ha}$ a lift of the Hasse invariant (for example ${\rm Ha}=E_{p-1}$, the normalized Eisenstein series of level $1$
 and weight $p-1$, if $p>3$)
which we view as a modular form on $X_1(N)$. We define the rigid analytic space
$$
X(w):=\{x\in X_1(N)\quad |\quad |{\rm Ha}(x)|\ge p^{-w}\}\subset X_1(N),
$$
and remark that the morphism $X(N,p^r)\lra X_1(N)$ has a canonical section
over $X(w)$ whose image we also denote by $X(w)$). We define
$X(p^r)(w):=X_1(Np^r)\times_{X_1(N,p^r)}X(w)$ and view $X(w)$ (respectively $X(p^r)(w)$) as a
connected affinoid subdomain of $X_1(N)$ and via the above mentioned section of $X(N,p^r)$ (respectively
of $X_1(Np^r)$).

We denote by $\cX_1(N)$, $\cX(N,p^r)$ and $\cX_1(Np^r)$ the $p$-adic formal schemes over $\cO_K$ obtained by
completing the proper schemes over $\cO_K$ classifying generalized elliptic curves with
$\Gamma_1(N)$, respectively $\Gamma_1(N)\cap\Gamma_0(p^r)$, respectively $\Gamma_1(Np^r)$-level
structures along, respectively, their special fibers.
Let $\cX(w)$ denote the open formal sub-scheme of the formal blow-up
of $\cX_1(N)$ defined by the ideal sheaf of $\cO_{\cX_1(N)}$ generated by
the sections $p^w$ and ${\rm Ha}(\cE/\cX_1(N),\omega)$ which is the complement
of the section at $\infty$ of the exceptional divisor of the blowing-up. Here
$\cE\lra \cX_1(N)$ is the universal generalized elliptic curve and $\omega$ is
a global invariant $1$-differential form of $\cE$ over $\cX_1(N)$.
Finally we let $\cX(p^r)(w)$ denote the normalization of $\cX(w)$
in $X_1(Np)(w)$ (see \S 3 of \cite{andreatta_iovita_stevens} for more details).

Let us remark that we have constructed a natural commutative diagram
of formal schemes, rigid analytic spaces and morphisms which is our basic geometric setup:
$$
\begin{array}{ccccccccc}
\cX(p^r)(w)&\lra&\cX(w)&=&\cX(w)\\
u\uparrow&&u\uparrow&&u\uparrow\\
X(p^r)(w)&\lra&X(w)&=&X(w)\\
\cap&&\cap&&\cap\\
X_1(Np^r)&\lra&X(N,p^r)&\lra&X_1(N)
\end{array}
$$
In the above diagram $u$ denotes the various specialization (or reduction) morphisms.

Finally, we have the following basic commutative diagram of formal schemes and rigid spaces:
$$
\begin{array}{cccccccccc}
&\cX(w)&\stackrel{\nu}{\lra}&\cX(N,p)\\
(\ast)&u\uparrow&&u\uparrow\\
&X(w)&\subset&X(N,p)
\end{array}
$$

\subsection{Log Structures}
\label{sec:logstr}

In this section we will describe log structures on the formal schemes and rigid spaces appearing in
the commutative diagram $(\ast)$ in section \S \ref{sec:setup}.

Let us now fix $N$, $r$ and $w$ as above and denote by $\ucX(w), \ucX(N,p)$ the formal schemes denoted $\cX(w), \cX(N,p)$
in section \ref{sec:setup}. We denote by $\pi$
a fixed uniformizer of $K$.
By its definition, if $\cU=\Spf(R_U)\hookrightarrow \ucX(w)$ is an affine open then
$\cU$ is either smooth over $\cO_K$ or if $\cU$ contains a supersingualr point,
then there is $a\in \N$, which depends only on $K$ and $w$,
and a formally \'etale morphism $\cU\lra \Spf(R')$, where
$R':=\cO_K\{X,Y\}/(XY-\pi^a)$.

Let us consider on $\uS:=\Spf(\cO_K)$ the log structure $M$ given by the closed point
and let us denote by $S:=(\uS,M)$ the associated log formal scheme. Let us recall that it
has a local chart given by $\N\lra \cO_K$ sening $1\rightarrow \pi$.

There exists a fine and saturated log structure $N_\cX$ on $\ucX(w)$, with a morphism of
log formal schemes $f: \cX(w):=(\ucX(w),N)\lra S=(\uS,M)$ which can be described locally as follows.
Let $\cU=\Spf(R)$ be an open affine of $\ucX$ as above, then

i) if $\cU$ is smooth over $S$ let us denote by
$\Sigma_\cU$ the divisor of cusps of $\cU$. Then $N|_{\cU}$ is the log structure
associated to the divisor $\Sigma_\cU$. It has a local chart of the form
$\N\lra R$ sending $1$ to a uniformizer at all the cusps.

ii) if $\cU$ is not smooth over $\uS$ we'll suppose that it does not contain cusps,
let $R'$ be as above and let us denote
$\Psi_R:R'\lra R_\cU$ the \'etale morphism of $\cO_K$-algebras defined above.

Let us consider the following commutative diagram of monoids and morphisms
of monoids (see \cite{andreatta_iovita3}, \S 2.1).
$$
\begin{array}{rrrrrr}
\N^2&\stackrel{\psi_R}{\lra}&R'\\
\Delta\uparrow&&\uparrow\\
\N&\stackrel{\psi_a}{\lra}&\cO_K
\end{array}
$$
where $\psi_R(m,n)=X^mY^n$, $\psi_a(n)=\pi^{an}$ and $\Delta(n)=(n,n)$
for all $n,m\in \N$.

Then the above diagram induces a natural isomorphism of $\cO_K$-algebras
$$
R'\cong \cO_K\{\N^2\}\otimes_{\cO_K\{\N\}}\cO_K.
$$
Let $P$ denote the amalgamated sum (or co-fibered product),
$\displaystyle P:=\N^2\oplus_{\N}\N$ associated to the diagram of monoids
$$
\begin{array}{ccccccccc}
&&\N\\
&&\uparrow\\
\N^2&\stackrel{\Delta}{\longleftarrow}&\N
\end{array}
$$
where the vertical morphism sends $n\rightarrow an$. By functoriality
we obtain a canonical morphism of monoids
$\displaystyle P\lra R'\stackrel{\Psi_R}{\lra} R_\cU$ which
defines a local chart of $\cX$, i.e. $N_{\cX}|_{\cU}$ is the log structure
associated to the pre log structure $P\lra R_\cU$. Let us consider the natural diagram of
monoids which defines $P$ as the amalgamated sum $\N^2\oplus_\N\N$

$$
\begin{array}{ccccccccc}
P&\stackrel{h}{\longleftarrow}&\N\\
\uparrow&&\uparrow\\
\N^2&\stackrel{\Delta}{\longleftarrow}&\N
\end{array}
$$

The morphism $h:\N\lra P$ defined by  the above diagram is a local chart of the morphism
$f:\cX\lra (S,M)$.

\begin{lemma}
\label{lemma:logsmooth}
The morphism $f:\cX(w)\lra S$ is log smooth.
\end{lemma}

\begin{proof}
Given the local description of $\cX(w)=(\ucX(w),N_\cX)$, $S=(\uS,M)$ and $f$ in terms of charts, it
is enough to consider
the case ii) above, i.e. we have a local chart $\displaystyle P\lra R'\stackrel{\Psi_R}{\lra}R_\cU$
and $\Psi_R$ is \'etale. By the description in \cite{kato1} \S 5 of log smooth morphisms,
it is enough the show that the morphism $h$ is injective and that the order of torsion of the group
$P^{\rm gp}/h(\N^{\rm gp})$ is invertible in $R_\cU$. For this it would be useful
to have an explicit description of $P$ as amalgamated sum  of monoids (see also
\cite{andreatta_iovita3} \S 2.1).

Let us define the sequence of monoids
$$ \frac{1}{a}\Delta(\N)+N^2\subset \frac{1}{a}\N^2\subset \Q^2
$$
as: $\displaystyle \frac{1}{a}\N^2$ is the (additive) submonoid of $\Q^2$ of
pairs of rational numbers $\displaystyle \bigl(\frac{n}{a}, \frac{m}{a}\bigr)$ with
$n,m\in \N$ and $\displaystyle
\frac{1}{a}\Delta(\N)+\N^2$ is its submonoid of pairs of rational numbers of the form
$\displaystyle \bigl(\frac{n}{a}+\alpha, \frac{n}{a}+\beta\bigr)$, where
$n,\alpha,\beta\in \N$.

We have natural morphisms of monoids
$\displaystyle \N^2\lra \frac{1}{a}\Delta(\N)+\N^2$ sending
$(\alpha,\beta)\rightarrow (\alpha,\beta)$ and $\displaystyle h':\N\lra   \frac{1}{a}\Delta(\N)+\N^2$
given by $\displaystyle n\rightarrow \bigl(\frac{n}{a}, \frac{n}{a}\bigr) $
such that the diagram is commutative

$$
\begin{array}{ccccccccc}
\frac{1}{a}\Delta+\N^2&\stackrel{h'}{\longleftarrow}&\N\\
\uparrow&&\uparrow\\
\N^2&\stackrel{\Delta}{\longleftarrow}&\N
\end{array}
$$

It is then easy to see that $\displaystyle P=\N^2\oplus_{\N}\N\cong \frac{1}{a}\Delta(\N)+\N^2$
by verifying that the latter monoid satisfies the universal properties of
the co-fibered product.
It follows that the above chart on $R'$ is explicitly given by $\displaystyle \frac{1}{a}\Delta(\N)+\N^2\lra R'$
where $\displaystyle \bigl(\frac{n}{a}+\alpha, \frac{n}{a}+\beta\bigr)\lra X^\alpha Y^\beta\pi^n$.
Of course one has to first verify that the association is well defined, which it is. One sees immediately that
the monoid $P$ is fine and saturated (as claimed at the beginning of this section) and moreover
that the morphism $h:\N\lra P$ is under the identifications between $P$ and
$\displaystyle \frac{1}{a}\Delta(\N)+\N^2$
equal $h'$, therefore it is injective and moreover it follows that
the quotient group $P^{\rm gp}/h(\Z)$ is torsion free.
This proves the lemma.
\end{proof}

\bigskip
\noindent
Recall that in the previous section we defined a morphism of formal schemes
$\cX(p^r)(w)\lra \cX(w)$.
We denote by $\ucX^{(r)}(w):=\cX(p^r)(w)$, we let also $N_r\lra \cO_{\ucX^{(r)}(w)}$ denote the
inverse image log structure via the above morphism and denote by
$\cX^{(r)}(w):=\bigl(\ucX^{(r)}(w), N_r\bigr)$ the associated log formal scheme. We also denote by
$\uX(w):=\cX(w)^{\rm rig}=X(w), \uX^{(r)}(w):=\cX(p^r)(w)^{\rm rig}=X(p^r)(w)$
the rigid analytic  generic fibers of the two formal schemes. We recall that $\uX^{(r)}(w)\lra \uX(w)$
is a finite, \'etale,
Galois  morphism with Galois group $G_r:=(\Z/p^r\Z)^\times$ and
as $\ucX^{(r)}(w)$ is the normalization of $\ucX(w)$ in $\uX^{(r)}(w)$,
we have that $G_r$ acts without fixed points on $\ucX^{(r)}(w)$ and $\ucX^{(r)}(w)/G_r\cong\ucX(w)$.
 In what follows we denote by
$X(w)$ and $X^{(r)}(w)$ the log rigid analytic generic fibers
of the log formal schemes $\cX(w)$ and respectively $\cX^{(r)}(w)$.
Their log structures are the horizontal ones defined by the divisors of cusps.

Moreover, as the formal scheme $\ucX(N,p)$ is semistable, its special fiber is a divisor with normal crossings.
We define the log structure on $\ucX(N,p)$ to be the one associated to the divisor consisting in the union of
the special fiber and the divisor of cusps and denote by $\cX(N,p)$ the corresponding log formal scheme.
Moreover we define on $\uX(N,p)$ the log structure associated to the cusps of this modular curve and by
$X(N,p)$ the corresponding log rigid space. Let us remark that the diagram $(\ast)$ of section
\S \ref{sec:setup}, written there for formal schemes and rigid spaces in fact  holds for log formal
scheme and log rigid spaces and it is commutative.

\begin{corollary}
\label{cor:normal}
The formal scheme $\ucX(w)$ is flat over $\Spf(\cO_K)$, it is Cohen-Macaulay and so
in particular normal. If $a=1$ then
$\ucX(w)$ is a regular formal scheme.
\end{corollary}

\begin{proof}
See \cite{andreatta_iovita3} \S 2.1.1 (3), where we show how to reduce to \cite{kato3}
Thm.~4.1.
\end{proof}

In particular, if $\cU=\Spf(R_\cU)\hookrightarrow \ucX$ is an affine open of $\ucX$,
then the $\cO_K$-algebra $R_\cU$ satisfies the assumptions (1), (2), (3) (FORM) and
(4) of section \S 2.1 of \cite{andreatta_iovita3}.

\subsection{Faltings' topoi}
\label{sec:topoi}

Our main reference for the constructions in this section is \cite{andreatta_iovita3}
section \S 1.2.

We will define Faltings' sites and topoi associated to the pairs of a log formal schemes and log rigid spaces: $\bigl(\cX(w), X(w)\bigr)$
and respectively $\bigl(\cX(N,p), X(N,p)\bigr)$ which will be denoted $\fX(w)$ and respectively $\fX(N,p)$.

We start by  writing $\bigl(\cX, X\bigr)$ for any one of the two pairs above. We'll define Faltings' site associated to this pair which we
denote  by $\fX$.  Namely we first let $\cX^{\rm ket}$
be the Kummer \'etale site of $\cX$, which is the full sub-category of the category of log schemes $\cT$, endowed with a Kummer log \'etale morphism $\cT\lra
\cX$ (see \cite{andreatta_iovita3} \S 1.2 or \cite{illusie2} section \S 2.1). We recall that the fiber product in this category is the fiber product of log
formal schemes in the category of fine and saturated log formal schemes so in particular the underlying formal scheme of the fiber product is not necessarily the
fiber product of the underlying formal schemes (see \cite{kato1}).

If $\cU$ is an object in $\cX^{\rm ket}$
then we denote by $\cU_{\Kbar}^{\rm fket}$ the finite Kummer \'etale site
attached to $\cU$ over $\Kbar$ as defined in \cite{andreatta_iovita3}
\S 1.2.2. An object in this site is a pair $(W, L)$
where $L$ is a finite extension of $K$ contained in $\Kbar$ and
$W$ is an object of the finite Kummer \'etale site of
$\cU_L$ which we denote by $\cU_L^{\rm fket}$. Given two objects
$(W', L')$ and $(W, L)$ of $\cU_{\Kbar}^{\rm fket}$, we define the morphisms
in the category as
$$
{\rm Hom}_{\cU_{\Kbar}^{\rm fket}}\bigl((W', L'), (W, L)\bigr):=
\lim_{\rightarrow}{\rm Hom}_{L''}\bigl(W'\times_{L'}L'',
W\otimes_LL''\bigr)
$$
where the limit is over the finite extensions $L''$ of $K$ contained in $\Kbar$
which contain both $L$ and $L'$.

Now, to define $\fX$ we denote by $E_{\cX_\Kbar}$ the category such that

i)  the objects are pairs $(\cU, W)$ such that $\cU\in \cX^{\rm ket}$ and
$W\in \cU_{\Kbar}^{\rm fket}$

ii) a morphism $(\cU',W')\lra (\cU,W)$ in $E_{\cX(w)_{\Kbar}}$ is a pair
$(\alpha, \beta)$, where $\alpha: \cU'\lra \cU$ is a morphism in $\cU^{\rm ket}$
and $\beta:W'\lra W\times_{\cU_K}\cU'_K$ is a morphism in $\bigl(\cU'\bigr)_{\Kbar}^{fket}$.

The pair $(\cX, X)$ is a final object in $E_{\cX_{\Kbar}}$ and moreover in this
category finite projective limits are representable and in particular fiber products exist
(see \cite{andreatta_iovita3} section \S 1.2.3 and \cite{erratum} proposition
2.6 for an explicit description of the fiber product).

\noindent
A family of morphisms $\{(\cU_i,W_i)\lra (\cU,W)\}_{i\in I}$
is a covering family if either

\noindent
($\alpha$) $\{\cU_i\lra \cU\}_{i\in I}$ is a covering family in $\cX^{\rm ket}$
and $W_i\cong W\times_{\cU_K}\cU{i,K}$ for every $i\in I$

or

\noindent
($\beta$) there exists $\cU$ in $\cX^{\rm ket}$ such that $\cU_i\cong \cU$ for all $i\in I$
and $\{W_i\lra W\}_{i\in I}$ is a covering in $\cU_{\Kbar}^{\rm fket}$.

We endow $E_{\cX_{\Kbar}}$ with the topology generated by the covering families
defined above and denote by $\fX$ the associated site.

Finally, the basic commutative diagram of log formal schemes and log rigid spaces
$$
\begin{array}{cccccccccc}
&\cX(w)&\stackrel{\nu}{\lra}&\cX(N,p)\\
(\ast)&u\uparrow&&u\uparrow\\
&X(w)&\subset&X(N,p)
\end{array}
$$
defines a natural functor $\nu:\fX(N,p)\lra \fX(w)$ defined, say on objects by:
$\nu(\cU,W):=\bigl(\cU\times_{\cX(N,p)}\cX(w), W\times_{X(N,p)}X(w)  \bigr)$.
This functor sends covering families to covering families and final objects to final objects therefore
it is a continuous functor of sites.

\begin{remark} Due to the mild singularities of the special fiber of the formal scheme
$\ucX(w)$, the site $\fX(w)$ and the sheaves on it were studied carefully and were well understood in \cite{andreatta_iovita3}. In the present paper however we
would need to study Faltings' site associated to $\cX^{(r)}(w)$, for various $r$'s. Unfortunately we do not understand well enough the geometry of $\ucX^{(r)}(w)$
to be able to work with this site directly, so instead we'll use a trick. Let us observe that $(\cX(w),X^{(r)}(w))$ is an object of $E_{\cX(w)_\Kbar}$, for all
$r\ge 1$ (while $(\cX^{(r)}(w),X^{(r)}(w))$ is not) so we define the induced (or localized) site $\fX^{(r)}(w):=\fX(w)_{/\bigl(\cX(w),X^{(r)}(w)\bigr)}$ and the
sheaves on it and this will be our substitute for Faltings' site attached to $\cX^{(r)}(w)$. Everything will be defined and explained in the next two sections.
\end{remark}

\subsection{Generalities on induced topoi}
\label{sec:induced}

In this section we'll recall some fundamental constructions and
results from \cite{SGA4}, Expos\'e IV, \S 5 (Topos induit), in the
restricted generality that we need.

Let $E$ denote a topos, namely the category of sheaves of sets on a site $S$, whose underlying category will be henceforth denoted $C$ and let $X$ be an object of
$C$. We denote by $C_{/X}$ the category of pairs $(Y,u)$ where $Y$ is an object of $C$ and $u\colon Y\lra X$ is a morphism in $C$. A morphism $(Y,u)\lra (Y',u')$ in
$C_{/X}$ is a morphism $\gamma:Y\lra Y'$ in $C$ such that $u'\circ \gamma=u$.

Let $\alpha_X\colon C_{/X}\lra C$ be the functor forgetting the morphism to $X$ i.e., for example, on objects it is defined by $\alpha_X(Y,u)=Y$. We endow the
category $C_{/X}$ with the topology induced from $C$ via $\alpha_X$ and denote the site thus obtained $S_{/X}$. We denote   by $E_{/X}$ the topos of sheaves on
$S_{/X}$ and call $S_{/X}$ and $E_{/X}$ the site and respectively the topos induced by $X$. We have natural functors $\alpha_X^\ast\colon E_{/X}\lra E$,
$\alpha_{X,\ast}\colon E\lra E_{/X}$ such that $\alpha_X^\ast$ is left adjoint to $\alpha_{X,\ast}$.

Suppose that $C$ has a final object $f$ and that it has fiber products. Then we have another functor $j_X\colon C\lra C_{/X}$ defined by $j_X(Z):=\bigl(X\times_fZ,
{\rm pr}_1\bigr)$, i.e. $j_X$ is the base change to $X$-functor. Then $j_X$ defines a continuous functor of sites $j_X\colon S\lra S_{/X}$ sending final object to
final object and so it defines a morphism of topoi $j_X^\ast\colon E\lra E_{/X}$ and $j_{X,\ast}\colon E_{/X}\lra E$. In particular,  $j_X^\ast$ is left adjoint to
$j_{X,\ast}$. Moreover by loc. cit. we have
$$
j_{X}^\ast(\cF)(Y,u)=\cF\bigl(Y)=\cF(\alpha_X(Y,u))=\alpha_{X,\ast}(\cF)(Y,u)
\mbox{ for every }\cF\in E, (Y,u)\in C_{/X}.
$$
Therefore we have a canonical isomorphism of functors $j_X^\ast\cong \alpha_{X,\ast}$ which implies that $j_X^\ast$ has a canonical left adjoint, namely
$\alpha_X^\ast$. This left adjoint of $j_X^\ast$ is denoted $j_{X,!}$ and we have an explicit description of it. Namely, for every $\cF\in E_{/X}$ we have
$j_{X,!}(\cF)=\alpha_X^\ast(\cF)$ is the sheaf associated to the presheaf on $C$ given by $\displaystyle Z\lra \lim_{\rightarrow}\cF(Y,u)$, where the limit is over
the category of triples  $(Y,u,v)$ where $(Y,u)$ is an object of $C_{/X}$ and $v\colon Z\lra Y$ is a morphism in $C$. As the limit is isomorphic to $\amalg_{g\in
Hom_C(Z,X)}\cF(Z,g)$  we conclude that $j_{X,!}(\cF)$ is the sheaf associated to the presheaf
$$
Z\lra \amalg_{g\in Hom_C(Z,X)}\cF(Z,g).
$$

\subsection{The site $\fX_{/(\cX, Z)}$}
\label{sec:fy}

Our main application of the theory in section \ref{sec:induced} is the following. Let us recall
that we denoted in section \ref{sec:topoi} by
$(\cX,X)$ any one of the two pairs $\bigl(\cX(w), X(w)\bigr)$ and $\bigl(\cX(N,p), X(N,p)\bigr)$ and
let $Z\lra X$ be a finite Kummer \'etale morphism of log rigid spaces, i.e. a morphism in
the category $\cX_\Kbar^{\rm fket}$.  Therefore the pair $(\cX, Z)$ is an object of $E_{\cX_\Kbar}$. We denote by $\bigl(E_{\cX_\Kbar}\bigr)_{/(\cX, Z)}$ the
induced category and by $\fZ:=\fX_{/(\cX, Z)}$ the associated induced site.

\bigskip
 As pointed out in section \ref{sec:induced} we have a functor
$\alpha:=\alpha_{(\cX,Z)}\colon \fZ\lra \fX$ and the associated adjoint functors $\alpha^\ast, \alpha_\ast$. As in the category  $E_{\cX_\Kbar}$ fiber products exist and
the category has a final element $(\cX, X)$, we also have a base change functor $j:=j_{(\cX, Z)}\colon\fX\lra \fZ$ defined by $j(\cU, W):=\bigl(\cU,
Z\times_{X}W, {\rm pr}_1\bigr)$. This functor commutes with fiber products, final elements and maps covering families to covering families so it it induces a
morphism of topoi $j^\ast\colon {\rm Sh}(\fX)\lra {\rm Sh}(\fZ)$ and $j_\ast \colon {\rm Sh}(\fZ)\lra {\rm Sh}(\fX)$ such that $j^\ast$ is left adjoint to
$j_\ast$. We further have a left adjoint $j_!\colon {\rm Sh}(\fZ)\lra {\rm Sh}(\fX)$ to $j^\ast$. More precisely, for every sheaf of abelian groups $\cF$ on
$\fZ$, the sheaf $j_!(\cF)$ on $\fX$ is the sheaf associated to the presheaf
$$
(\cU,W)\lra \oplus_{g\in Hom_{\cX(w)_\Kbar^{\rm fket}}(W,Z)}\cF(\cU,W,g).
$$

We have the following fundamental facts:

\begin{proposition}
\label{prop:left=right} For all $Z$ as above there is a natural isomorphism of functors $j_{(\cX, Z),!}\to j_{(\cX, Z),\ast}$.

\end{proposition}

\begin{proof} We divide the proof in two steps. First we construct a natural transformation of functors $j_{(\cX, Z),!}\to j_{(\cX, Z),\ast}$. then we prove
that it is an isomorphism.

\bigskip
\noindent
{\it Claim 1} Let $(\cU, W)$ be an object of $\fX$ and let $g\colon W\lra Z$ be a morphism in $\cX_\Kbar^{\rm fket}$. Then  we have a canonical isomorphism:
$Z\times_{X}W\cong W\amalg Z'_g$ for some object $Z'_g$ in $\cX_\Kbar^{\rm fket}$, such that this isomorphism composed with the morphism induced by $g$ is the
natural inclusion $W\hookrightarrow W\amalg Z'_g$.

\bigskip
\noindent Let us first point out that Claim 1 implies the existence of a canonical isomorphism
$$
(\ast) \ Z\times_{X}W\cong \bigl(\amalg_{g\colon W\rightarrow Z}W\bigr)\amalg Z'_W, \mbox{ for some object } Z'_W \mbox{ of  }\cX_\Kbar^{\rm fket}.
$$
Thus for every sheaf $\cF\in {\rm Sh}(\fZ)$ we have a canonical morphism
$$
j_!(\cF)(\cU,W)=\oplus_{g\colon W\rightarrow Z}\cF(\cU,W,g)\lra \cF(\cU\times_{X(w)}Z, {\rm pr}_1)=:j_\ast(\cF)(\cU,W).
$$

\bigskip
\noindent
Now let us prove Claim 1. We first prove the following

\begin{lemma}
\label{lemma:normal} Suppose $f\colon U\lra V$ is a finite Kummer log \'etale map of log affinoid spaces given by a chart of the form
$$
\begin{array}{rcccccc}
P&\lra&B\\
\uparrow&&\uparrow f\\
Q&\lra&A
\end{array}
$$
with $P,Q$ fine saturated monoids. We also suppose that $A,B$ are normal $K$-algebras, $A$ is an integral domain and the images of the elements of $P$ in $B$ are
not zero divisors. Then if $g\colon V\lra U$ is a morphism of log affinoids over $f$, there is an object $W$   and an isomorphism $U\cong V\amalg W$  in the
category $V^{\rm fket}$ such that the following diagram is commutative
$$
\begin{array}{ccccc}
U&\cong&V\amalg W\\
g\uparrow&&\iota\uparrow\\
V&=&V
\end{array}
$$
where $\iota\colon V\hookrightarrow V\amalg W$ is the natural map.
\end{lemma}

\begin{proof}
Let $s\in A$ be the product of the images in $A$ of a set of generators of $Q$ and let us remark that the image of $s$ in $B$ is not a zero divisor. Then $f_s\colon
A[1/s]\lra B[1/s]$ is a finite and \'etale morphism of $K$-algebras such that we have a morphism of $K$-algebras $g_s\colon B[1/s]\lra A[1/s]$ which is a section of
$f_s$. Then there is an $A[1/s]$-algebra $C'$, finite and \'etale and an isomorphism of $K$-algebras $B[1/s]\cong A[1/s]\times C'$ such that the following diagram
is commutative
$$
\begin{array}{ccccccc}
B[1/s]&\cong&A[1/s]\times C'\\
f_s\downarrow&&{\rm pr}_1\downarrow\\
A[1/s]&=&A[1/s]
\end{array}
$$
This is a well known fact but let us briefly recall the idea. As $B[1/s]$ is a finite $A[1/s]$-algebra, it is a finite projective and separable
$A[1/s]$-algebra. Then there is a unique element $e\in B[1/s]$ such that $g_s(x)={\rm Tr}_{B[1/s]/A[1/s]}(ex)$, for every $x\in B[1/s]$. Now it is not difficult to
prove that $e$ is an idempotent which gives a decomposition first  as $A[1/s]$-modules $B[1/s]\cong A[1/s]\times {\rm Ker}(g_s)$. Now one proves that $C':={\rm
Ker}(g_s)$ has a natural structure of $K$-algebra and the isomorphism is as $K$-algebras.

Let as above  $e$ and $1-e$ be  the indempotents which give the decomposition $B[1/s]\cong A[1/s]\times C'$.  Then $e$ satisfies $e^2-e=0$ i.e. $e$ is an element of
$B[1/s]$ integral over $A$ therefore $e\in B$. Therefore $e$, $1-e$ give an isomorphism as $K$-algebras $B\cong B'\times C$, where $A\subset B'\subset A[1/s]$. But
$B'$ is finite over $A$ therefore $B'=A$ as $A$ was supposed normal. Moreover $C$ is a finite $A$-algebra, which is an affinoid algebra as it is a quotient of $B$.

Now we endow $C$ with the prelog structure: $P\lra B\stackrel{\rm pr_2}{\lra} C$
and notice that the log rigid space $W:=\bigl({\rm Spm}(C), P^a\bigr)$ satisfies
$U\cong V\times W$ and makes the diagram of the lemma commutative.

Moreover as $U\lra V$ is a finite Kummer log \'etale map, therefore
$W\lra V$ is also Kummer log \'etale, as this can be read on stalks
of geometric points.

\end{proof}

\bigskip
\noindent Now let, as in Claim 1, $W\rightarrow X, Z\rightarrow X$ be morphisms in $\cX_\Kbar^{\rm fket}$, i.e. there is a finite extension $L$ of $K$ such that $W$ and
$Z$ are both defined over $L$ and we have finite Kummer log \'etale morphisms $W\lra X_L$ and $Z\lra X_L$. In fact it is enough to assume that
$\uZ$ and $\uW$ are both affinoids (if $\fX=\fX(w)$ this is always the case: as $\uX(w)$ is an affinoid and both maps
$W\rightarrow X(w)$ and $Z\rightarrow X(w)$ are finite it follows that $\uW$ and $\uZ$ are affinoids).
Moreover, as the morphism $X\lra {\rm Spm}(K)$ (with trivial log structure on ${\rm Spm}(K)$) is
log smooth, it follows that $Z\lra {\rm Spm}(K)$ and $W\lra {\rm Spm}(K)$ are both log smooth and so $\uX$, $\uZ$ and $\uW$ are all normal affinoids.

If $g\colon W\lra Z$ is a morphism over $X$, we have a natural morphism $g'\colon W\lra Z\times_{X}W$ which is a section of the projection
$Z\times_{X}W\lra W$. We apply the lemma \ref{lemma:normal} and we get Claim 1.

\bigskip
\noindent To conclude the proof of proposition \ref{prop:left=right} we make the following

\bigskip
\noindent {\it Claim 2} For every $(\cU,W)$ object of $\fX$, there is $W'\lra W$ a surjective morphism in $\cX_\Kbar^{\rm fket}$ with the property
$\displaystyle Z\times_{X}W'\cong \amalg_{g\colon W'\lra Z}W'$, i.e. in formula $(\ast)$ we have $Z'_{W'}=\phi$.

\bigskip

Clearly, Claim 2 implies that for a sheaf $\cF$ on $\fX$ the natural morphism $j_!(\cF)\lra j_\ast(\cF)$  is an isomorphism.
So we are left with the task of proving
Claim 2.

Here and elsewhere if $W$ is a log rigid space or a log formal scheme, we denote by $W^{\rm triv}$ the sub-space (or sub-formal scheme) on which the log structure
is trivial. Given our finite, Kummer \'etale  morphism $Z\lra X$ let ${\rm deg}_{Z/X}\colon Z^{\rm triv}\lra \Z$ denote the degree of $Z^{\rm triv}$ over
$X^{\rm triv}$. By restricting to a connected component of $Z$ we may suppose that ${\rm deg}_{Z/X}$ is constant equal to $n$. We prove Claim 2 by induction on
$n={\rm deg}_{Z/X}$.

If $n=0$  there is nothing to prove so let us suppose $n>0$. As the morphism $Z\lra X$ is a morphism of normal affinoids, if we regard $Z\times_{X}Z$ as a log
rigid space over $Z$ via the second projection then the diagonal $\Delta\colon Z\lra Z\times_{X}Z$ provides a section as in lemma \ref{lemma:normal}. Therefore
there is an object  $Z'$ in $\cX_\Kbar^{\rm fket}$ such that $Z\times_{X}Z\cong Z\amalg Z'$. It follows that ${\rm deg}_{Z'/Z}=n-1$ and applying the induction
hypothesis we find an object $W\lra Z$ in $Z^{\rm fket}$ such that $Z'\times_ZW=\amalg_{i=1}^m W_i$, where  $W_i= W$ for all $1\le i\le m$. Then the composition
$W\lra Z'\lra X$ makes $W$ an object in $\cX_\Kbar^{\rm fket}$ and we have $Z\times_{X}W$ is isomorphic to a disjoint union of objects isomorphic to $W$.

\end{proof}

Proposition \ref{prop:left=right} has the following immediate consequence.

\begin{corollary}
\label{cor:jacyclic} Suppose $Z\lra X$ is a morphism in $\cX_\Kbar^{\rm fket}$. Then we have\smallskip

a) The functor $j_{\ast}$ is an exact functor.\smallskip

b) $R^ij_\ast=0$ for all $i\ge 1$.\smallskip

\end{corollary}

\begin{proof} For a) we remark that $j_\ast\cong j_!$ by proposition \ref{prop:left=right} and $j_!\cong \alpha^\ast$.
It follows that $j_!$, and so also $j_\ast$, is right exact. As $j_\ast$ is left exact, it is exact.  This immediately implies b).
\end{proof}

As $j^\ast$ admits a left adjoint $j_!$ by adjunction we get a morphism
\begin{equation}\label{def:trace} S_Z\colon j_\ast\bigl( j^\ast(\cF)\big) \cong  j_!\bigl( j^\ast(\cF)\bigr) \longrightarrow \cF,
\end{equation}  functorial on the category of sheaves of abelian groups on the site $E_{\cX_\Kbar}$. We call it the {\it trace map} relative to $Z$. More
explicitly, given a sheaf of abelian groups  $\cF$ on $E_{\cX_\Kbar}$ it is the map of sheaves associated to the map of presheaves:
$$j_!\bigl(j^\ast(\cF)\bigr)(\cU,W)=  \oplus_{g\in Hom_{\cX_\Kbar^{\rm fket}}(W,Z)}j^\ast(\cF)(\cU,W,g)=\oplus_{g\in Hom_{\cX_\Kbar^{\rm fket}}(W,Z)}\cF(\cU,W)
\longrightarrow \cF(\cU,W),$$given by the sum.

\subsection{Sheaves on $\fX$}
\label{sec:hodgetate}

We continue, as in the previous section, to denote by $\fX$ any one of the sites $\fX(w)$ or $\fX(N,p)$ and
we will describe certain sheaves on this site.

We denote by $\cO_{\fX}$ the presheaf of $\cO_\Kbar$-algebras on $\fX$ defined by
$$
\cO_{\fX}(\cU, W):=\mbox{ the normalization of }\rH^0(\ucU, \cO_{\ucU})\mbox{ in } \rH^0(\uW, \cO_{\uW}).
$$

We also define by $\cO_{\fX}^{\rm un}$ the sub-presheaf  of $\WW(k)$-algebras of $\cO_{\fX}$ whose sections over $(\cU, W)$ consist of the elements $x\in
\cO_{\fX}(\cU, W)$ such that there exist a finite unramified extension $M$ of $K$ contained in $\Kbar$, a Kummer log \'etale morphism $\cV\lra
\cU\otimes_{\cO_K}\cO_M$ and a morphism $W\lra \cV_K$ over $\cU_K$ such that $x$, viewed as an element of $\rH^0(\uW, \cO_{\uW})$ lies in the image of $\rH^0(\ucV,
\cO_{\ucV})$. Then we have

\begin{proposition}[\cite{andreatta_iovita3}, Proposition 1.10]
The presheaves $\cO_{\fX}$ and $\cO_{\fX}^{\rm un}$ are sheaves and
$\cO_{\fX}^{\rm un}$ is isomorphic to the sheaf $v_{\cX}^\ast\bigl(\cO_{\cX^{\rm ket}}\bigr)$
\end{proposition}

We denote by $\hO_{\fX(w)}$ and $\hO_{\fX}^{\rm un}$ the continuous sheaves on $\fX$ given by the
projective systems of sheaves $\{\cO_{\fX}/p^n\cO_{\fX}\}_{n\ge 0}$ and respectively
$\{\cO_{\fX}^{\rm un}/p^n\cO_{\fX}^{\rm un}\}_{n\ge 0}$.

\bigskip
In the notations of section \ref{sec:fy} let $r\ge 1$ and $w\in \Q$ adapted to $r$ and let $Z:=X^{(r)}(w)\lra X(w)$
if $\fX$ denotes $\fX(w)$ or $Z:=X^{(r)}\lra X(N,p)$ if $\fX$ denotes $\fX(N,p)$ and let us denote by
$\fX^{(r)}:=\fX_{/(\cX, Z)}$ the site induced by $(\cX, Z)$ and  by $j_r^\ast$, $j_{r,\ast}(\cong j_{r,!})$ the associated morphism of
topoi. We have functors $$v_{\fX}\colon \cX^{\rm ket}\lra \fX, \qquad v_{r}\colon\cX^{\rm ket}\lra \fX^{(r)}$$defined by $v_{\fX}(\cU):=(\cU,
\cU_K)$ and $v_r:=j_r\circ v_{\fX}$. More explicitly,  $v_r(\cU):=j(\cU, \cU_K)=\bigl(\cU, Z\times_{X}\cU_K, {\rm pr}_1\bigr)$. These functors send
covering families to covering families, commute with fiber products and send final objects to final objects. In particular they define morphisms of topoi. Corollary
\ref{cor:jacyclic} implies that the Leray spectral sequence for $v_{r,\ast}=v_{\fX,\ast}\circ j_{r,\ast}$ degenerates and we have $R^iv_{r,\ast}\cong
R^iv_{\fX,\ast}\circ j_{r,\ast}.$

We denote by $\cO_{\fX^{(r)}}:=j_r^\ast(\cO_{\fX})$ and by $\hO_{\fX^{(r)}}:=j_r^\ast(\hO_{\fX})$. Let us recall the morphism $\theta_r\colon
\ucX^{(r)}\lra \ucX$ which is finite and defines $\ucX^{(r)}$ as the normalization of $\ucX$ in $X^{(r)}$ and let $G_r\cong (\Z/p^r\Z)^\times$ denote
the Galois group of $X^{(r)}/X$. Then $G_r$ acts naturally on $\ucX^{(r)}$ over $\ucX$ and $\ucX\cong \ucX^{(r)}/G_r$.

\begin{lemma}
\label{lemma:vyox} We have a natural isomorphism of sheaves on $\cX^{\rm ket}$
$$\bigl(v_{r,\ast}(\cO_{\fX^{(r)}})\bigr)^{G_r}\cong
\cO_{\ucX}\mbox{ and similarly } \bigl(v_{r,\ast}(\hO_{\fX^{(r)}})\bigr)^{G_r}\cong \hO_{\ucX}.
$$

\end{lemma}

\begin{proof}
Let $\cU\lra \cX$ be a morphism in $\cX^{\rm ket}$. Then we have
$$
v_{r,\ast}(\cO_{\fX^{(r)}})(\cU)=\cO_{\fX}(\cU, X^{(r)}\times_{X}\cU_K)=\rH^0(\cX^{(r)}\times_{\cX}\cU,
\cO_{\ucX^{(r)}})=\theta_{r,\ast}(\cO_{\ucX^{(r)}})(\cU).
$$
Form this the claim follows.

\end{proof}

\subsection{The localization functors}
\label{sec:localization}

As in the previous section, $\fX$ denotes any one of the sites $\fX(w)$ or $\fX(N,p)$.
We recall here the localization of a sheaf or a continuous sheaf on $\fX$ to a ``small affine of $\cX^{\rm ket}$'' (for more details see
\cite{andreatta_iovita3} section \S 1.2.6). Let $\cU=\bigl(\Spf(R_\cU), N_\cU\bigr)$ be a connected small affine object of $\cX^{\rm ket}$ and we denote by
$U:=\cU_K$ the log rigid analytic generic fiber of $\cU$. Let us recall that under the above hypothesis $\cU$ is a log formal scheme whose log structure is given by
the sheaf of monoids denoted $N_\cU$.

We write $R_\cU\otimes \Kbar=\prod_{i=1}^n R_{\cU,i}$ with $\Spf(R_{\cU,i})$ connected, we let $N_{\cU,i}$ denote the monoids which give the respective log
structures and we let $U_i$ denote the respective log rigid analytic generic fiber. Then each $R_{\cU,i}$ is an integral domain, so we let $\C_{\cU,i}$ denote an
algebraic closure of the fraction field of $R_{\cU,i}$ for all $1\le i\le n$ and let $\C_{\cU,i}^{\rm log}:=\Spec\bigl(\C_{\cU,i}), N_{\cU,i}\bigr)$ denote the log
geometric point of $\cU_i:=\bigl(\Spf(R_{\cU,i}), N_{\cU,i}) \bigr)$ over $\C_{\cU,i}$ (see \cite{illusie2} definition 4.1 or \cite{andreatta_iovita3} section 1.2.5
for the definition of a log geometric point). We denote by $\cG_{U,i}:= \pi_1^{\rm log}\bigl(U_i, \C_{\cU,i}^{\rm log}\bigr)$ the Kummer \'etale fundamental group
of $U_i$. We have then that the category $U_i^{\rm fket}$ is equivalent to the category of finite sets with continuous $\cG_{U,i}$-action. We write $(\Rbar_{\cU,i},
\Nbar_{\cU,i})$ for the direct limit over all finite normal extensions $R_{\cU,i}\subset S\subset \C_{\cU,i}$, all log structures $N_S$ on $\Spm(S_K)$ such that
there are Kummer \'etale morphisms $\C_{\cU,i}\lra \bigl(\Spm(S_K), N_S\bigr)\lra U_i$ compatible with the one between the underlying formal schemes. Finally we
denote $\Rbar_\cU:=\prod_{i=1}^n\Rbar_{\cU,i}$, $\Nbar_\cU:=\prod_{i=1}^n \Nbar_{\cU,i}$ and $\cG_{U_\Kbar}:=\prod_{i=1}^n\cG_{U,i}$.

We denote by ${\rm Rep}\bigl(\cG_{U_\Kbar}\bigr)$ and  ${\rm Rep}\bigl(\cG_{U_\Kbar}\bigr)^\N$ the category
of discrete abelian groups with continuous action by $\cG_{U_\Kbar}$, respectively the category
of projective systems of such. It follows from \cite{illusie2} section \S 4.5 that
we have an equivalence of
categories
$$\mbox{Sh}(U_\Kbar^{\rm fket})\cong {\rm Rep}\bigl(\cG_{U_\Kbar}\bigr)$$
sending $\displaystyle \cF\rightarrow \lim_{\rightarrow}\cF(\Spm(S_K), N_S)$.
Therefore composing with the restriction ${\rm Sh}(\fX)\lra {\rm Sh}(U_\Kbar^{\rm fket})$
defined by $\cF\rightarrow \bigl(W\rightarrow \cF(\cU, W)\bigr)$, we obtain a functor,
called localization functor
$$
{\rm Sh}(\fX)\lra {\rm Rep}\bigl(\cG_{U_\Kbar}\bigr) \mbox{ denoted } \cF \rightarrow \cF(\Rbar_{\cU}, \Nbar_{\cU}).
$$

We consider the following variant. Let $Z\to X$, with $X=X(w)$ or $X=X(N,p)$, be a finite Kummer \'etale morphism in $X_\Kbar^{\rm fket}$. Consider the associated site $\fZ:=\fX_{/(\cX, Z)}$ as in section \S\ref{sec:fy} and let $j_{(\cX, Z)}\colon \fX\to \fZ$ be the induced morphism of sites. Consider a sheaf $\cF\in {\rm Sh}(\fZ)$ and fix  a connected small affine object $\cU=\bigl(\Spf(R_\cU), N_\cU\bigr)$ of $\cX^{\rm ket}$  as before. Denote by $\Upsilon_\cU$ the set of homomorphisms of $R_\cU\otimes \Kbar$-algebras $\Gamma\bigl(Z\times_X U,\cO_{Z\times_X U}\bigr)\to \Rbar_{\cU}[1/p]$. For any $g\in \Upsilon_\cU$  we write $\cF\bigl(\Rbar_{\cU}, \Nbar_{\cU},g\bigr):=\lim \cF\bigl(\cU,W\bigr)$, where the limit is taken over all finite and Kummer \'etale maps $\Spm(S_K)=W\to Z\times_X U$  with $S_K\subset  \Rbar_{\cU}[1/p]$ a $\Gamma\bigl(Z\times_X U,\cO_{Z\times_X U}\bigr)$-subalgebra (using $g$). Let $\cG_{U_\Kbar,Z,g}$ be the subgroup of $\cG_{U_\Kbar} $ fixing $\Gamma\bigl(Z\times_X U,\cO_{Z\times_X U}\bigr)$.  Then ${\rm Sh}\bigl((Z\times_X U)^{\rm fket}\bigr)\cong {\rm Rep}\bigl(\cG_{\cU_\Kbar,Z,g}\bigr)$ and we obtain as before a localization functor:

$$
{\rm Sh}(\fZ)\lra {\rm Rep}\bigl(\cG_{\cU_\Kbar,Z,g}\bigr),\qquad \cF \rightarrow \cF\bigl(\Rbar_{\cU}, \Nbar_{\cU},g\bigr).
$$

If $\{\cU_i\}_i$ is a covering of $\cX^{\rm ket}$ and for every $i$ we choose $g_i\in \Upsilon_{\cU_i}$, it follows from the definition of coverings in the site $\fZ$
that the map ${\rm Sh}(\fZ)\lra \prod_i {\rm Rep}\bigl(\cG_{U_{i,\Kbar},Z,g_i}\bigr) $ is faithful. It also follows from proposition \ref{prop:left=right} that
\begin{equation}\label{localizationjast} j_{(\cX, Z),\ast}(\cF)\bigl(\Rbar_{\cU}, \Nbar_{\cU}\bigr)\cong \oplus_{g\in \Upsilon_\cU}\cF\bigl(\Rbar_{\cU}, \Nbar_{\cU},g\bigr).\end{equation}

\

\section{Analytic Modular Symbols}\label{sec:dist}


\subsection{Analytic functions and distributions}
\label{sec:locallyanalytic}

In this section we recall a number of definitions and results from \cite{ash_stevens} and \cite{harris_iovita_stevens} and also
define some new objects.
Let $T_0:=\Z_p^\times\times\Z_p$, which we regard as a compact open subset
of the space of row vectors $(\Z_p)^2$.  We have the following structures on $T_0$:
\begin{itemize}
\item[a)] a natural left action of $\Z_p^\times$ by scalar multiplication;
\item[b)] a natural right action of the semigroup
$$
\Xi(\Z_p) = \left\{\left( \begin{array}{cc}{a} & {b} \\ {c} &
{d}\end{array}\right)\in M_2(\Z_p)\, \biggm|\,(a,c) \in \Z_p^\times\times p\Z_p\,\right\}
$$
and its subgroup
$$
\Iw(\Z_p) :=  \Xi(\Z_p)\cap \GL(\Z_p)
$$
given by matrix multiplication on the right.
\end{itemize}
The two actions obviously commute.

Let us recall that we denoted by $\cW^\ast$ the rigid subspace of $\cW$ of weights
$k$ such that $|k(t)^{p-1}-1|<p^{-1/(p-1}$. Let $U\subset \cW^\ast$ denote a
wide open disk which is an admissible open rigid subspace of $\cW_1$, let $A_U$ denote the $O_K$-algebra
of rigid functions on $U$  and denote by
$$
\Lambda_U=\{f \mbox{ rigid function on } U \mbox{ such that } |f(x)|\le 1 \mbox{ for every point } x\in U\},
$$
the $\cO_K$-algebra of bounded rigid functions on $U$. Then, as remarked in section \ref{sec:modularsheaves},
$\Lambda_U$ is a complete, regular, local, noetherian $\cO_K$-algebra, in fact
$\Lambda_U$ is (non canonically) isomorphic to the $\cO_K$-algebra $\cO_K[\![T]\!]$. The completeness refers to the
$\underline{m}_U$-adic topology (called the {\bf weak topology} of $\Lambda_U$), for $\underline{m}_U$ the maximal ideal of $\Lambda_U$.
As remarked in
section \ref{sec:modularsheaves}, $\Lambda_U$ is also complete for the $p$-adic topology.

Let now $B$ denote one of the complete, regular, local, noetherian rings: $\cO_K$ or $\Lambda_U$, for
$U\subset \cW^\ast$ a wide open disk as above. Let also $k\in \cW^\ast(B_K)$ be as follows: if
$B=\cO_K$, $k\in \cW^\ast(K)$ and if $B=\Lambda_U$ we set $k=k_U:\Z_p^\times\lra \Lambda_U^\times$
defined by $t^{k_U}(x)=t^x$ for all $x\in U(K)$.

\begin{definition}
\label{def:A_k}
We set
 $$
 A_{k}^o := \Bigl\{\,f: T_0\lra B\,\biggm|\, i) \  \forall a\in \Z_p^\times,\,t\in T_0,\ \mbox{ we have }  f(at) = k(a) f(t)
$$
and
$$
ii) \mbox{ the function } z\rightarrow f(1,z) \mbox{ extends to a  rigid analytic function
on the closed unit disk } B[0,1] \Bigr\}.
$$
\end{definition}

Let us make ii) of the above definition more precise. Let ${\rm ord}_p:B-\{0\} \lra \Z$ be defined by
${\rm ord}_{\pi}(\alpha)=\sup\{n\in \Z\quad | \quad\alpha\in \pi^nB\}$. Then we say that
the function $z\rightarrow f(1,z)$ in the definition above is rigid analytic
on $B[0,1]$ if there exists a power series $F(X)=\sum_{n=0}^\infty a_nX^n\in B[\![X]\!]$
with $\displaystyle {\rm ord}_{\pi}(a_n)\stackrel{n\rightarrow\infty}{\lra}\infty$ and such that
$f(1,z)=F(z)$ for all $z\in \Z_p$. Let us denote by $A_k:=A_k^o\otimes_{\cO_K}K$, which is naturally a
$B_K$-module. As $B_K$ si a $K$-Banach space (for its
$p$-adic toplogy) let us point out that $A_k$ is an orthonormalizable  Banach $B_K$-module, where
an orthonormal basis is given by: $\{f_n\}_{n\ge 0}$ where $f_n\in A_k^o$ are the unique elements such that
$f_n(1,z)=z^n$ for all $z\in \Z_p$. In other words $f_n(x,y)=k(x)(y/x)^n$ for all $(x,y)\in T_0$.

For every $\gamma\in \Xi(\Z_p)$ and function $f:T_0\lra  B$ we define $(\gamma f)(v):=f(v\gamma)$. We have

\begin{lemma}
\label{lemma:Akaction}
 If $f\in A_k^o$ and $\gamma\in \Xi(\Z_p)$ then $\gamma f\in A_k^o$.
\end{lemma}

\begin{proof}
Let $\gamma=\left( \begin{array}{cc}{a} & {b} \\ {c} &{d}\end{array}\right)$. Then for every
$v\in T_0$ and $a\in \Z_p^\times$  we have $(\gamma f)(av)=f\bigl((av)\gamma\bigr)=k(a)f(v\gamma)
k(a)(\gamma f)(v).$

Moreover,
$$
(\gamma f)(1,z)=f(a+cz, b+dz)=k(a+cz)f\Bigl(1,\frac{b+dz}{a+cz}\Bigr)=k(a)k(1+ca^{-1}x)\sum_{n=0}^\infty a_n\Bigl(\frac{b+dz}{a+cz}  \Bigr)^n.
$$
Using the fact that $k$ is analytic and $\displaystyle a_n\stackrel{n\rightarrow \infty}{\lra}0$ we deduce that the function
$z\rightarrow (\gamma f)(1,z)$ extends to an analytic function on the closed unit disk $B[0,1]$.
\end{proof}

\begin{definition}
\label{def:dk}
a) Let $k\in \cW^\ast(K)$ be a weight. We define
 $D_k^o:={\rm Hom}_{\rm cont,\cO_K}\bigl(A_k^o, \cO_K\bigr)$, i.e. the $\cO_K$-module of continuous,
$\cO_K$-linear homomorphisms from $A_k^o$ to $\cO_K$. We also denote by $D_k:=D_k^o\otimes_{\cO_K}K$.

b) If $U\subset \cW^\ast$ is a wide open disk defined over $K$, we define
$D_U^o:={\rm Hom}_{\Lambda_U}\bigl(A_{U}^o, \Lambda_U\bigr)$, i.e. the $\Lambda_U$-module of continuous
for the $\underline{m}_U$-adic topology $\lambda_U$-linear homomorphisms from $A_{U}^o:=A_{k_U}^o$ to $\Lambda_U$. We denote by
$D_{U}:=D_{U}^o\otimes_{\cO_K}K.$
\end{definition}

\begin{remark}
\label{rmk:topologies}
i) The (left) action of the semigroup $\Xi(\Z_p)$ on $A_k^o$ induces a (right) action on $D^o$
by $(\mu|\gamma)(f):=\mu(\gamma f)$ for all $\gamma\in \Xi(\Z_p)$, $f\in A_k^o$ and
$\mu\in D^o$.

ii) We have a natural, fundamental homomorphism of $B$-modules
$$
\psi\colon D^o\lra \prod_{n\in \N}B, \mbox{ defined by } \mu\rightarrow \bigl(\mu(f_n)\bigr)_{n\in \N}.
$$
As the family $\bigl(f_n\bigr)_n$ is an orthonormal basis of $A_k$ over $B_K$, the above
morphism is a $B$-linear isomorphism. Moreover, under this isomorphism, the weak topology on $D^o$
corresponds to the weak topology (i.e. the product of the $\underline{m}_B$-adic
topologies on the product).

iii) A more common definition in the literature (see \cite{ash_stevens}) would be:
$$
\widetilde{D}_{U}^o={\rm Hom}^{\rm cpt}_{p, \Lambda_U}(A_{U}^o, \Lambda_U)\subset D_{U}^o,
$$
i.e. $\widetilde{D}_{U}^o$ consists of the  continuous and compact (or completely continuous) in the $p$-adic topology,
$\Lambda_U$-linear homomorphisms.
Then $\widetilde{D}_{U}:=\widetilde{D}_{U}^o\otimes_{\cO_K}K$ is an orhtonormalizable
$\Lambda_{U,K}$-module.

\end{remark}

Our $D_U^o$ has a property very similar to being orthonormalizable, which will be sufficient to define Fredholm
characteristic series of compact $\Lambda_U$-linear operators on it. More precisely we have the following.

\begin{lemma}
\label{lemma:weakON}
Let $\bigl( e_i \bigr)_{i\in I}$ be a sequence of elements in $D_U^o$ such that their image
$\bigl( \overline{e}_i \bigr)_{i\in I}$ in $D_U^o/\underline{m}_UD_U^o$ is a basis of this vector space over $\F:=\Lambda_U/\underline{m}_U\cong \cO_K/\pi\cO_K$. Then for every $n\in \N$ the natural map
$$
\Psi_n:\oplus_{i\in I}\Bigl(\Lambda_U/\underline{m}_U^n \Bigr)e_i\lra D_U^o/\underline{m}_U^nD_U^o
$$
is an isomorphism of $\Lambda_U$-modules.
\end{lemma}

\begin{proof}
Let us first recall that we have a natural isomorphism of $\Lambda_U$-modules $D_U^o\cong \prod_{i\in \N}\Lambda_U$.
In particular it follows that for every $n\ge 0$
$$
\underline{m}^nD_U^o/\underline{m}_U^{n+1}D_U^o\cong \prod_{i\in \N}\underline{m}_U^n/\underline{m}_U^{n+1}\cong \underline{m}_U^n/\underline{m}_U^{n+1}\otimes_{\F}
D_U^o/\underline{m}_UD_U^o\cong \oplus_{i\in I}\bigl(\underline{m}_U^n/\underline{m}_U^{n+1} \bigr)e_i.
$$
The second isomorphism follows from the fact that $\underline{m}_U^n/\underline{m}_U^{n+1}$ is a finite dimensional $\F$-vector space.

Now we prove  the lemma by induction on $n$. The case $n=1$ is clear, therefore let us suppose that the property is true for
$n\ge 1$ and we'll prove that $\Psi_{n+1}$ is an isomorphism. We have the following commutative diagram withexact rows:
$$
\begin{array}{ccccccccccc}
0&\lra&\underline{m}_U^nD_U^o/\underline{m}_U^{n+1}D_U^o&\lra&D_U^o/\underline{m}_U^{n+1}D_U^o&
\lra&D_U^o/\underline{m}_U^n&\lra&0\\
&&\uparrow \varphi&&\uparrow\Psi_{n+1}&&\uparrow\Psi_n\\
0&\lra&\oplus_{i\in I}\bigl(\underline{m}_U^n/\underline{m}_U^{n+1} \bigr)e_i&\lra&\oplus_{i\in I}\bigl(\Lambda_U/\underline{m}_U^{n+1}\bigr)e_i&\lra&\oplus_{i\in I}\bigl(\Lambda_U/\underline{m}_U^n\bigr)e_i&\lra&0
\end{array}
$$
By inductive hypothesis $\Psi_n$ is an isomorphism and by the comment before the diagram, $\varphi$ is an isomorphism as well.
Therefore $\Psi_{n+1}$ is an isomorphism.

\end{proof}

\begin{corollary}
\label{cor:ONbasis}
Let us fix  a family $\bigl(e_i\bigr)_{i\in I}$ as in lemma \ref{lemma:weakON}. Then for every $x\in D_U^o$ there is a unique
sequence $a_i\in \Lambda_U$, $i\in I$ such that

i) $a_i\rightarrow 0$ in the filter of complements of finite sets in $I$, in the $\underline{m}_U$-topology, i.e.  for every $h\in \N$ the subset $i\in I$ with the property
$a_i\not\in \underline{m}_U^h$ is finite.

ii) $x=\sum_{i\in I}a_ie_i$.

\end{corollary}

\begin{proof}
The corollary follows immediately from lemma \ref{lemma:weakON}.
\end{proof}

\begin{remark}
\label{rmk:ONbase}
1) A family of elements $\bigl(e_i  \bigr)_{i\in I}$ as in corollary \ref{cor:ONbasis} plays the role of an orthonormal basis
of a Banach $B_U$-module. In particular it can be used to define the Fredholm series of a compact $B_U$-linear operator on
${\rm H}^1\bigl(\Gamma, D_U\bigr)$.

2) If $k\in U(K)$ is a weight, the image of a family of elements $\bigl(e_i\bigr)_{i\in I}$ of $D_U^o$ as at 1) above in
$D_k^o$ is a true ON basis of $D_k$ over $K$.

\end{remark}

Keeping the notations above, let $V\subset U$ be an affinoid disk with affinoid algebra $B_V$ and
let $B_V^o$ denote the bounded by $1$ rigid functions on $V$. We have

\begin{lemma}
\label{lemma:restrizione}
$D_U|_V:=D_U\hat{\otimes}_{B_U}B_V$ is an ON-able $B_V$-module with orthorormal basis $(e_i\hat{\otimes}1)_{i\in I}$ and in fact coincides with the $B_V$-module of completely continuous
$B_V$-linear maps from $A_V$ to $B_V$.
\end{lemma}

\begin{proof}
Let us first make precise the completed tensor product in the statement of the lemma.
We have $D_U|_V=\bigl(D_U^o|_V\bigr)\otimes_{\cO_K}K$ where we define
$$
D_U^o|_V:=\lim_{\infty\leftarrow n} \Bigl(D_U^o\otimes_{\Lambda_U} B_V^o/p^nB_V^o\Bigr).
$$
Let us remark that for every $n\in \N$ there is $N:=N(n)\in\N$ such that
the image of $\underline{m}_U^N$ in $B_V^o/p^nB_V^o$ is $0$. Therefore we may write
$$
D_U^o|_V=\lim_{\infty\leftarrow n}\Bigl(D_U^o/\underline{m}_U^{N(n)}\otimes_{\Lambda_U}B_V^o/p^nB_V^o\Bigr).
$$
Therefore the first claim follows from lemma \ref{lemma:weakON} and the second is clear.

\end{proof}

Let now $(B, \underline{m})$ be any of the pairs $(\cO_K, \underline{m}_K)$ or $(\Lambda_U, \underline{m}_U)$, for $U\subset \cW^\ast$ a wide open disk defined over $K$
and let $A^o$, $D^o$ be either $A_k^o$, $D_k^o$, for some $k\in \cW^\ast(K)$ in the first case or $A_U^o$, $D_{U}^o$ in the second.

\begin{definition}
\label{def:filtration}

i) We define for every $n\in \N$ the following $B/\underline{m}^n$-submodule ${\rm Fil}_n (A^o/\underline{m}^nA^o)$ of $A^o/\underline{m}^nA^o$ by setting
$$
{\rm Fil}_n (A^o/\underline{m}^nA^o):=\oplus_{j=0}^n\bigl(\underline{m}^j/\underline{m}^n\bigr)f_j,
$$
where  $\{f_i\}_{i\in \N}$ is the orthornomal basis of $A^o$ described above.\smallskip

ii) We define the following filtration ${\rm Fil}^\bullet(D^o)$ of $D^o$. For $n\in \N$ we set
$$
{\rm Fil}^n(D^o):=\{\mu\in D^o\quad |\quad \mu(f_j)\in \underline{m}_B^{n-j} \mbox{ for all }
0\le j\le n.\}
$$
\end{definition}

\begin{proposition}
\label{prop:finitefil}

i) $D^o/{\rm Fil}^n(D^o)$ is a finite
$B/\underline{m}^n$-module and consequently an artinian $\cO_K$-module. Moreover the image of ${\rm Fil}^n(D^o)$ in
$D^o/\underline{m}^nD^o$, which we identify with the $B/\underline{m}^n$-dual of $A^o/\underline{m}^nA^o$, is the
orthogonal complement of ${\rm Fil}_n(A^o/\underline{m}^nA^o)$.

ii) For every $\gamma\in \Xi(\Z_p)$ and $\mu\in {\rm Fil}^n(D^o)$ we have
$\mu|\gamma\in {\rm Fil}^n(D^o)$. In particular $D^o/{\rm Fil}^n(D^o)$
is an artinian $\cO_K$ and $\Xi(\Z_p)$ module for every $n\ge 0$.

iii) The natural $B$-linear morphism $\displaystyle D^o\lra \lim_{\infty\leftarrow n}D^o/
{\rm Fil}^n(D^o)$ is an isomorphism.
\end{proposition}

\begin{proof}
Let us recall the $B$-linear map $\psi\colon D^o\lra \prod_{n\in \N} B$ of remark \ref{rmk:topologies}  defined by
$\psi(\mu)=\bigl(\mu(f_n)  \bigr)_{n\in \N}$. Let us remark that
$$
\psi\bigl({\rm Fil}^n(D^o)\bigr)=\prod_{j=0}^{n-1}\underline{m}_B^{n-j}\times \prod_{m\ge n}B,
$$
and therefore $\psi$ induces a $B$-linear isomorphism $D^o/{\rm Fil}^n(D^o)
\cong \prod_{j=0}^{n-1}B/\underline{m}_B^{n-j}$.
This proves the first statement in i) and also iii) because it shows that $D^o$ is separated and complete in the topology given by
${\rm Fil}^\bullet(D^o)$. For the second statetent of i) let us remark that ${\rm Ann}_{B/\underline{m}^n}\bigl(\underline{m}^i/\underline{m}^n\bigr)=\underline{m}^{n-i}/\underline{m}^n$ for all
$0\le i\le n$.

In order to prove ii) we could proceed as in the proof of lemma \ref{lemma:Akaction}, or better let us recall
that we have a natural decomposition $\Xi(\Z_p)=N^{\rm opp}T^+N$ where $N$ is the subgroup of $\GL_2(\Z_p)$
of upper triangular matrices, $N^{\rm opp}$ is the subgroup of lower triangular matrices and $T^+$ is the semigroup
of matrices $\left( \begin{array}{cc}{a} & {0} \\ {0} &{d}\end{array}\right)$ with $a\in \Z_p^\times$ and $d\in \Z_p-\{0\}$.
In order to show that ${\rm Fil}^n(D^o)$ is preserved by the matrices in $N$ it is enough to do it
for $\gamma:=\left( \begin{array}{cc}{1} & {1} \\ {0} &{1}\end{array}\right)$ (which is a topological generator of $N$).
Let $\mu\in {\rm Fil}^n(D^o)$, then we have
$(\mu |\gamma)(f_j)=\mu(\gamma f_j)$ and moreover
$$
(\gamma f_j)(x,y)=f_j(x, y+x)=k(x)(1+y/x)^j=\sum_{k=0}^j\left( \begin{array}{cc}{n} \\ {k}\end{array}\right)k(x)(y/x)^k
=\sum_{k=0}^j\left( \begin{array}{cc}{n} \\ {k}\end{array}\right)f_k(x,y).
$$
Therefore $(\mu |\gamma)(f_j)\in \sum_{k=0}^j\underline{m}_B^{n-k}\subset \underline{m}_B^{n-j}.$

To show that the matrices in $T^+$ preserve
${\rm Fil}^n(D^o)$ is is enough to show it for
$\delta:=\left( \begin{array}{cc}{1} & {0} \\ {0} &{p}\end{array}\right)\in T^+$.
We have $(\delta f_j)(x,y)=f_j(x, py)=k(x)p^j(y/x)^j=p^jf_j(x,y).$ Therefore for all
$\mu\in {\rm Fil}^n(D^o)$ we have $(\mu |\delta)(f_j)=p^j\mu(f_j)\in p^j\underline{m}_B^{n-j}\subset \underline{m}_B^{n-j}.$

Finally we leave it to the reader to check that for $\epsilon:=\left( \begin{array}{cc}{1} & {0} \\ {p} &{1}\end{array}\right)\in
N^{\rm opp}$ (which topologically generates this group), if $\mu\in {\rm Fil}^n(D^o)$ then
$(\mu |\epsilon)(f_j)\in \underline{m}_B^{n-j}$.
\end{proof}

We'd now like to show that the formation of the above defined filtrations commutes with base change.
More precisely, let as at the beginning of this section $U\subset \cW^\ast$ denote a wide open disk
and $\Lambda_U$ the $\cO_K$-algebra of bounded rigid functions on $U$. Let $k\in U(K)$.
Then if we denote by $t_k\in \Lambda_U$ a uniformizer at $k$, i.e. an element which vanishes of order $1$
at $k$ and nowhere else, then on the one hand $(\pi, t_k)=\underline{m}_U$ (let us recall that we denoted by
$\pi$ a fixed uniformizer of $K$) and we have an exact sequence
$$
0\lra \Lambda_U\stackrel{t_k}{\lra}\Lambda_U\stackrel{\rho_k}{\lra} \cO_K\lra 0
$$
which we call the specialization exact sequence.  Moreover the weak topolgy
on $\Lambda_U$ induces the $p$-adic topology on $\cO_K\cong \Lambda_U/t_k\Lambda_U$.

We have natural specialization maps

\begin{equation}\label{specialization_defn}
\begin{array}{rclcrcl}
A_{U}^o &\lra& A_k^o & and & \eta_k : D_{U}^o &\lra& D^o_k\\
f &\longmapsto & f_k && \mu &\longmapsto & \mu_\kappa\\
\end{array}
\end{equation}
where $f_U(x,y) := f(x,y)(k)$ and $\mu_k \in D_k^o$ is given by
$\mu_k \colon f \in A_k^o \longmapsto \mu(f_U)(k)$, where $f_U\in A_{U}^o$ is given
by $f_U(x,y) :=k_U(x)f(1,y/x)$.

\begin{proposition}\label{specialization_prop}
Let $U \subseteq \cW^\ast$ be a wide open disk, let $k\in U(K)$, and let $t_k\in \Lambda_U$
be a uniformizer at $k$.   Then we have canonical exact
sequences of $\Xi(\Z_p)$-modules
$$
\begin{array}{rccccccl}
0 &\lra& A_{U}^o &\buildrel t_k \over \lra& A_{U}^o &\buildrel \over \lra & A_k^o&\ra 0\\
\\
0 &\lra& D_{U}^o &\buildrel t_k \over \lra& D_{U}^o &\buildrel \eta_k \over \lra & D_k^o&\ra 0.\\
\end{array}
$$
\end{proposition}

For the proof if $\Lambda_U$ is replaced by an affinoid algebra  see \cite{harris_iovita_stevens}, section \S 3.
The arguments are the same in this case.

\begin{lemma}
\label{lemma:filspec}
With notations as above we have: $\eta_k\bigl({\rm Fil}^n(D_{U}^o)\bigr)={\rm Fil}^n(D_k^o).$
\end{lemma}

\begin{proof}
Let us recall that we have two commutative diagrams
$$
\begin{array}{ccccccccc}
D_{U}^o&\stackrel{\psi_{U}}{\cong}&\prod_{n\in \N}\Lambda_U&&D_k^o&\stackrel{\psi_k}{\cong}&
\prod_{n\in\N}\cO_K\\
\cup&&\cup&\mbox{ and }&\cup&&\cup\\
{\rm Fil}^n(D_{U})^o&\cong&\prod_{j=0}^{n-1}\underline{m}_U^{n-j}\times\prod_{m\ge n}\Lambda_U&&
{\rm Fil}^n(D_k^o)&\cong&
\prod_{j=0}^{n-1}\underline{m}_K^{n-j}\times\prod_{m\ge n}\cO_K
\end{array}
$$
The lemma follows observing that the diagram
$$
\begin{array}{llllllll}
D_{U}^o&\stackrel{\eta_k}{\lra}&D_k^o\\
\downarrow\psi_{k_U}&&\downarrow\psi_{k}\\
\prod_{n\in \N}\Lambda_U&\stackrel{\prod_{n\in \N}\rho_k}{\lra}&\prod_{n\in \N}\cO_K
\end{array}
$$
is commutative and that for every $n\in \N$ we have $\rho_k(\underline{m}_U^n)=\underline{m}_K^n.$
\end{proof}

Let us now suppose that $k\in \cW^\ast(K)$ is a classical weight i.e. $k$ is associated to a
pair $(k,k)$ with $k\in \N$. We define ${\cal P}_k^o\subset A_k^o$ as the subset of functions $f\colon T_0\to \cO_K$ which are homogeneous polynomials of degree $k$. It is an $\cO_K$-submodule invariant for the action of the semigroup $\Xi(\Z_p)\cap \GL_2(\Q_p)$. Dualizing we obtain a $\Xi(\Z_p)\cap \GL_2(\Q_p)$-equivariant, surjective, $\cO_K$-linear map
$$
\rho_{k}\colon D_k^o\lra V_k^o:={\rm Hom}_{\cO_K}\big({\cal P}_k^o,\cO_K\big).
$$
Let us remark that we may identify $V_k^o$ with ${\Symm}^k(T)\otimes_{\Z_p} \cO_K$ with its natural right action of $\Xi(\Z_p)\cap \GL_2(\Q_p)$.  For every $n\ge 1$ we set ${\rm Fil}^n(V_k^o):=\rho_k({\rm Fil}^n(D_k^o))$. We get a filtration inducing the $p$-adic topology on $V_k^o$. We view $V_k^o$ as the continuous representation of $\Gamma$ defined by the projective system $\big(V_{k,m}^o\big)_{m\in\N}$ with  $V_{k,m}^o:= V_k^o/{\rm Fil}^m V_k^o$ and we set $V_k:=V_k^o\otimes_{\cO_K} K$.

Therefore, if $U\subset \cW^\ast$ is a wide open disk which contains the classical weight
$k$, we have natural $\Xi(\Z_p)\cap \GL_2(\Q_p)$-equivariant maps
\begin{equation}\label{rhok}
D_{U}\lra D_k \stackrel{\rho_{k}}{\lra} V_k,
\end{equation}
and the maps are compatible with the filtrations.

\subsection{Overconvergent and $p$-Adic Families of Modular Symbols}
\label{sec:overmodularsymbols}

\bigskip
\noindent
Let us fix an integer $N\ge 3$ as
in the introduction and let $\Gamma=\Gamma_1(N)\cap \Gamma_0(p)\subset \Iw\subset \Xi(\Z_p)$.
Let us also fix a wide open disk $U\subset \cW^\ast$ and its associated universal character $k_U$
and let $k\in U(K)$ be a weight.
Let us also recall that we have natural orthonormalizable $\Lambda_{U,K}$-Banach-modules
$D_U:= D_{k_U}$ with continuous action of the monoid $\Xi(\Z_p)$ and
specialization maps $D_U\lra D_k$ which are $\Xi(\Z_p)$-equivariant. In particular we'll be interested in the
$\Lambda_{U,K}$-module $\rH^1\bigl(\Gamma,
D_U\bigr)$, which we call {\bf module of $p$-adic families of modular symbols},
 the $K$-vector space $\rH^1\bigl(\Gamma, D_k\bigr)$, which we call {\bf overconvergent modular symbols of weight $k$} and
the specialization map $\rH^1\bigl(\Gamma, D_U\bigr)\lra \rH^1\bigl(\Gamma, D_k\bigr)$.

Moreover, if $k\in U(K)$ happens to be a classical weight we also have the module
$\rH^1\bigl(\Gamma, V_k\bigr)$, which we call {\bf module of classical modular
symbols}, and maps $\rho_{k}\colon \rH^1\bigl(\Gamma, D_k\bigr)\lra \rH^1\bigl(\Gamma, V_k\bigr)$ obtained from (\ref{rhok}) in \S\ref{sec:locallyanalytic}.

\bigskip
\noindent
{\em Relationship with continuous $\Gamma$-cohomology}.

\bigskip\noindent
In the notations above let $D$ be any of the $\Xi(\Z_p)$-modules $D_{U}^o, D_k^o$, with $k\in U(K)$  and if $k\in U(K)$
is a classical weight $V_k^o$. We consider the pair $\bigl( D, {\rm Fil}^\bullet(D)\bigr)$ and wish to study the relationship
between $\rH^1(\Gamma, D)$ and $\rH^1_{\rm cont}\Bigl(\Gamma, \bigr(D/{\rm Fil}^n(D)\bigl)_{n\in \N}\Bigr)$.

Let us recall that if we denote by ${\rm Cont}(\Gamma)$ the category of projective systems
$(M_n)_{n\in \N}$ where each $M_n$ is a discrete, torsion $\Gamma$-module, then
$\rH^i_{\rm cont}(\Gamma, - )$ is the $i$-th right derived functor of the functor
$$
(M_n)_{n\in \N}\lra \rH^0\bigl(\Gamma, \lim_{\infty \leftarrow n}M_n\bigr).
$$
The degeneration of the Leray spectral sequence for the composition of the two left exact functors in
the above definition gives for every object  $(M_n)_{n\in \N}\in {\rm Cont}(\Gamma)$ an exact sequence of abelian groups
$$
0\lra {\lim}^{(1)}\rH^0(\Gamma, M_n)\lra \rH^1_{\rm cont}\bigl(\Gamma, (M_n)_{n\in \N}\bigr)\lra
\lim_{\infty \leftarrow n}\rH^1(\Gamma, M_n)\lra 0.
$$
For a pair $\bigl(D, {\rm Fil}^\bullet(D)\bigr)$ as above we have the projective system of {\bf artinian} $\cO_K$ and $\Gamma$-modules $(D_n)_{n\in \N}$, where
$D_n:=D/{\rm Fil}^n(D)$ and the morphisms in the projective limit are the natural projections. Therefore the projective system $\Bigl(\rH^0(\Gamma, D_n)
\Bigr)_{n\in \N}$ satisfies the Mittag-Leffler condition and  consequently we have a natural isomorphism
$$
\rH^1_{\rm cont}\bigl(\Gamma, (D_n)_{n\in \N}\bigr)\cong \lim_{\infty \leftarrow n}\rH^1(\Gamma, D_n).
$$

On the other hand we have:

\begin{lemma}
\label{lemma:arithmeticoh}
The natural map
$$
\rH^1(\Gamma, D)\lra \lim_{\infty\leftarrow n}\rH^1(\Gamma, D_n)
$$
is an isomorphism.
\end{lemma}

\begin{proof}
If $M$ is a $\Gamma$-module we denote by $B^1(\Gamma, M)$ the group of $1$-coboundaries with coefficients in $M$
and by $Z^1(\Gamma, M)$ the group of $1$-cocycles with values in $M$ (no continuity condition is involved in the
definition of either coboundaries or cocycles).
We have a natural commutative diagram with exact rows:

$$
\begin{array}{ccccccccccc}
\rH^0(\Gamma, D)&\hookrightarrow&B^1(\Gamma, D)&\stackrel{d}{\lra}&Z^1(\Gamma, D)&\lra& \rH^1(\Gamma, D)&\lra&0\\
\downarrow f&&\downarrow g&&\downarrow h&&\downarrow u\\
\lim_{\leftarrow}\rH^0(\Gamma, D_n)&\hookrightarrow&\lim_{\leftarrow}B^1(\Gamma, D_n)&\stackrel{\lim_{\leftarrow}d_n}{\lra}&
\lim_{\leftarrow}Z^1(\Gamma, D_n)&\stackrel{\alpha}{\lra}&\lim_{\leftarrow}\rH^1(\Gamma, D_n)
\end{array}
$$
We first claim that the map $\alpha$ in the above diagram is surjective. For this, let us recall that for every
$n\in\N$ we have an exact sequence of abelian groups
$$
0\lra B^1(\Gamma, D_n)/\rH^0(\Gamma, D_n)\lra Z^1(\Gamma, D_n)\lra \rH^1(\Gamma, D_n)\lra 0.
$$
Then let us remark that $B^1(\Gamma, D_n)$ is an artinian $\cO_K$-module and so is $B^1(\Gamma, D_n)/H^0(\Gamma, D_n)$. This implies that the projective system
$\bigl(B^1(\Gamma, D_n)/\rH^0(\Gamma, D_n)   \bigr)_{n\in \N}$ satisfies the Mittag-Leffler condition and so $\alpha$ is indeed surjective.

Moreover, let us remark that in the above diagram the maps $f$, $g$, $h$ are all isomorphisms.  For $f$ and $g$ this follows immediately from the definition and the fact that $\displaystyle D=\lim_{\infty \leftarrow n}D_n$.
The injectivity of $h$ follows from the fact that $\cap_{n\ge 0}{\rm Fil}^n(D)=\{0\}$.
On the other hand let us  first remark that $\Gamma$ is a finitely generated group and moreover
that a $1$-cocycle
of $\Gamma$ is determined by its values on a family of generators. This implies that $h$ is surjective.

Now by the five lemma,
$u$ is then an isomorphism as well.
\end{proof}

\begin{remark}
\label{rmk:toph1}
In the notations of the proof of lemma \ref{lemma:arithmeticoh}, let us remark that
for every $n\in \N$, $Z^1(\Gamma, D_n)$ is also an artinian $\cO_K$-module. The reason is: $\Gamma$ is a finitely generated group and
a $1$-cocycle is determined on its values on a set of generators of $\Gamma$. It follows that
$\rH^1(\Gamma, D_n)$ is an artinian $\cO_K$-module for all $n\in \N$ and therefore $\rH^1(\Gamma, D)$ has a natural structure
as profinite $\cO_K$-module. In particular it is compact. Moreover, if $D=D_{U}$ for a certain $U\subset \cW^\ast$, then
$\rH^1(\Gamma, D_{U})$ is naturally a $\Lambda_U$-module and its profinite topology is the same as the weak topology, i.e.
the $\underline{m}_U$-adic topology on $\rH^1(\Gamma, D_{U})$. In particular $\rH^1(\Gamma, D_{U})$ is
complete and separated for the $\underline{m}_U$-adic topology.
\end{remark}

Let us remark that we proved the following theorem.

\begin{theorem}
\label{thm:cohcont}
Let $D$ be one of $D_{U}^o$, $D_k^o$ or if $k$ is a classical weight, $V_k^o$.

a) We have canonical isomorphisms
$$
\rH^1_{\rm cont}\bigl(\Gamma, (D_n)_{n\in \N}\bigr)\cong \lim_{\infty \leftarrow n}\rH^1(\Gamma, D_n)\cong \rH^1(\Gamma, D).
$$

b) The isomorphisms at a) above are compatible with specializations.
\end{theorem}

\begin{proof}
a) follows from lemma \ref{lemma:arithmeticoh} and b) from the discussion on specialization in the previous section.
\end{proof}

\bigskip
\noindent {\em Hecke operators.}

\bigskip\noindent
Let $M$ be any one of the modules $\rH^1\bigl(\Gamma, D_U\bigr)$, $\rH^1\bigl(\Gamma, D_k\bigr)$,
$\rH^1\bigl(\Gamma,
V_k\bigr)$. The action of $\Xi(\Z_p)$ on the coefficients defines actions of Hecke operators $T_\ell$ for $\ell$ not dividing
$Np$ and $U_p$ on the module $M$ (see
\cite{ash_stevens} and \cite{harris_iovita_stevens}). Moreover $U_p$ is completely continuous on $M$.
Let now fix $h\in \Q$, $h\ge 0$ and denote by $D$ one of
$D_k$ or $V_k$. Then we have a natural direct sum decomposition
$$
\rH^1\bigl(\Gamma, D\bigr)\cong \rH^1\bigl(\Gamma, D\bigr)^{(h)}\oplus \rH^1\bigl(\Gamma, D\bigr)^{(>h)},
$$
where the decomposition is characterized by the following properties:

a) $\rH^1\bigl(\Gamma, D\bigr)^{(h)}$ is a finite $K$-vector space and  for every
$x\in \rH^1\bigl(\Gamma, D\bigr)^{( h)}$ there is a non-zero polynomial
$Q(t)\in K[t]$ of slope smaller or equal to $h$ such that $Q^\ast(U_p)\cdot x=0$, where $Q^\ast(t)=t^{{\rm deg}Q}Q(1/t)$.

b) For every polynomial $Q(t)$ as at a) above the linear map $Q^\ast(U_p)\colon \rH^1\bigl(\Gamma, D\bigr)^{(> h)}\lra \rH^1\bigl(\Gamma, D\bigr)^{(> h)}$ is an
isomorphism.

Moreover we have the following two results (see \cite{ash_stevens})

\begin{theorem}
\label{thm:classicity}
Suppose $k=(k_0,i)\in \cW^\ast(K)$ a classical weight, $h\ge 0$ is a slope such that $h<k_0-1$.
Then $\rho_{k}$ induces an isomorphisms
$$
\rH^1\bigl(\Gamma, D_k\bigr)^{( h)}\cong \rH^1\bigl(\Gamma, V_k\bigr)^{(h)}.
$$
\end{theorem}

and

\begin{theorem}
\label{familiesslopedec}
Let $k\in \cW^\ast(K)$ be an accessible weight and $h\ge 0$ a slope. Then there is a wide open disk
defined over $K$, $U\subset \cW^\ast$, containing $k$ such that

i) We have a natural, $B_U$-linear  slope $\le h$-decomposition
$$
\rH^1\bigl(\Gamma, D_U\bigr)\cong \rH^1\bigl(\Gamma, D_U\bigr)^{(h)}\oplus \rH^1\bigl(\Gamma, D_U\bigr)^{(> h)}
$$
satisfying analogue properties as in a) and b) above.\smallskip

ii) The slope decomposition at i) above is compatible with specialization i.e. the map $$\psi_k\colon \rH^1\bigl(\Gamma,
D_U\bigr)\lra \rH^1\bigl(\Gamma,
D_k\bigr)$$satisfies
$\psi_k\bigl(\rH^1\bigl(\Gamma, D_U\bigr)^{(\&)}\bigr)\subset \rH^1\bigl(\Gamma, D_k\bigr)^{(\&)}$,
where $(\&)\in \{ (h), (> h)\}$

\end{theorem}

\begin{proof}
Claim ii) follows from the functoriality of the slope $\le h$-decompositions proved in \cite{ash_stevens}.

To show i)  it suffices to produce a slope $\le h$-decomposition at the level of cochains $C^\bullet(\Gamma, D_U)$ using a finite resolution of $\Z$ by finite and free $\Z[\Gamma]$-modules as explained in \cite{ash_stevens} section \S 4.

We start by choosing a wide open disk $U'\subset \cW^\ast$ defined over $K$ such that
$k\in U'(K)$.
Let  $(e_i)_{i\in I}$ be a $B_{U'}$-orthonormal basis of $D_{U'}$  defined in lemma
\ref{lemma:weakON}) and let $V={\rm Spm}(K\langle T\rangle)$ be a closed disk centered at
$k$ contained in $U'$. Thanks to lemma \ref{lemma:restrizione} and the theory in \cite{ash_stevens}
section \S 4, possibly after shrinking $V$ to a smaller affinoid disk, we may assume that
the Fredholm determinant of the $U_p$-operator $F^\bullet_V$ on the complex of group cochains
$C^\bullet\bigl(\Gamma, D_{U'}|_V\bigr)$ admits a slope $\le h$-decomposition. Let $U$ be the wide open disk associated to the
noetherian local $\cO_K$-algebra $\Lambda_U:=\cO_K[\![T]\!]$. Let us remark that the Fredholm determinant of $U_p$  acting on $C^\bullet\bigl(\Gamma, D_{U'}|_V\bigr)$, $F^\bullet_V$, is the same as the Fredholm determinant of $U_p$ acting on $C^\bullet\big(\Gamma, D_U\bigr)$, $F^\bullet_U$,
 as they are both computed using the same (weak) ON basis.
Because the Banach norm of $\Lambda_U[1/p]=B_U$ restricts to the Gauss norm of $B_V$ it follows that
the slope $\le h$-decomposition of $F^\bullet_V$ determines a slope $\le h$-decompositon of
$F^\bullet_U$.  Therefore we obtain the slope $\le h$-decomposition of $C^\bullet(\Gamma, D_U)$.

\end{proof}

\subsection{The geometric picture}
\label{sec:geometric}

Let us recall from the beginning of section \S 2 the modular curves $X(N,p)\lra X_1(N)$, their natural formal models
$\cX(N,p)\lra \cX_1(N)$ and if $w\in \Q$, $0\le w\le p/(p+1)$ we also had  a rigid analytic space $X(w)\subset X(N,p)$
and its formal model $\cX(w)$, with its natural morphism $\cX(w)\lra \cX(N,p)$. All these rigid spaces and formal schemes are
in fact log rigid spaces and respectively
log formal schemes, which are all log smooth and all the maps described are maps of log formal schemes or log rigid spaces.

\
\noindent
{\em Sheaves on $X(N,p)_{\Kbar}^{\rm ket}$ associated to modular symbols:}\enspace Let $\cE\lra X(N,p)$ be the universal generalised
elliptic curve, and let us denote, as
in section \S 3.1, by $\cT$ the $p$-adic Tate-module of $\cE$, seen as a continuous sheaf on the Kummer \'etale site of $X(N,p)$,
denoted $X(N,p)^{\rm ket}$. If
$\eta=\Spec({\Bbb K})$ denotes a geometric generic point of $X(N,p)$, let $\cG$ denote the geometric Kummer \'etale
fundamental group associated to
$\bigl(X(N,p),\eta\bigr)$ and let $T:=\cT_\eta$. One can easily see that $T$ is a free $\Z_p$-module of rank $2$ with continuous
action of $\cG$. Let us choose a
$\Z_p$-basis $\{\epsilon_1,\epsilon_2\}$ of $T$ satisfying the properties:
$\langle \epsilon_1, \epsilon_2\rangle=1$ and $\epsilon_1(\mbox{ mod }pT)\in \cE_\eta[p]$
does not belong to the universal level $p$-subgroup $C$. We let $T_0:=\{a\epsilon_1+b\epsilon_2\quad |\quad a\in
\Z_p^\ast, b\in \Z_p\}$. Then $T_0$ is a compact subset
of $T$ preserved by $\cG$ which can be identified to $\Z_p^\times\times \Z_p$.
Moreover the right action of $\cG$ on the above chosen basis defines a continuous
group homomorphism
$$
\gamma\colon \cG\lra \Iw\mbox{ defined by } (\epsilon_1\sigma,\epsilon_2\sigma)= (\epsilon_1,\epsilon_2)
\gamma(\sigma)\mbox{ for }\sigma\in \cG.
$$
Therefore, if $k\in U\subset \cW^\ast$ and $n\ge 1$ as in section \ref{sec:locallyanalytic} then via the homomorphism
$\gamma$ above, the $\Iw$-modules $A_U$, $D_U$, $A_k$,
$D_k$ and if $k$ is classical $V_k$, can be seen as ind-continuous representations of $\cG$.  More
precisely, for example  $\displaystyle A_U=\lim_{\rightarrow}\Bigl( \lim_{\infty \leftarrow n} (A_U^o/\underline{m}_U^n A_U^o)  \Bigr)$ and $\displaystyle D_U=\lim_{\rightarrow}\Bigl( \lim_{\infty \leftarrow n}(D_U^o/{\rm Fil}^n(D_U^o)  \Bigr)$, where the inductive limits are taken with respect to the multiplication by $p$-map. We denote by $\cA_U$,
$\cD_U$, $\cA_k$, $\cD_k$  and if $k$ is classical by $\cV_k$ the ind-continuous shaves on $X(N,p)_{\Kbar}^{\rm ket}$ associated to these representations.  Let us remark that if the classical weight is associated to the pair $\bigl(k,k({\rm mod} (p-1)\bigr)$,
for $k\in \Z$, $k\ge 0$ in fact we have $\cV_k\cong \Symm^{k}(\cT)\otimes_{\Z_p}K$, as ind-continuous sheaves on $X(N,p)_{\Kbar}^{\rm ket}$.

\

{\it Notation:} For later use, for $A=A_U$ or $A=A_k$ we write $\cA^o:=\bigl(\cA^o_{n}\bigr)_{n\in \N}$ for the continuous sheaf on Faltings' site $\fX(N,p)$
associated to  the continuous representation of  $A^o=\bigl( A^o/\underline{m}^n A^o \bigr)_{n\in\N}$  of the Kummer \'etale fundamental group $\cG$ of $X(N,p)$.
The ind-continuous sheaf  $\cA$ is simply $\cA^o[1/p]$.

Analogously, for $D=D_U$ or $D_k$, we write $\cD^o:=\bigl(\cD_{n}\bigr)_{n\in\N}$  for the continuous sheaf associated to $D^o/{\rm Fil}^n(D^o)$. Then, $\cD$ is the
ind-continuous sheaf $\cD^o[1/p]$.

\

We proceed as in section \ref{sec:heckeoperators} in order to define Hecke operators
on $\rH^1\bigl( X(N,p)_{\Kbar}^{\rm ket}, \cD_U \bigr)$ and $\rH^1\bigl( X(N,p)_{\Kbar}^{\rm ket}, \cD_k \bigr)$
and if $k$ is classical on $\rH^1\bigl( X(N,p)_{\Kbar}^{\rm ket}, \cV_k \bigr)$. Let $\ell$ denote a prime integer not dividing
$N$. Let $X(N,p)_{\ell}$ be the
modular curve classifying generalized elliptic curves with $\Gamma_1(N)\cap \Gamma_0(p)\cap \Gamma_0(\ell)$-level structure for
$\ell$ not dividing $p$ and with
$\Gamma_1(N)$-level structure and two group scheme $C$, $H\subset \cE$ defining a $\Gamma_0(p)$-level structures such that
$C\cap H=\{0\}$. We  have morphisms
$$X(N,p) \stackrel{\pi_1}{\longleftarrow} X(N,p)_{\ell} \stackrel{\pi_2}{\lra} X(N,p),$$
where $\pi_1$ forgets $H$ while $\pi_2$ is defined by taking the quotient
by $H$. They are finite and Kummer log \'etale. The dual $\pi_\ell^\vee\colon \bigl(\cE/H\bigr)^\vee \to \cE^\vee$ of the
universal isogeny $\pi_\ell\colon \cE\to
\cE/H$ over $X(N,p)_{\ell}$ provides a map $\pi_\ell^\vee \colon p_2^\ast(T) \to p_2^\ast\bigl(T\bigr)$.
Here we identify $\cE\cong \cE^\vee$ and $\cE/H\cong
\bigl(\cE/H\bigr)^\vee$ via the principal polarizations. The condition $H\cap C=\{0\}$ implies that $\pi_\ell^\ast$
restricts to a map $p_2^\ast\big(T_0\big) \to
p_1^\ast\bigl(T_0\bigr)$. Proceeding as in section \ref{sec:heckeoperators} we get Hecke operators  on
$\rH^1\bigl(X(N,p)_{\Kbar}^{\rm ket}, \cD_U\bigr)$ and on
$\rH^1\bigl(X(N,p)_{\Kbar}^{\rm ket}, \cD_k\bigr)$, which commute with the action of $G_K$ and are compatible with specializations.

\begin{proposition} We have natural isomorphisms as Hecke modules, compatible with specializations
$$
\rH^1\bigl(\Gamma, D_U\bigr)\cong \rH^1\bigl(X(N,p)_{\Kbar}^{\rm ket}, \cD_U\bigr)\mbox{ and }\rH^1\bigl(\Gamma, D_k\bigr)\cong
\rH^1\bigl(X(N,p)_{\Kbar}^{\rm ket}, \cD_k\bigr).
$$
Here on the left modules the Hecke operators are defined by the action of the monoid $\Xi(\Z_p)$ on the coefficients.
\end{proposition}
\begin{proof} We first prove that, given a finite representation $F$ of $\cG$ and considering the associated locally constant sheaf
$\cF$ on $X(N,p)_{\Kbar}^{\rm ket}$, we have an isomorphism $\rH^1\bigl(\Gamma,F\bigr)\cong \rH^1\bigl(X(N,p)_{\Kbar}^{\rm ket}, \cF\bigr)$, functorial in $F$.
First of all notice that  restriction to the open modular curve where the log-structure is trivial, i.e. to $Y(N,p)\subset X(N,p)$, induces an isomorphism
$\rH^1\bigl(X(N,p)_{\Kbar}^{\rm ket}, \cF\bigr)\cong \rH^1\bigl(Y(N,p)_{\Kbar}^{\rm et}, \cF\bigr)$ thanks to \cite[Cor.~7.5]{illusie2}. As $Y(N,p)_{\Kbar}$ is a
smooth affine curve and for every embedding $\Kbar\subset \mathbb{C}$  the fundamental group of $Y(N,p)_{\mathbb{C}}$ is $\Gamma$, it is a classical result that
$\rH^1\bigl(Y(N,p)_{\Kbar}^{\rm et}, \cF\bigr)\cong \rH^1\bigl(Y(N,p)_{\mathbb{C}}^{\rm et}, \cF\bigr)\cong \rH^1\bigl(\Gamma,F\bigr)$.

Therefore, if we denote by $D$ any one of
$D_U$, $D_k$ or if $k$ is classical $V_k$ and by $\displaystyle\bigl((D^o_n)_{n\in \N}  \bigr)\otimes_{\cO_K}K$ with
$D_n^o:=D^o/{\rm Fil}^n(D^o)$, the ind-continuous representation of $\cG$ associated to $D$, let
$\displaystyle \cD:=\bigl((\cD_n^o)_{n\in \N}\bigr)\otimes_{\cO_K}K$ be the ind-continuous \'etale sheaf
on $X(N,p)_{\Kbar}^{\rm ket}$ associated to it. By theorem \ref{thm:cohcont} and the discussion above
we have natural isomorphisms
$$
\rH^1\bigl(\Gamma, D\bigr)\cong \rH^1_{\rm cont}\Bigl(\Gamma, \bigl((D_n^o)_{n}\otimes K\Bigr)
\cong \rH^1_{\rm cont}\Bigl(X(N,p)_{\Kbar}^{\rm ket}, \bigl((\cD_n^o)_{n}\bigr)\otimes K\Bigr)
=: \rH^1\bigl(X(N,p)_{\Kbar}^{\rm ket}, \cD\bigr).
$$

As the definition of the Hecke
operators uses the Hecke
correspondence $X(N,p)_{\ell}$, it is clear that for $D$ one of $D_U$, $D_k$ or $V_k$ the isomorphism $\rH^1\bigl(X(N,p)_{\Kbar}^{\rm ket},
\cD\bigr)\cong
\rH^1\bigl(Y(N,p)_{\mathbb{C}}^{\rm et}, \cD\bigr)$ is Hecke equivariant and the claim follows.

\end{proof}

\

{\em Sheaves on Faltings' site associated to modular symbols:}\enspace Let us now denote by $\fX(N,p)$, Faltings' site associated to
the pair $\bigl(\cX(N,p),
X(N,p)\bigr)$. The map of sites $u\colon \fX(N,p) \lra X_{\Kbar}^{\rm ket}$, given by $(U,W)\mapsto W$,
sends covering families to covering families, commutes with
fibre products and sends the final object to the final object. It defines a morphism of topoi
$u_\ast\colon \Sh(X_{\Kbar}^{\rm ket}) \lra \Sh(\fX(N,p))$ which
extends to inductive systems of continuous sheaves. In particular all the ind-continuous Kummer \'etale sheaves
$\cD_U$, $\cD_k$ and, if $k$ is a classical weight, $\cV_k$ can be seen as ind-continuous sheaves on $\fX(N,p)$ by
 applying $u_\ast$. For simplicity we omit
$u_\ast$ from the notation.

\begin{proposition}\label{prop:XfXcompcoh} The natural morphisms $$ \rH^1\bigl(\fX(N,p), \cD_U\bigr)\lra
\rH^1\bigl(X(N,p)_{\Kbar}^{\rm ket}, \cD_U\bigr) \mbox{ and }\rH^1\bigl(\fX(N,p), \cD_k\bigr)\lra \rH^1\bigl(X(N,p)_{\Kbar}^{\rm ket},
\cD_k\bigr)$$are
isomorphisms of $G_K$-modules,compatible with specializations and action of the Hecke operators.
\end{proposition}
\begin{proof} As $\cX(N,p)$ is proper and log smooth over $\Spf(\cO_K)$ it follows from
\cite[Thm.~9]{faltingsAsterisque} that for $\cF$ a finite locally constant sheaf the natural maps
$\rH^1\bigl(\fX(N,p), \cF\bigr)\lra \rH^1\bigl(X(N,p)_{\Kbar}^{\rm
ket}, \cF\bigr)$ are isomorphisms. As $\cD_U$ and $\cD_k$ are inductive limits of projective limits of finite locally constant
sheaves in a compatible way by \ref{prop:finitefil},
the maps in the proposition are isomorphisms. The Hecke operators are defined in terms of the Hecke correspondence
$p_1$, $p_2\colon X(N,p)_{\ell} \lra X(N,p)$ and
the trace maps $p_{1,\ast}\bigl(p_1^\ast(\cF)\big)\to \cF$ on $X(N,p)_{\Kbar}^{\rm ket}$ and on $\fX(N,p)$ defined in
(\ref{def:trace}). As they are compatible via
$u_\ast$, the displayed isomorphisms are equivariant for the action of the Hecke operators.
\end{proof}



\section{Modular sheaves}
\label{sec:modularsheaves}

Modular sheaves are only defined on the site $\fX(w)$, $0\le w< p/(p+1)$.
Let $\cW$ denote the weight space for $\GL_{2/\Q}$ i.e. the rigid analytic
space over $\Q_p$ associated to the noetherian $\Z_p$-algebra $\Z_p[\![\Z_p^\times]\!]$
and let us fix $(B, \underline{m})$, a complete, local, regular, noetherian $\cO_K$-algebra.
Let us recall that $B$ is complete and separated for its $\underline{m}$-adic toplogy
(the weak topology) and therefore  also for the $p$-adic toplogy.
We denote by $B_K:=B\otimes_{\cO_K}K$ and by $|| \ ||$ the Gauss norm on the Banach $K$-algebra
$B_K$. Let $k\in \cW(B_K)$ be a $B_K$-valued weight, i.e. a continuous
group homomorphism $k:\Z_p^\times\lra B^\times$.
We embed $\Z$ in $\cW(\Q_p)$ by sending $k\in \Z$ to the character $a\rightarrow a^k$ and in general if
$k\in \cW(B_K)$ as above and $t\in \Z_p^\times$ we use the additive notation $t^k:=k(t)$.

 Once we fixed $B$ and $k\in \cW(B_K)$ we define
$$
r:=\min\{n\in \N\quad | \quad n>0 \mbox{ and } ||k(1+p^n\Z_p)-1||<p^{\frac{-1}{p-1}}\},
$$
and we fix $w\in \Q$, $w\ge 0$ and $w<2/(p^r-1)$ if $p>3$ and $w<1/3^r$ if $p=3$. We say that
$w$ is ${\bf adapted}$ to $r$ (and to $k$). Let us also note that given $B,k,r$ as above there is a unique
$a\in B_K$ such that for all $t\in 1+p^r\Z_p$ we have $t^k=\exp(a\log(t))$.

There are two instances of the above general situation which will be relevant in what follows:

a) $B=\cO_K$ so $B_K=K$, in that case $k\in \cW(K)$ is a $K$-valued weight.

b) We fix first $r>0$, $r\in \N$ and denote $\cW_r:=\{k\quad |\quad ||k(1+p^r\Z_p)-1||<p^{-1/(p-1)}\}$.
It is a wide open rigid subspace of $\cW$. Let now $U\subset \cW_r$ be a  wide open
disk of $\cW_r$, let $A_U$ be the affinoid algebra of $U$ and
we define
$$
\Lambda_U=A_U^{\mathfrak{b}}:=\{f\in A_U\mbox{ such that } |f(x)|\le 1\mbox{ for all points } x\in U\}
$$
the $\cO_K$-algebra of the {\bf bounded rigid functions} on $U$. Then $\Lambda_U$
is a complete, local, noetherian $\cO_K$-algebra non-canonically isomorphic to $\cO_K[\![T]\!]$.
Let also:
$k_U:\Z_p^\times\lra \Lambda^\times_U$ be the universal character, i.e. if $t\in \Z_p^\times$ and
$x\in U$ we have $t^{k_U}(x)=t^x$.

\begin{remark}
\label{rmk:lambdar}
Let $(B, \underline{m})$ be as above and let $R$ be a $p$-adically complete and separated $\cO_K$-algebra
in which $p$ is not a zero divisor. We denote by $R\hat{\otimes}B$ the ring
$$
R\hat{\otimes}B:=\lim_{\infty \leftarrow n}\bigl(R/p^nR\otimes_{\cO_K}B/\underline{m}^n \bigr)\cong
\lim_{\leftarrow,n}\bigl(R/p^nR\hat{\otimes}_{\cO_K}B/p^nB\bigr),
$$
where we denoted by $R/p^nR\hat{\otimes}_{\cO_K}B/p^nB$ the completion of the usual tensor product with
respect to the ideal generated by the image of $R\otimes_{\cO_K}\underline{m}$.

In particular, if $B=\cO_K$ then $R\hat{\otimes}B=R$ and if $B\cong \cO_K[\![T]\!]$ then
$R\hat{\otimes}B\cong R[\![T]\!]$.
\end{remark}

\subsection{The sheaves $\Omega^k_\fX(w)$}
\label{sec:omegakappa}

In this section we recall the main constructions of chapter 3 of
\cite{andreatta_iovita_stevens}
in a slightly different context. Let $w\in \Q$ be such that $0\le w <p/(p+1)$.

Let $f\colon \cE\lra \cX(w)$ and $f_K\colon \cE_K\lra X(w)$ denote the universal semi-abelian scheme over $\cX(w)$ and respectively its generic fiber and let
$\cT\lra X(w)$ denote the $p$-adic Tate module of $\cE_K^\vee\lra X(w)$ seen as a continuous sheaf on $\cX(w)_\Kbar^{\rm ket}$. Notice that $\cE$ admits a canonical
subgroup $\cC_1\subset \cE$. Let $\cT_0\subset \cT$ be the inverse image of $\cC_{1,K}^\vee\backslash \{0\}$ in $\cT$ via the natural map $\cT\to \cE_K^\vee[p] \to
\cC_{1,K}^\vee$. Then $\cT$ and $\cT_0$ are continuous, locally constant sheaves on $\cX(w)_\Kbar^{\rm ket}$ and as such can be seen as a continuous sheaf on
$\fX(w)$. Notice that the sheaf $\cT$ is a continuous sheaf of abelian groups, while $\cT_0$ is a continuous sheaves of sets and it is endowed with a natural action
of $\Z_p^\ast$.

Let $e\colon \cX(w)\lra \cE$ denote the identity section of $f$ and
let $\omega_{\cE/\cX(w)}:=e^\ast\bigl(\Omega^1_{\cE/\cX(w)}\bigr)$. It is a locally free
$\cO_{\cX(w)}$-module of rank $1$ and we denote by $\omega_{\cE/\fX(w)}:=
v_{\cX(w)}^\ast\bigl(\omega_{\cE/\cX(w)}\bigr)$. Then $\omega_{\cE/\fX(w)}$ is a
a continuous sheaf on $\fX(w)$, a locally free
$\hO_{\fX(w)}^{\rm un}$-module of rank $1$.

We have a natural sequence of sheaves and morphisms of sheaves on $\fX(w)$ called the Hodge-Tate
sequence of sheaves for $\cE/\cX(w)$
$$
0\lra \omega_{\cE/\fX(w)}^{-1}\otimes_{\hO_{\fX(w)}^{\rm un}}\hO_{\fX(w)}(1)\lra \cT\otimes\hO_{\fX(w)}
\stackrel{\rm dlog}{\lra}\omega_{\cE/\fX(w)}\otimes_{\hO_{\fX(w)}^{\rm un}}\hO_{\fX(w)}\lra 0.
$$

\begin{lemma}
\label{lemma:locht}
For every connected, small affine object $\cU=\bigl(\Spf(R_\cU), N_\cU\bigr)$
of $\cX(w)^{\rm ket}$, the localization of the Hodge-Tate sequence of sheaves at $\cU$ is the
Hodge-Tate sequence of continuous $\cG_{\cU}$-representations which appears in \cite{andreatta_iovita_stevens}
section \S 2
$$
0\lra \omega_{\cE/\cX(w)}^{-1}(\cU)\otimes_{R_\cU}\hR_\cU(1)\lra T_p\bigl(\cE^\vee_{\cU}\bigr)\otimes\hR_\cU
\stackrel{\rm dlog_\cU}{\lra}\omega_{\cE/\cX(w)}(\cU)\otimes_{R_\cU}\hR_\cU\lra 0
$$

\end{lemma}

\begin{proof} The proof is clear.
\end{proof}

\begin{lemma}
\label{lemma:correctionHT}
Let $\cF^0:={\rm Im}({\rm dlog})$ and $\cF^1:={\rm Ker}({\rm dlog})$. Then

i) $\cF^0, \cF^1$ are locally free sheaves of $\hO_{\fX(w)}$-modules on $\fX(w)$ of rank $1$.
We denote by $\cF^{i,(r)}:=j_r^\ast(\cF^i)$ for $i=0,1$, they are locally free
$\hO_{\fX^{(r)}(w)}$-modules of rank $1$.

ii) We set $v:=w/(p-1)$ and let us suppose that $w$ is adapted to $r$, for a certain
$r\ge 1$, $r\in \N$. We denote $\cC_r\subset \cE[p^r]$ the canonical subgroup
of level $p^r$ of $\cE[p^r]$ over $X^{(r)}(w)$ (which exists by the assumption on $w$),
 denote by $\cC_r^\vee$ its Cartier dual and we also denote by
$\cC_r$ and $\cC_r^\vee$ the groups of points of these group-schemes over $X^{(r)}(w)$, and by the same
symbols the constant abelian sheaves on $\bigl(X^{(r)}(w)\bigr)^{\rm ket}$. We have natural
isomorphisms as $\hO_{\fX(w)}$-modules on $\fX(w)$:
$\cF^0/p^{(1-v)r}\cF^0\cong \cC_r^\vee\otimes\cO_{\fX(w)}/p^{(1-v)r}\cO_{\fX(w)}$ and
$\cF^1/p^{(1-v)r}\cF^1\cong \cC_r\otimes\cO_{\fX(w)}/p^{(1-v)r}\cO_{\fX(w)}$.

iii) we have natural isomorphisms of $\cO_{\cX^{(r)}(w)}$-modules with $G_r$-action:
$$
v_{\fX^{(r)}(w),\ast}\bigl(\cF^{i,(r)}\bigr)\cong \cF_i^{(r)}\otimes_{\cO_K}\cO_{\C_p},
$$
for $i=0,1$. Here $\cF_i^{(r)}$ are the sheaves on
$\cX^{(r)}(w)$ defined in section \S ???? of \cite{andreatta_iovita_stevens}. Moreover
$\cF^{i,(r)}\cong v_{\fX^{(r)}(w)}^\ast(\cF_i^{(r)})\otimes_{\hO_{\fX^{(r)}(w)}^{\rm un}}\hO_{\fX^{(r)}(w)}$.
\end{lemma}

\begin{proof}
We consider a connected, small affine object $\cU=\bigl(\Spf(R_\cU), N_\cU\bigr)$
of $\cX^{\rm ket}$ as in lemma \ref{lemma:locht}. Then the localizations of
$\cF^i$ at $\cU$ are: $\cF^0\bigl(\Rbar_\cU, \Nbar_\cU\bigr)={\rm Im}({\rm dlog}_\cU)=F^0$
and $\cF^1\bigl(\Rbar_\cU,\Nbar_\cU\bigr)={\rm Ker}({\rm dlog}_\cU)=F^1$ and we apply
proposition 2.4 of \cite{andreatta_iovita_stevens}. This proves i) and ii).
Now we apply proposition 2.6 of \cite{andreatta_iovita_stevens} and iii) follows.
\end{proof}

\bigskip
\noindent
Let now $B$, $k$, $r$, $w$ be as at the beginning of  section \ref{sec:modularsheaves}, i.e. $B$ is a complete, regular, local, noetherian
 $\cO_K$-algebra, $k\in \cW(B_K)$ and $r\in \N$ and $w\in \Q$  such that
$w$ is adapted to $r$ and $k$. Let us also recall that we denoted $v:=w/(p-1)$.

Let us denote $S_{\fX^{(r)}(w)}:=\Z_p^\times\bigl(1+p^{(1-v)r}\hO_{\fX^{(r)}(w)}\bigr)$, it is a sheaf of abelian
groups on $\fX^{(r)}(w)$ which acts on $\displaystyle \hO_{\fX^{(r)}(w)}\hat{\otimes}B:=\lim_{\infty \leftarrow n}
\bigl(\hO_{\fX^{(r)}(w)}/\pi^n\hO_{\fX^{(r)}(w)}\otimes_{\cO_K}B/\underline{m}^n\bigr)$ as follows:
let $s=c\cdot x\in S_{\fX^{(r)}(w)}(\cU,W,\alpha)=\Z_p^\times\bigl(1+p^{(1-v)r}\hO_{\fX(w)}(\cU,W)\bigr)$
and $y\in \hO_{\fX^{(r)}(w)}(\cU,W,\alpha)\hat{\otimes}B=\hO_{\fX(w)}(\cU,W)\hat{\otimes}B$. Then we define
$$
s\ast y:=\exp(a\log(x))\cdot c^k\cdot y, \mbox{ where } a\in B_K \mbox{ is such that } t^k=\exp(a\log(t)), t\in 1+p^r\Z_p.
$$
Let us remark that $s\ast y\in \hO_{\fX^{(r)}(w)}(\cU,W,\alpha)\hat{\otimes}B$. We denote by $\bigl(\hO_{\fX^{(r)}(w)}\hat{\otimes}B\bigr)^{(k)}$
the continuous sheaf $\hO_{\fX^{(r)}(w)}\hat{\otimes}B$ with the above defined action of $S_{\fX^{(r)}(w)}$.

Thanks to lemma \ref{lemma:correctionHT} ii) that we have an isomorphism of sheaves
$$
\varphi\colon \cF^{0,(r)}/p^{(1-v)r}\cF^{0,(r)}\cong  \bigl(\cC_r\bigr)^\vee\otimes\hO_{\fX^{(r)}(w)}/p^{(1-v)r}\hO_{\fX^{(r)}(w)}.
$$
Let $\cF^{(r)'}$ denote the inverse image under
the isomorphism $\varphi$ above of
the sheaf of sets $(\cC_r)^\vee-(\cC_r)^\vee[p^{r-1}]$.   It is endowed with an action of $S_{\fX^{(r)}(w)}$.

Recall from \S\ref{sec:hodgetate} that we have define a morphism of sites $j_r\colon \fX(w) \to \fX^{(r)}(w)$. It then follows from the construction that  ${\rm
dlog}$ induces a map
$${\rm dlog}\colon j_r^\ast(\cT_0) \lra \cF^{(r)'},$$compatible with the actions of $\Z_p^\ast$ on the two sides.

\begin{lemma}\label{lemma:F'torsor} The sheaf
$\cF^{(r)'}$ is an $S_{\fX^{(r)}(w)}$-torsor and there exists a covering of $\cX(w)$ by small affine objects $\{U_i\}$ such that
$\cF^{(r)'}\vert_{(U_i,U_i\times_{\cX(w)} X^{(r)}}$ is the trivial torsor for every $i$.
\end{lemma}

\begin{proof}
We localize at a connected, small affine object $\cU$ of $(\cX^{(r)}(w))^{\rm ket}$ and apply lemma \ref{lemma:correctionHT}(iii) and
\cite{andreatta_iovita_stevens} section \S 3.
\end{proof}

Let us now consider the $\cO_{\fX^{(r)}(w)}\hat{\otimes}B$-module
$$
\cM^{(r)}_k(w):={\mathfrak Hom}_{S_{\fX^{(r)}(w)}}\bigl(\cF^{(r)'}, (\hO_{\fX^{(r)}(w)}\hat{\otimes}B)^{(-k)}\bigr).
$$Thanks to lemma \ref{lemma:F'torsor} it is a locally free $\cO_{\fX^{(r)}(w)}\hat{\otimes}B$-module of rank $1$ and we have a natural isomorphism of $\cO_{\fX^{(r)}(w)}\hat{\otimes}B$-modules
$${\mathfrak Hom}_{\hO_{\fX^{(r)}(w)}\hat{\otimes}B}\left(\cM^{(r)}_k(w), \hO_{\fX^{(r)}(w)}\hat{\otimes}B\right) \lra  \cM^{(r)}_{-k}(w).$$
We also have the continuous sheaf of $\cO_{\fX^{(r)}(w)}\hat{\otimes}B$-modules $${\mathbb A}_k^{(r)}:= {\mathfrak Hom}_{\Z_p^\ast}\bigl(j_r^\ast(\cT_0),
(\hO_{\fX^{(r)}(w)}\hat{\otimes}B)^{(-k)}\bigr)$$ and a map of continuous sheaves of $\cO_{\fX^{(r)}(w)}\hat{\otimes}B$-modules $${\rm dlog}^{\vee,k}\colon
\cM_k^{(r)}(w)\lra {\mathbb A}_k^{(r)}$$ induced by ${\rm dlog}$.

\

For every element $\sigma\in G_r$ we denote also by $\sigma$ the functor
$\bigl(E_{\cX(w)_\Kbar}\bigr)_{/(\cX(w),X^{(r)}(w))}\lra \bigl(E_{\cX(w)_\Kbar}\bigr)_{/(\cX(w),X^{(r)}(w))}$
defined on objects by $(\cU,W,\alpha)\rightarrow (\cU,W,\sigma\circ \alpha)$ and by identity on the morphisms.
This functor induces a continuous functor on the site $\fX^{(r)}(w)$. If $\cH$ is a sheaf (or continuous sheaf) on
$\fX^{(r)}(w)$ we denote by $\cH^\sigma$ the sheaf: $\cH^\sigma(\cU,W,\alpha):=\cH\bigl(\sigma(\cU,W,\alpha)\bigr)=
\cH(\cU,W, \sigma\circ\alpha)$.

\begin{lemma}
\label{lemma:sigma}
a) Let us suppose that $\cG$ is a sheaf of abelian groups on $\fX(w)$ and $\cH:=j_r^\ast(\cG)$.
Then $\cH^\sigma=\cH$ for all $\sigma\in G_r$.\smallskip

b) For $\cH= \cF^{(r)'}$, $(\hO_{\fX^{(r)}(w)}\hat{\otimes}B)^{(-k)}$, or $j_r^\ast(\cT_0)$ we have $\bigl(\cH\bigr)^\sigma=\cH$ for every $\sigma\in G_r$. Hence,
the same applies for $\cM_k^{(r)}(w)$ and ${\mathbb A}_k^{(r)}$ compatibly with ${\rm dlog}^{\vee,k}$.\smallskip

c) Suppose that $\cH$ is a sheaf on $\fX^{(r)}(w)$ such that $\cH^\sigma=\cH$ for all $\sigma\in G_r$. Then each element $\sigma\in G_r$ defines a canonical
automorphism of the sheaf $j_{r,\ast}(\cH)$, i.e., we have a canonical action of the group $G_r$ on the sheaf $j_{r,\ast}(\cH)$.
\end{lemma}

\begin{proof}
a) follows immediately as $j_r^\ast(\cG)(\cU,W,\alpha)=\cG(\cU,W)$.

b) As $\cO_{\fX^{(r)}(w)}=j_r^\ast(\cO_{\fX(w)})$ we only need to verify
the property for the sheaf $\cF^{(r)'}$ which is clear from its definition.

c) Let us recall that $j_{r,\ast}(\cH )(\cU,W):=\cH\bigl(\cU,W\times_{X(w)_\Kbar}X^{(r)}(w)_\Kbar, {\rm pr}_1 \bigr)$.
We define the automorphism
$$
\sigma\colon j_{r,\ast}(\cH)(\cU,W)=\cH\bigl(\cU,W\times_{X(w)_\Kbar}X^{(r)}(w)_\Kbar,{\rm pr}_1\bigr) \lra
$$
$$
\lra j_{r,\ast}(\cH)(\cU,W)=
\cH\bigl(\cU,W\times_{X(w)_\Kbar}X^{(r)}(w)_\Kbar, \sigma^{-1}\circ{\rm pr}_1\bigr)
$$
by the fact that $\sigma\colon X^{(r)}(w)\lra X^{(r)}(w)$ is an automorphism over $X(w)$.
\end{proof}

\begin{definition}\label{def:omegak}
We define the sheaves $\Omega^k_{\fX(w)}$ and $\omega^{\dagger,k}_{\fX(w)}$ on $\fX(w)$ by

$$
\Omega^k_{\fX(w)}:=\Bigl(j_{r,\ast}\bigl({\mathfrak Hom}_{S_{\fX^{(r)}(w)}}(\cF^{(r)'}, (\hO_{\fX^{(r)}(w)}\hat{\otimes}B)^{(-k)})\bigr)\Bigr)^{G_r}
$$
and
$$
\omega^{\dagger,k}_{\fX(w)}:=\Bigl(j_{r,\ast}\bigl({\mathfrak Hom}_{S_{\fX^{(r)}(w)}}(\cF^{(r)'}, (\hO_{\fX^{(r)}(w)}\hat{\otimes}B)^{(-k)})\bigr)[1/p]\Bigr)^{G_r}.
$$
\end{definition}

The sheaves thus defined enjoy the following properties.

\begin{lemma}\label{lemma:omegaklocallyfree} For every $B$, $k$ and $w$ as above we have

i)  $\omega^{\dagger,k}_{\fX(w)}$ is a locally free $(\hO_{\fX(w)}\hat{\otimes}B)[1/p]$-module of rank $1$.

ii) $v_{\fX(w),\ast}\bigl(\omega_\fX^{\dagger,k}\bigr)\cong \omega^{\dagger,k}_w\otimes_K\C_p$, where $\omega^{\dagger,k}_w$ is the sheaf on $X(w)$ given in
definition 3.2 of \cite{andreatta_iovita_stevens}.

iii) $\omega^{\dagger,k}_{\fX(w)}\cong \omega^{\dagger,k}_w\hat{\otimes}_{\hO_{\cX(w)}} \hO_{\fX(w)} $.

\end{lemma}

\begin{proof} i) is a consequence of the fact that $\cF^{(r)'}$ is a locally trivial $S_{\fX^{(r)}(w)}$-torsor
and ii) and iii) follow  by localization at a connected small affine $\cU$ of $\cX(w)^{\rm ket}$ using lemma \ref{lemma:F'torsor}.
\end{proof}

As mentioned at the beginning of this section we shall be most interested in two instances of these constructions corresponding to choices of
pairs $(B,k)$ as above.

1) The first is the simplest, i.e. when $B=\cO_K$ and so an $B_K=K$-valued weight is simply
an element $k\in \cW(K)$.

2) The second instance appears as follows.
Let $U\subset \cW^\ast$ be a wide open disk. Let us
recall that $\Lambda_U$ is a $\cO_K$-algebra of bounded rigid functions on $U$ and
let us denote by  $|| \ ||$ the norm on the $K$-Banach algebra $\Lambda_{U,K}$.

We denote by $k_U\colon \Z_p^\times\lra \Lambda_U^\times$ the universal character of $U$ defined by the relation
$k_U(t)(x)=t^x$ for $t\in \Z_p^\times, x\in U$.

The constructions for the two instances above are connected as follows. Let $U, \Lambda_U, k_U$ be as at 2) above.
Let also $k\in U(K)$ be a $K$-valued weight
and $t_k$ a uniformizer at $k$, i.e. an element of $\Lambda_U$ which vanishes of order $1$ at $k$ and nowhere else.
Then we have an exact sequence of $K$-algebras
$$
0\lra \Lambda_U\stackrel{t_k}{\lra}\Lambda_U\lra \cO_K\lra 0
$$
which induces an exact sequence of sheaves on $\fX(w)$:
$$
0\lra \omega^{\dagger,k_U}_{\fX(w)}\stackrel{t_k}{\lra}\omega^{\dagger,k_U}_{\fX(w)}\lra \omega^{\dagger,k}_{\fX(w)}\lra 0
$$
which will be called the specialization exact sequence.

\subsection{The map ${\rm dlog}^{\vee,k}$.}
\label{sec:dkappa}

We start by fixing a triple $B$, $k$ as in the previous section such that
the ssociated $r=1$ and let $w$ be adapted to $k$.   We have explained in section \S\ref{sec:geometric} how to construct a continuous sheaf
$\cA^o_k(w)=\bigl(\cA_{k,n}^o(w)\bigr)_{n\in\N}:=\bigl(\nu^\ast(\cA_{k,n}^o)\bigr)_{n\in\N}$ on Faltings' site $\fX(w)$ associated to  the continuous representation
of  $A_{k}^o=\bigl( A_{k}^o/\underline{m}^n A_{k}^o \bigr)_{n\in\N}$ (see definition \ref{def:A_k}) of the Kummer \'etale fundamental group $\cG$ of $X(N,p)$.
Similarly we have the sheaves $\cD^o_k(w)=\bigl(\cD_{k,n}^o(w)\bigr)_{n\in\N}:=\bigl(\nu^\ast(\cD_{k,n}^o)\bigr)_{n\in\N}$. By construction and proposition
\ref{prop:finitefil}, the sheaf $\cD_{k,n}^o(w)$ is a quotient of $\mathfrak {Hom}_{B}\bigl(\cA_{k,n}^o(w), B/\underline{m}^m \bigr)$.

Write $\cT_0$ as the continuous sheaf on $\fX(w)$ obtained similarly from the $\cG$-representation $T_0$. Then we have an inclusion of sheaves
$$\cA_{k,n}^o(w)\subset \mathfrak{Hom}_{\Z_p^\ast}\bigl(\cT_0, \bigl(B/\underline{m}^n\bigr)^{(k)}\bigr)$$on $\fX(w)$, which for every $r$ and $n\in \N$ provides a
map of sheaves of $\cO_{\fX^{(r)}(w)}\otimes B/\underline{m}^n$-modules $$\beta_n^{(r)}\colon j_r^\ast\bigl(\cA_{k,n}^o(w)\bigr)\otimes_{\cO_K}
\bigl(\cO_{\fX^{(r)}(w)}/p^n \cO_{\fX^{(r)}(w)}\bigr) \lra \mathfrak{Hom}_{\Z_p^\ast}\bigl(j_r^\ast\bigl(\cT_{0}\bigr), (\cO_{\fX^{(r)}(w)} \otimes
B/\underline{m}^n)^{(-k)}\bigr).$$These maps are compatible for varying $n$  and define a map of continuous sheaves $$\beta^{(r)}\colon
j_r^\ast\bigl(\cA_{k}^o(w)\bigr)\widehat{\otimes}_{\cO_K} \hO_{\fX^{(r)}(w)} \lra {\mathbb A}_k^{(r)}.$$

\begin{proposition}\label{prop:factorization}  (1) The map $\beta^{(r)}$ is injective  and $G_r$-invariant.

(2) The map ${\rm dlog}^{\vee,k}$ is $G_r$-invariant and factors via $\beta^{(r)}$.
\end{proposition}
\begin{proof}
The fact that $\beta^{(r)}$ is $G_r$-invariant is clear as it is defined already over $\fX(w)$. The $G_r$-invariance of ${\rm dlog}^{\vee,k}$ follows from the
$G_r$-invariance of ${\rm dlog}$ which is clear from its definition. We prove that $\beta_n^{(r)}$ is injective for every $n\in \N$ and that the map $${\mathfrak
Hom}_{S_{\fX^{(r)}(w)}}\bigl(\cF^{(r)'}, (\cO_{\fX^{(r)}(w)}\otimes B/\underline{m}^n)^{(-k)}\bigr)\lra {\mathfrak Hom}_{\Z_p^\ast}\bigl(j_r^\ast(\cT_0),
(\cO_{\fX^{(r)}(w)}\otimes B/\underline{m}^n)^{(-k)}\bigr)$$ factors via $ \beta_n^{(r)}$. It suffices to show this on localizations after localizing at small
affine objects of $\cX(w)^{\rm ket}$ covering $\cX(w)$; see \S\ref{sec:localization}. Let  $\cU=(\Spf(R_\cU,N_\cU)$ be any small affine.  \smallskip

{\it The localization of $\cA_{k}^o(w)$:} Form now until the end of this secrion we set $r=1$.
Using the notation of section \S\ref{sec:geometric} choose a $\Z_p$-basis $\{\epsilon_0,\epsilon_1\}$ of $T$ such that
$\langle \epsilon_0, \epsilon_1\rangle=1$ and $\epsilon_1(\mbox{ mod }p T)\in \cE_\eta[p]$ belongs to the canonical subgroup $C=C_1$ of level $p$. Then
$T_0:=\{a\epsilon_0+b\epsilon_1 \vert  a\in \Z_p^\ast, b\in \Z_p\}$. Let $x$ be the $\Z_p$-dual of $\epsilon_0$ and $y$ the $\Z_p$-dual of $\epsilon_1$. We deduce
from the discussion after definition \ref{def:A_k} that $$j_1^\ast\bigl( \cA_{k,n}^o(w)\bigr)(\Rbar_\cU,\Nbar_\cU,g):=\oplus_{h\in \N} (B/\underline{m}^n)  x^{k}
\bigl(y/x\bigr)^h. $$Thus, if we let $D:=\bigl(\cO_{\fX^{(1)}(w)}\otimes B/\underline{m}^n\bigr)(\Rbar_\cU,\Nbar_\cU,g)$, then
$$\bigl(j_1^\ast\bigl(\cA_{k,n}^o(w)\bigr)\otimes_{\cO_K} \cO_{\fX^{(1)}(w)} \bigr)(\Rbar_\cU,\Nbar_\cU,g)=\oplus_{h\in \N} D \cdot x^{k} \bigl(y/x\bigr)^h
.$$Similarly $\mathfrak{Hom}_{\Z_p^\ast}\bigl(j_1^\ast\bigl(\cT_0\bigr), (\hO_{\fX^{(1)}(w)}\hat{\otimes}B/\underline{m}^n)^{(-k)}\bigr)(\Rbar_\cU,\Nbar_\cU,g)$ is
the $D$-module of continuous maps  ${\rm Hom}_{\Z_p^\ast}\bigl(\Z_p^\ast \epsilon_0+ \Z_p \epsilon_1, D^{(-k)}\bigr)$. To an element $f(x,y):=\sum \alpha_h x^k
(y/x)^h$ we associate the function  $\Z_p^\ast \epsilon_0+ \Z_p \epsilon_1 \to D$ sending $a \epsilon_0 + b \epsilon_1\mapsto f(a,b)=\sum \alpha_h k(a) (b/a)^h$. If
such function is zero then $f(x,y)$ is zero, proving the first claim.\smallskip

{\it The localization of $\cM_{k}^{(1)}(w)$:}  Using the notation of lemma \ref{lemma:locht} we let $e_0$, $e_1$ denote an $\hR_\cU$-basis of $T\otimes \hR_\cU$
such that $e_1$ is a basis of $F^1$ over $\hR_\cU$  reducing to $\epsilon_1$ modulo $p^{1-v}$ and ${\rm dlog}_\cU(e_0)$ is a basis of $F^0$ over $\hR_\cU$
reducing to $\epsilon_0$ modulo $p^{1-v}$. Let $X$, $Y$ denote the basis of $T\otimes \hR_\cU$ which is $\hR_\cU$-dual to $e_0$, $e_1$ respectively (i.e.
$X(e_1)=Y(e_0)=0$ and $X(e_0)=Y(e_1)=1$).  Then,  $$\mathfrak{Hom}_{S_{\fX^{(1)}(w)}}\bigl(\cF^{(1)'}, (\cO_{\fX^{(1)}(w)} \otimes B/\underline{m}^n)^{(-k)}\bigr)
(\Rbar_\cU,\Nbar_\cU,g)=  D \cdot X^{k}.$$As $X= ux + v y$ with $u\in \hR_{\cU}$ congruent to $1$ modulo $p^{1-v} \hR_{\cU}$ and $v\in \hR_{\cU}$ congruent to
$0$ modulo $p^{1-v} \hR_{\cU}$, it follows that $X^k=x^k \gamma $ with $\gamma\in  1+ p^{1-v}\hR_\cU\langle (y/x) \rangle$. Here, $\hR_\cU\langle y/x \rangle$
denotes the $p$-adically convergent power series in the variable $y/x$. The  second claim follows.

\end{proof}

In particular,  ${\rm dlog}^{\vee,k}$ induces a $G_1$-invariant morphism of $\hO_{\fX^{(1)}(w)}\hat{\otimes}B $-modules $$\cM^{(1)}_k(w)\lra
j_1^\ast\bigl(\cA_{k}^o(w)\bigr)\widehat{\otimes}_{\cO_K} \hO_{\fX^{(1)}(w)}.$$Taking $\mathfrak{Hom}_{\cO_{\fX^{(1)}(w)} \hat{\otimes} B} \bigl( \ -\ ,
\cO_{\fX^{(1)}(w)} \hat{\otimes} B\bigr)$ and using the identification $$\mathfrak{Hom}_{\cO_{\fX^{(1)}(w)} \hat{\otimes} B} \bigl( \cM^{(1)}_k(w),
\cO_{\fX^{(1)}(w)} \hat{\otimes} B\bigr)\cong \cM^{(1)}_{-k}(w),$$we get an induced $G_1$-invariant morphism $\widehat{B}$-modules
$$\delta\colon \mathfrak{Hom}_{\widehat{B}} \bigl(j_1^\ast\bigl(\cA_{k}^o(w)\bigr) ,
\widehat{B}\bigr)   \lra \cM^{(1)}_{-k}(w).$$Then,

\begin{lemma}\label{lemma:continuousfactorization} For every $n\in\N$ there exists $m\geq n$
such that the map $$\mathfrak{Hom}_{B}\bigl(j_1^\ast\bigl(\cA_{k,m}^o(w)\bigr), B/\underline{m}^m \bigr) \lra {\mathfrak Hom}_{\cO_{\fX^{(1)}(w)}\otimes
B}\bigl(\cM_{k}^{(1)}(w), (\cO_{\fX^{(1)}(w)}\otimes B/\underline{m}^n)^{(-k)}\bigr),$$induced by $\delta$, factors via $j_1^\ast\bigl(\cD_{k,m}^o(w)\bigr)$.

\end{lemma}

\begin{proof}
As $\cX(w)^{\rm ket}$ can be covered by finitely many small affines, it suffices to show the claim  on localizations  at a small affine $\cU=(\Spf(R_\cU,N_\cU)$ of
$\cX(w)^{\rm ket}$ (see section \S\ref{sec:localization}). We use the notation of the proof of proposition \ref{prop:factorization}. Thanks to proposition
\ref{prop:finitefil} the quotient map
$$
{\mathfrak Hom}_{B}\bigl(j_1^\ast\bigl(\cA_{k,m}^o(w)\bigr), B/\underline{m}^m \bigr)(\Rbar_\cU,\Nbar_\cU,g)=\oplus_{h\in \N} (B/\underline{m}^m) \cdot \bigl(x^k
(y/x)^h\bigr)^\vee \lra j_r^\ast\bigl(\cD_{k,m}^o(w)\bigr)(\Rbar_\cU,\Nbar_\cU,g) $$identifies the latter with the $B$-module $\oplus_{0\leq h\leq m}
(B/\underline{m}^{m-h}) \cdot \bigl(x^k (y/x)^h\bigr)^\vee$.   Recall from the proof of proposition \ref{prop:finitefil} that, setting $D:=\bigl(\cO_{\fX(w)}\otimes
B/\underline{m}^n\bigr)(\Rbar_\cU,\Nbar_\cU)$, we have defined a generator $X^k$ of $\cM^{(r)}_{-k}(w)(\Rbar_\cU,\Nbar_\cU,g)$ as $D$-module. We conclude that
$${\mathfrak Hom}_{\cO_{\fX^{(1)}(w)}\otimes B}\left(\cM_{k}^{(1)}(w), (\cO_{\fX^{(1)}(w)}\otimes B/\underline{m}^n)^{(-k)}\bigr)\right)(\Rbar_\cU,\Nbar_\cU,g)\cong D
\cdot (X^k)^\vee.$$Let $N(n)$ be the degree of  $X^k$ in $\oplus_{h\in \N} D \cdot x^{k} \bigl(y/x\bigr)^h$, i.e., the maximal $N$ such that the coordinate of $X^k$
with respect to $x^{k} \bigl(y/x\bigr)^N$ is non zero. If we take $m$ so that $m-N\geq n$, then the map $\delta$ localized at $\cU$ factors via $\mathfrak{Hom}_{\cO_{\fX^{(1)}(w)}\otimes B}\bigl(\cM_{k}^{(1)}(w), (\cO_{\fX^{(1)}(w)}\otimes B/\underline{m}^n)^{(-k)}\bigr)(\Rbar_\cU,\Nbar_\cU,g)$ as wanted.
\end{proof}

It follows from the lemma \ref{lemma:continuousfactorization} that we get a map of continuous sheaves of $\hO_{\fX^{(r)}(w)}\hat{\otimes}  B$-modules $$\cD_{k}^o(w)
\lra \left(j_{1,\ast} \bigl(j_1^\ast\bigl(\cD_{k,m}^o(w)\bigr)\bigr)\right)^{G_1} \lra
\left(j_{1,\ast}\bigl(\cM^{(1)}_{-k}(w)\bigr)\right)^{G_1}=\Omega^{\dagger,k}_{\fX(w)}.$$Passing to ind-sheaves and using lemma \ref{lemma:omegaklocallyfree} we
obtain a map
\begin{equation}\label{deltakw} \delta^\vee_k(w)\colon \nu^\ast(\cD_k)=\cD_{k}^o(w)[1/p] \lra \omega^{\dagger,k}_{\fX(w)}\cong
\omega^{\dagger,k}_w\hat{\otimes}_{\hO_{\cX(w)}} \hO_{\fX(w)}.\end{equation}

In the next section we will calculate the cohomology of the ind-continuous sheaves $\omega^{\dagger,k}_{\fX(w)}\cong
\omega^{\dagger,k}_w\hat{\otimes}_{\hO_{\cX(w)}} \hO_{\fX(w)}$.

\subsection{The cohomology of the sheaves $\omega^{\dagger,k}_{\fX(w)}$}
\label{sec:cohomegadaggerk}

Let $\iota\colon Z\lra X(w)$ be a morphism in $\cX(w)_\Kbar^{\rm fket}$. Let $\fZ:=\fX(w)_{/(\cX(w),Z)}$ the associated induced site and
$j:=j_{(\cX(w),Z)}\colon\fX(w)\lra \fZ$ the map $j(\cU, W):=\bigl(\cU, Z\times_{X(w)}W, {\rm pr}_1\bigr)$; see \ref{sec:fy}. It induces a morphism of topoi. For
$i\ge 0$ we shall calculate $\rH^i\bigl(\cZ, j^\ast\big(\omega^{\dagger,k}_{\fX(w)}\bigr)\bigr)$. For $\iota={\rm id}$ we get in particular the calculation of
$\rH^{i}\bigl(\fX(w), \omega^{\dagger,k}_{\fX(w)}\bigr)$. We will need the following:

\begin{lemma}\label{computRcont} Let $\cF$ be a locally free $(\hO_{\fX(w)}\hat{\otimes}B)[1/p]$-module of finite rank.
The sheaf $R^b v_{\fX(w),\ast}\bigl(\cF\bigr)$ is the sheaf associated to the presheaf on $\cX(w)^{\rm ket}$:
$$
\cU=(\Spf(R_\cU),N_\cU)\rightarrow \rH^b\bigl(\cG_\cU,\cF(\Rbar_\cU,\Nbar_\cU)\bigr),
$$
where $\cG_\cU$ is the Kummer-\'etale geometric fundamental group of $\cU$, for a choice of
a geometric generic point, i.e. $\cG_\cU={\rm Gal}\bigl(\Rbar_\cU[1/p]/(R_\cU\Kbar)\bigr)$.
\end{lemma}
\begin{proof} The lemma follows arguing as in  \cite[Prop. 2.12 and Lemma 2.24]{andreatta_iovita3}.

\end{proof}

\begin{theorem}
\label{thm:cohomologyomega} We have isomorphisms as $G_K$-modules\smallskip

a) $\rH^0\bigl(\fZ, j^\ast\big(\omega^{\dagger,k}_{\fX(w)}(1)\big)\bigr)\cong \rH^0\bigl(Z,
\iota^\ast\bigl(\omega^{\dagger,k}_w\bigr)\bigr)\hat{\otimes}_K\C_p(1)$;\smallskip

b) $\rH^1\bigl(\fZ, j^\ast\big(\omega^{\dagger,k}_{\fX(w)}(1)\big)\bigr)\cong \rH^0\bigl(Z,
\iota^\ast\bigl(\omega^{\dagger,k+2}_w\bigr)\bigr)\hat{\otimes}_K\C_p$;\smallskip

c) $\rH^i\bigl(\fZ, j^\ast\big(\omega^{\dagger,k}_{\fX(w)}(1)\big)\bigr)=0$ for $i\ge 2$.

\end{theorem}

\begin{proof}
As $R^i j_\ast=0$ for all $i\ge 1$ by \ref{cor:jacyclic} we have $\rH^i\bigl(\fZ, j^\ast\big(\omega^{\dagger,k}_{\fX(w)}(1)\big)\bigr)\cong \rH^1\bigl(\fX(w),
j_\ast\big(j^\ast(\omega^{\dagger,k}_{\fX(w)}(1))\big)\bigr)$. Set $\cF:=j_\ast\bigl(j^\ast\big(\omega^{\dagger,k}_{\fX(w)}(1)\big)\big)$. Recall that
$\omega^{\dagger,k}_{\fX(w)}$ is isomorphic to $\omega^{\dagger,k}_w\hat{\otimes}_{\hO_{\cX(w)}} \hO_{\fX(w)}$ by \ref{lemma:omegaklocallyfree}. Thus,
$j^\ast\big(\omega^{\dagger,k}_{\fX(w)}(1)\big)\cong \omega^{\dagger,k}_w\hat{\otimes}_{\hO_{\cX(w)}} \hO_{\fZ}(1)$ as $j^\ast\bigl(\hO_{\fX(w)} \bigr) \cong
\hO_{\fZ}$ and $\cF \cong \omega^{\dagger,k}_w\hat{\otimes}_{\hO_{\cX(w)}} j_\ast\bigl(\hO_{\fZ}\bigr)(1)$. Due to \ref{prop:left=right} the natural map
$\hO_{\fX(w)}\hat{\otimes}_{\hO_{\cX(w)}} \iota_\ast(\hO_{Z}) \lra j_\ast\bigl(\hO_{\fZ}\bigr)\bigl[p^{-1}\bigr]$ is an isomorphism. Hence, $\cF\cong
\omega^{\dagger,k}_w\hat{\otimes}_{\hO_{\cX(w)}} \iota_\ast(\hO_{\cZ})(1)$ is a locally free $(\hO_{\fX(w)}\hat{\otimes}B)[1/p]$-module.

To prove the theorem we will first calculate the sheaves $R^bv_{\fX(w),\ast}\bigl(\cF\bigr)$ using lemma \ref{computRcont} and then we'll use the Leray spectral
sequence (see \cite{andreatta_iovita3}):
$$ \rH^a\Bigl(\cX(w)^{\rm ket}, R^b v_{\fX(w),\ast}\bigl(\cF \bigr)\Bigr) \Longrightarrow \rH^{i}\bigl(\fX(w), \cF\bigr).
$$We compute the Galois cohomology of the localization of the ind-continuous sheaf $\cF$
$$
\cF(\Rbar_\cU,\Nbar_\cU)=\omega^{\dagger,k}_w(\cU)\hat{\otimes}_{R_{\cU,K}} \iota_\ast(\cO_Z)(\cU_K)\hat{\otimes}_{R_{\cU}} \widehat{\Rbar}_\cU(1).
$$for $\cU=(\Spf(R_\cU,N_\cU)$ a small affine open of $\cX(w)^{\rm ket}$. Using that
$\omega^{\dagger,k}_{\fX(w)}$ is a locally free $(\hO_{\fX(w)}\hat{\otimes}A)[1/p]$-module of rank $1$, it follows from the main result of \cite{faltingsmodular}
that $\rH^0\bigl(\cG_\cU, \hR_{\cU,K}\bigr)=R_{\cU,K} \hat{\otimes}_K\C_p$ so that
$$
\rH^0\bigl(\cG_\cU, \cF(\Rbar_\cU,\Nbar_\cU)\bigr)=\omega^{\dagger,k}_w(\cU)\hat{\otimes}_{R_{\cU,K}} \iota_\ast(\cO_Z)(\cU_K) \hat{\otimes}_K\C_p(1).
$$
Moreover $\rH^1\bigl(\cG_\cU, \hR_{\cU,K}\bigr)\cong \Omega^1_{\cU_K/K}\hat{\otimes}_K\C_p(-1)$ so that
$$
\rH^1\bigl(\cG_\cU, \cF(\Rbar_\cU,\Nbar_\cU)\bigr)\cong \bigl(\Omega^1_{\cU_K/K}\hat{\otimes}_K\C_p(-1)\bigr)\hat{\otimes}_{R_{\cU,K}} \iota_\ast(\cO_Z)(\cU_K)
\hat{\otimes}_K\omega^{\dagger,k}_w(\cU)(1).
$$
The Kodaira-Spencer isomorphism gives $ \Omega^1_{\cU_K/K}\hat{\otimes}_K B_K\cong \omega^{\otimes 2}_{\cE_{\cU_K}/\cU_K}\hat{\otimes}_K
B_K\cong\omega^{\dagger,2}_w(\cU)$. Therefore we obtain
$$
\rH^1\bigl(\cG_\cU, \cF(\Rbar_\cU,\Nbar_\cU)\bigr)\cong \omega^{\dagger, k+2}_w(\cU)\hat{\otimes}_{R_{\cU,K}} \iota_\ast(\cO_Z)(\cU_K) \hat{\otimes}_K\C_p.
$$
Finally $\rH^i\bigl(\cG_\cU, \cF(\Rbar_\cU,\Nbar_\cU)\bigr)=0$ for $i\ge 2$ because $\cG_\cU$ has cohomological dimension $1$. It follows that we have
$$
R^0v_{\fX(w),\ast}\cF \cong \omega^{\dagger,k}_w\hat{\otimes}_{R_{\cU,K}} \iota_\ast(\cO_Z)(\cU_K) \hat{\otimes}_K\C_p(1),
$$
where the isomorphism is as sheaves on $\cX(w)^{\rm ket}$. Similarly we have
$$
R^1v_{\fX(w),\ast}\cF=\omega^{\dagger,k+2}_w\hat{\otimes}_{R_{\cU,K}} \iota_\ast(\cO_Z)(\cU_K) \hat{\otimes}_K\C_p,
$$
and $R^bv_{\fX(w),\ast}\cF=0$ for $b\ge 2$. Now let us observe that $\omega^{\dagger,k}\hat{\otimes}_{R_{\cU,K}} \iota_\ast(\cO_Z)(\cU_K) \hat{\otimes}_K\C_p(1)$ is
a sheaf of $K$-Banach modules on $X(w)$, as it is locally isomorphic to $\iota_\ast(\cO_Z) \hat{\otimes}B_K\hat{\otimes}\C_p$. As $X(w)$ is an affinoid we obtain
that
$$
\rH^1\bigl(\cX(w)^{\rm ket}, \omega^{\dagger,k}_w\hat{\otimes}_{\cO_{\cX(w)}} \iota_\ast(\cO_Z) \hat{\otimes}_K\C_p(1)\bigr)=\rH^1(X(w),
\omega^{\dagger,k}_w\hat{\otimes}_{\cO_{\cX(w)}} \iota_\ast(\cO_Z)\hat{\otimes} \C_p(1)\bigr)=0,
$$
by the main result of the Appendix of \cite{andreatta_iovita_pilloni}.
Therefore the Leray spectral sequence gives now the result of the theorem.
\end{proof}

\section{Hecke Operators}
\label{sec:heckeoperators}

Let $\ell$ denote a prime integer and $w\in \Q$ be such that $0\le w<p/(p+1)$. We assume that $w$ is adapted to some integer $r\ge 1$ (see the beginning of section
\S 3) We denote (see section \S 3.1.1 of \cite{andreatta_iovita_stevens}) by $X^{(r)}_\ell(w)$ the rigid analytic space over $K$ which represents the functor
associating to a $K$-rigid space $S$ a quadruple $(\cE/S, \psi_S,H,Y)$, where $\cE\lra S$ is a semiabelian scheme of relative dimension $1$ and $Y$ is a global
section of $\omega_{\cE/S}^{1-p}$ such that $Yh(\cE/S)=p^w$, where we have denoted by $h$ a lift to characteristic $0$ of the Hasse-invariant. Let us notice that
the existence of $Y$ as above implies that there is a canonical subgroup $C_r\subset \cE[p^r]$, of order $p^r$ defined over $S$. We continue to describe the
quadruple $(\cE/S,\psi_S, H,Y)$: $\psi_S$ is a $\Gamma_1(Np^r)$-level structure of $\cE/S$, more precisely $\psi_S=\psi_N\cdot\psi_{p^r}$, where $\psi_N$ is a
$\Gamma_1(N)$-level structure of $\cE/S$ and $\psi_{p^r}$ is a generator of $C_r$. Furthermore $H\subset \cE$ is  locally free subgroup scheme, finite of order
$\ell$ defining a $\Gamma_0(\ell)$-level structure such that $H\cap C_r=\{0\}$ (this condition is automatic if $\ell\ne p$). We consider on $X^{(r)}_\ell(w)$ the
log structure given by the divisor of cusps and denote the resulting log rigid space by the same notation: $X^{(r)}_\ell(w)$.

We have natural morphisms $p_1\colon X^{(r)}_\ell(w)\lra X^{(r)}(w)$ and $p_2\colon X^{(r)}_\ell(w)\lra X^{(r)}(w')$, where $w'=w$ if $\ell\ne p$ and $w'=p^rw$ if
$\ell=p$, in which case we will assume that $0\le w\le 2/p^(2r)$. These morphisms are defined on points as follows: $p_1(\cE, \psi,H,Y):=(\cE, \psi,Y)\in
X^{(r)}(w)$ and $p_2(\cE, \psi, H,Y)=\bigl(\cE/H, \psi', Y'\bigr)\in X^{(r)}(w')$ where $\psi',Y'$ are the induced level structure and global section associated to
$\cE/H$. The morphism $p_1$ is finite and Kummer log \'etale and if $\ell=p$ then $p_2$ is an isomorphism of $K$-rigid spaces.

Let us recall that we have denoted $\fX(w)$ Faltings' site associated to the log formal scheme $\cX(w)$ and with $\fX^{(r)}(w)$ the site $\fX(w)$ localized at its
object $\bigl(\cX(w),X^{(r)}(w)\bigr)$. Let us observe that $\bigl(\cX(w), X^{(r)}_\ell(w)\bigr)$ is also an objects of $\fX(w)$ therefore we will denote by
$\fX^{(r)}_\ell(w):=\fX(w)_{/(\cX(w),X^{(r)}_\ell(w))}$, i.e., the localized site (see sections 2.3 and 2.4).

The morphisms $p_1$, $p_2$ defined above induce continuous morphisms of sites:
$$
\begin{array}{cccccccccc}
&&\fX^{(r)}_\ell(w)\\
&\nearrow p_1&&p_2\nwarrow\\
\fX^{(r)}(w)&&&&\fX^{(r)}(w')
\end{array}
$$
We denote by $\cE^{(r)}_w$ the universal generalised elliptic curve  over
 $X^{(r)}(w)$ and by  $\pi_\ell\colon \cE\lra \cE/H$  the natural universal isogeny over
$X^{(r)}_\ell(w)$. Let
$\cT(\cE)$, $\cT(\cE/H)$, $\cT(\cE^{(r)}_w)$ denote the $p$-adic Tate modules of $\cE, \cE/H,\cE^{(r)}_w$ seen as
continuous sheaves on $\fX^{(r)}_\ell(w)$ and $\fX^{(r)}(w)$ respectively.
Then we have maps
$$
p_2^\ast\bigl(\cT(\cE^{(r)}_{w'})\bigr)=\cT\bigl((\cE/H)\bigr)\stackrel{\pi_\ell}{\longleftarrow}
p_1^\ast\bigl(\cT(\cE^{(r)}_w)\bigr)=\cT\bigl(\cE\bigr).
$$
which induce the following commutative diagram
$$
\begin{array}{cccccccccc}
p_2^\ast \Bigl(\cT((\cE^{(r)}_{w'})^\vee)\otimes\hO_{\fX^{(r)}(w')}\Bigr)&\cong &\cT\bigl((\cE/H)^\vee\bigr)
\otimes\hO_{\fX^{(r)}_\ell(w)}&
\stackrel{\rm dlog}{\lra}&\omega_{\cE/H}\otimes\hO_{\fX^{(r)}_\ell(w)}\\
&&\downarrow \pi^\vee_\ell\otimes {\rm Id}&&\downarrow d(\pi_\ell)\otimes {\rm Id}\\
p_1^\ast\Bigl(\cT\bigl((\cE^{(r)}_w)^\vee\bigr)\otimes\hO_{\fX^{(r)}_\ell(w)}\Bigr)&\cong&\cT\bigl(\cE^\vee\bigr)
\otimes\hO_{\fX^{(r)}_\ell(w)}&
\stackrel{\rm dlog}{\lra}&\omega_{\cE}\otimes\hO_{\fX^{(r)}_\ell(w)}
\end{array}
$$

\bigskip
\noindent
Until the rest of this section we suppose $r=1$.

\begin{lemma}
\label{lemma:piell}

Let now $k\in \cW^\ast(B_K)$ be a weight with associated integer $r=1$ such that $w$ as above is associated to $k$. Then the above diagram induces morphisms:\smallskip
$$\pi_\ell\colon p_2^\ast\bigl(\cM^{(1)}_{-k}(w')\bigr)\lra p_1^\ast\bigl(\cM^{(1)}_{-k}(w)\bigr),
\quad \pi_\ell\colon p_2^\ast\bigl(j_1^\ast(\cD_k^o(w'))\bigr)\lra p_1^\ast\bigl(j_1^\ast(\cD_k^o(w'))\bigr),$$

\noindent where $\cM^{(1)}_{-k}(w)$ is the sheaf $\mathfrak{Hom}_{S_{\fX^{1r)}(w)}}\bigl(\cF^{(1)'}, (\hO_{\fX^{(1)}(w)}\hat{\otimes}B)^{(-k)}\bigr)$ on
$\fX^{(1)}(w)$ defined in section \S\ref{sec:omegakappa}, such that the diagram

$$
\begin{array}{ccccc}
p_2^\ast\bigl(j_1^\ast(\cD_k^o(w'))\bigr) & \stackrel{p_2^\ast(\delta)}{\lra} & p_2^\ast\bigl(\cM^{(1)}_{-k}(w')\bigr) \\
\downarrow\pi_\ell&&\downarrow\pi_\ell \\
p_1^\ast\bigl(j_1^\ast(\cD_k^o(w'))\bigr) & \stackrel{p_1^\ast(\delta)}{\lra}& p_1^\ast
\bigl(\cM^{(1)}_{-k}(w)\bigr),
\end{array}
$$where $\delta$ is the map defined in lemma \ref{lemma:continuousfactorization}, is commutative.
\end{lemma}

\begin{proof}
Let $\cF^0(\cE):={\rm Im}\Bigl({\rm dlog}\colon \cT(\cE^\vee)\otimes\hO_{\fX^{(1)}_\ell(w)}\lra\omega_{\cE/\fX^{(1)}_\ell(w)}\Bigr)$ as in lemma
\ref{lemma:correctionHT} and we denote $\cF^0(\cE/H)$ the analogue object constructed with $\cE/H$ instead of $\cE$. Then, if $C_1$ is the canonical subgroup of
$\cE$ and $C'_1$ is the canonical subgroup of $\cE/H$ we have a natural commutative diagram

$$
\begin{array}{ccccc}
\cE[p]&\stackrel{\pi_\ell}{\lra}&(\cE/H)[p]\\
\cup&&\cup\\
C_1&\cong&C'_1
\end{array}
$$
where the map on canonical subgroups is an isomorphism. It follows that the dual isogeny $\pi_\ell^\vee\colon (\cE/H)^\vee \to \cE^\vee$ induces an isomorphism of
torsors $\cF'\bigl((\cE/H)^\vee\bigr) \lra \cF'(\cE^\vee)$ and, hence, an isomorphism $$\pi_\ell^\vee\colon p_1^\ast\bigl(\cM^{(1)}_{k}(w)\bigr)\lra
p_2^\ast\bigl(\cM^{(1)}_{k}(w')\bigr).$$Dualizing with respect to $\hO_{\fX^{(1)}(w)}\hat{\otimes}B$ and identifying the dual of $\cM^{(1)}_k(w)$ with
$\cM^{(1)}_{-k}(w)$   we get the first map $\pi_\ell$, which is an simorphism.

The map  $\pi_\ell^\vee\colon \cT\bigl((\cE/H)^\vee\bigr)\lra  \cT(\cE^\vee)$  induces a map  $\cT_0\bigl((\cE/H)^\vee\bigr)\lra  \cT_0(\cE^\vee)$ and, hence, a
morphism $\pi_\ell^\vee\colon p_1^\ast\Bigl(j_1^\ast\bigl(A_{k,m}(w)\bigr)\Bigr)\lra p_2^\ast\Bigl(j_1^\ast\bigl(A_{k,m}(w')\bigr)\Bigr)$ for every $m\in\N$ which
dualized induces  the second morphism $$\pi_\ell\colon p_2^\ast\Bigl(j_1^\ast\bigl(\cD_{k,m}^o(w') \bigr)\Bigr)\lra p_1^\ast\Bigl(j_1^\ast\bigl(\cD_{k,m}^o(w)
\bigr)\Bigr).$$ As the diagram $$
\begin{array}{ccccc}
\cT_0\bigl((\cE/H)^\vee\bigr) & {\lra} & \cF'\bigl((\cE/H)^\vee\bigr) \\
\downarrow\pi_\ell^\vee&&\downarrow\pi_\ell^\vee  \\
\cT_0(\cE^\vee) & {\lra}& \cF'(\cE^\vee)
\end{array}
$$is commutative, the above maps $\pi_\ell^\vee$ are compatible with the morphisms $$p_2^\ast\bigl(\cM^{(1)}_{k}(w')\bigr)\lra
p_2^\ast\bigl(j_1^\ast\bigl(\cA_{k}^o(w')\bigr)\bigr)\widehat{\otimes}_{\cO_K} \hO_{\fX^{(1)}(w')}$$
and $$p_1^\ast\bigl(\cM^{(1)}_{k}(w)\bigr)\lra
p_1^\ast\bigl(j_1^\ast\bigl(\cA_{k}^o(w')\bigr)\bigr)\widehat{\otimes}_{\cO_K} \hO_{\fX^{(1)}(w)}$$defined using ${\rm dlog}^{\vee,k}$ (see lemma
\ref{prop:factorization} and the following discussion).  Dualizing the compatibility of the two maps $\pi_\ell$ in the statement via $\delta$ follows.
\end{proof}

Now we define the Hecke operators $T_\ell$ for $\ell$ not dividing $Np$ and $U_p$ on modular forms and cohomology. More precisely $T_\ell$ for $\ell$ not dividing
$Np$, respectively $U_p$ on overconvergent modular forms are defined  as the maps $$T_\ell, \quad U_p\colon \rH^0\bigl(\fX(w'), \omega^{\dagger,k}_{\fX(w')}\bigr)
\lra \rH^0(\fX(w), \omega^{\dagger,k}_{\fX(w)}\bigr)$$defined as follows. Recall that by the definition $\omega^{\dagger,k}_{\fX(w)}:=
\left(j_{1,\ast}\big(\cM^{(1)}_{-k}(w)\bigr)\right)^{G_1}[1/p]$, see definition \ref{def:omegak}. Using  the fact that $R^i j_{1,\ast}=0$ for $i\geq 1$ by corollary
\ref{cor:jacyclic} we may identify $$\rH^i\bigl(\fX(w'), \omega^{\dagger,k}_{\fX(w')}\bigr)\cong \rH^i\bigl(\fX(w'),
j_{1,\ast}\big(\cM^{(1)}_{-k}(w')\bigr)[1/p]\bigr)^{G_1} \cong \rH^i\bigl(\fX^{(1)}(w'), \cM^{(1)}_{-k}(w')[1/p]\bigr)^{G_1}$$and similarly $\rH^i(\fX(w),
\omega^{\dagger,k}_{\fX(w)}\bigr) \cong \rH^i\bigl(\fX^{(1)}(w), \cM^{(1)}_{-k}(w')[1/p]\bigr)^{G_1}$ for every $i\in\N$. The maps $T_\ell$ and $U_p$ are defined
using these identifications and taking $G_1$-invariants and inverting $p$ in

$$
\rH^i\bigl(\fX^{(1)}(w'), \cM^{(1)}_{-k}(w')\bigr) \lra \rH^i\bigl(\fX^{(1)}_\ell(w), p_2^\ast(\cM^{(1)}_{-k}(w'))\bigr)\stackrel{\pi_\ell}{\lra}
$$

$$
\stackrel{\pi_\ell}{\lra}\rH^i\bigl(\fX^{(1)}_\ell(w), p_1^\ast(\cM^{(1)}_{-k}(w)[1/p])\bigr) =\rH^0\bigl(\fX^{(1)}(w), p_{1,\ast}p_1^\ast(\cM^{(1)}_{-k}(w))\bigr)
\lra \rH^i\bigl(\fX^{(1)}(w), \cM^{(1)}_{-k}(w))\bigr).
$$
The equality $ \rH^i\bigl(\fX^{(1)}_\ell(w), p_1^\ast(j_1^\ast(\cM^{(1)}_{-k}(w)\bigr) =\rH^i\bigl(\fX^{(1)}(w), p_{1,\ast}p_1^\ast(\cM^{(1)}_{-k}(w)\bigr)$ follows
from a Leray spectral sequence argument using the vanishing of $R^h p_{1,\ast}$, for $h\geq 1$, proven in corollary \ref{cor:jacyclic}. The last map
$$\rH^i\bigl(\fX^{(1)}(w), p_{1,\ast}p_1^\ast(\cM^{(1)}_{-k}(w)[1/p])\bigr) \lra \rH^i\bigl(\fX^{(1)}(w), \cM^{(1)}_{-k}(w)[1/p]\bigr)
$$is the map on cohomology associated to the trace map $p_{1,\ast}p_1^\ast(\cF)\to \cF$ defined in (\ref{def:trace}) and
can be seen as the trace map for Faltings' cohomology ($\rH^i$).

\

Now we assume that $k\in \cW^\ast(B_K)$. We have a $G_K$-equivariant isomorphism
$$
\rH^0\big(\fX(w),\omega^{\dagger,k}_{\fX(w)}\bigr)\cong \rH^0\bigl(X(w), \omega^{\dagger,k}_w\hat{\otimes}_K\C_p\bigr)
$$
and that the latter is provided with Hecke operators given in \cite{andreatta_iovita_stevens}. Recall that in  (\ref{deltakw}) of section \S\ref{sec:dkappa} we
defined a map of sheaves on $\fX(w)$
$$
\delta^\vee_k(w)\colon  \nu^\ast(\cD_k) \lra \omega^{\dagger,k}_{\fX(w)},
$$for $w$  adapted to $k$ and $p^w$ is a uniformizer of $K$.
Moreover in section \ref{sec:cohomegadaggerk} we have calculated the cohomology group
$$
\rH^1\bigl(\fX(w), \omega^{\dagger,k}_{\fX(w)}(1)\bigr)\cong \rH^0\bigl(X(w), \omega^{\dagger,k+2}_w\bigr)\hat{\otimes}_K\C_p.
$$
Combining the remarks above we have an $B_K\hat{\otimes}_K\C_p$-linear map, $G_K$-equivariant
$$
\Psi_{k,w}\colon \rH^1\bigl(\fX(w), \nu^\ast(\cD_k)(1)\bigr)\lra \rH^0\bigl(X(w),\omega^{\dagger,k+2}_w\bigr)\hat{\otimes}_K\C_p.
$$
We have
\begin{theorem}
\label{thm:hecke} The isomorphism $\rH^0\big(\fX(w),\omega^{\dagger,k}_{\fX(w)}\bigr)\cong \rH^0\bigl(X(w), \omega^{\dagger,k}_w\hat{\otimes}_K\C_p\bigr)$ and the
map $\Psi_{k,w}$ commute with the Hecke operators $T_\ell$ and $U_p$ defined above.
\end{theorem}
\begin{proof} Due to the compatibility proven in \ref{lemma:piell} the map on cohomology induced by $\delta^\vee_k(w)$ is compatible with Hecke operators.
It suffices to prove that the isomorphisms $\rH^0\big(\fX(w),\omega^{\dagger,k}_{\fX(w)}\bigr)\cong \rH^0\bigl(X(w), \omega^{\dagger,k}_w\hat{\otimes}_K\C_p\bigr)$
and $\rH^1\bigl(\fX(w), \omega^{\dagger,k}_{\fX(w)}(1)\bigr)\cong \rH^0\bigl(X(w), \omega^{\dagger,k+2}_w\bigr)\hat{\otimes}_K\C_p$ are compatible with the Hecke
operators on $\omega^{\dagger,k}$ defined in \cite{andreatta_iovita_stevens}. By theorem \ref{thm:cohomologyomega} this amounts to prove that the trace map
$p_{1,\ast}\circ p_1^\ast\big(\hO_{\fX^{(1)}(w)}\big)=p_{1,\ast}\big(\hO_{\fX^{(1)}_\ell(w)}\big) \to  \hO_{\fX^{(1)}(w)}$ is compatible with the trace map
$p_{1,\ast} \big(\cO_{\cX^{(r)}_\ell(w)}\bigr)\to  \cO_{\cX^{(1)}(w)}$. This follows, for example, from the explicit description of the trace given after formula
(\ref{def:trace}) in section \S\ref{sec:fy}.
\end{proof}

\section{The Eichler-Shimura isomorphism}
\label{sec:eichlershimura}

Let $U\subset \cW^\ast$ be a wide open disk with universal weight $k_U$ and ring of bounded analytic functions
$\Lambda_U$. We have described in lemma \ref{lemma:compnuast}
the following sequence of maps
$$
\rH^1\bigl(\Gamma, D_U\bigr)\hat{\otimes}_K\C_p(1)\cong \rH^1\bigl(\fX(N,p), \cD_U\bigr)\hat{\otimes}_K\C_p(1) \lra \rH^1\Bigl(\fX(w),
\nu^\ast\bigl({\cD}_U(1)\bigr)\Bigr).
$$
These maps are equivariant for the action of the Galois group $G_K$ and the Hecke operators and $(1)$ denotes the usual Tate twist. Now let us recall that in
(\ref{deltakw}) of section \S\ref{sec:dkappa} we defined a map of sheaves $$\delta^\vee_{k_U}(w)\colon  \nu^\ast(\cD_U) \lra \omega_{\fX(w)}^{\dagger,k_U}$$ on
$\fX(w)$. Therefore, by theorem \ref{thm:cohomologyomega} we have maps:
$$
\rH^1\bigl(\fX(w), \nu^\ast\bigl(\cD_{U}\bigr)\otimes K)(1)\bigr)\lra \rH^1\bigl(\fX(w), \omega_{\fX(w)}^{\dagger,k_U}(1)\bigr)\cong \rH^0\bigl(X(w),
\omega_w^{\dagger,k_U+2}\bigr)\hat{\otimes}_K\C_p.
$$
These maps are also equivariant for the action of $G_K$ and the Hecke operators due to theorem \ref{thm:hecke}. Putting everything together we have a map:
$$
\Psi_U\colon \rH^1\bigl(\Gamma, D_U\bigr)\hat{\otimes}_K\C_p(1)\lra \rH^0\bigl(X(w), \omega^{\dagger,k_U+2}_w\bigr)\hat{\otimes}_K\C_p.
$$
This map is equivariant for the action of $G_K$ and the Hecke operators and commutes with specializations. In other words if $k\in U(K)$
is a weight, in similar way we obtain a map
$$
\Psi_k\colon \rH^1\bigl(\Gamma, D_k\bigr)\otimes_K\C_p(1)\lra \rH^0\bigl(X(w), \omega_w^{\dagger,k+2}\bigr)\otimes_K\C_p,
$$
such that the following diagram is commutative:
$$
\begin{array}{ccccccccccc}
\rH^1\bigl(\Gamma, D_U\bigr)\hat{\otimes}_K\C_p(1)&\stackrel{\Psi_U}{\lra}& \rH^0\bigl(X(w), \omega^{\dagger,k_U+2}_w\bigr)\hat{\otimes}_K\C_p\\
\downarrow\rho_{k}&&\downarrow\rho_k\\
\rH^1\bigl(\Gamma, D_k\bigr)\otimes_K\C_p(1)&\stackrel{\Psi_k}{\lra}& \rH^0\bigl(X(w), \omega_w^{\dagger,k+2}\bigr)\otimes_K\C_p
\end{array}
$$
The goal of this section is to study the maps $\Psi_k$ using the map $\Psi_U$.

\subsection{The Main result}
\label{sec:mainresult}

Let us fix a slope $h\in \Q$, $h\ge 0$ and an integer $k_0$ such that $h<k_0+1$ and think about
$k_0$ as a point in $\cW^\ast(K)$. We let $U\subset\cW^\ast$ be a wide open disk defined
over $K$ such that

a) $k_0\in U(K)$

b) Both $\rH^1\bigl(\Gamma, D_U\bigr)$ and $\rH^0\bigl(X(w),\omega^{\dagger,k_U+2}_w\bigr)$ have slope $h$ decompositions.

We denote by $k_U$ the universal weight of $U$, by $\Lambda_U$ the ring of bounded rigid functions on
$U$ and by $B_U:=\Lambda_U\otimes_{\cO_K}K$. Then $B_U$ is a $K$-Banach algebra and moreover it is a
principal ideal domain.

We let $\rH^1\bigl(\Gamma, D_U\bigr)^{(h)}$ and $\rH^0\bigl(X(w),\omega^{\dagger,k_U}_w\bigr)^{( h)}$
be the parts corresponding to slopes smaller or equal to
$h$ of the respective cohomology groups. Both these modules are free $B_U$-modules of finite rank and $\Psi_U$ induces an $(B_U\hat{\otimes}_K\C_p)$-linear map
$$
\Psi_U^{(h)}\colon \rH^1\bigl(\Gamma, D_U\bigr)^{(h)}\hat{\otimes}_K\C_p(1)\lra \rH^0\bigl(X(w),\omega^{\dagger,k_U+2}_w\bigr)^{(h)}\hat{\otimes}\C_p
$$
compatible with specializations and equivariant for the action of $G_K$ and the Hecke operators $T_\ell$, $(\ell, Np)=1$. We denote by $M_U$ the kernel of $\Psi_U$
and by $M_U^{(h)}$ the kernel of $\Psi_U^{(h)}$. Then $M_U$ is an $(B_U\hat{\otimes}_K\C_p)$-submodule of $\rH^1\bigl(\Gamma, D_U\bigr)\hat{\otimes}_K\C_p(1)$,
preserved by $G_K$ and the Hecke operators $T_\ell$, for $(\ell, Np)=1$ and $M_U^{(h)}$ is a $(B_U\hat{\otimes}\C_p)$-submodule on which $U_p$-acts by slopes
smaller or equal to $h$. We have

\begin{theorem}
\label{thm:muh}
a) There is a non-zero element $b\in B:=(B_U\hat{\otimes}\C_p)$ such that
$b$ annihilates ${\rm Coker}(\Psi_U^{(h)})$.

b)Let $Z\subset U(\C_p)$ be the (finite) set of zeroes of $b\in B$ at a) above and let $V\subset U$ be a
wide open disk defined over $K$ satisfying:
$V(K)$ contains an integer $k$ such that $k>h-1$ and $V(\C_p)\cap Z=\phi$. Then restriction to $V$ induces an exact sequence
$$
0\lra M_V^{(h)}\lra \rH^1\bigl(\Gamma, D_V\bigr)\hat{\otimes}_K\C_p(1)\stackrel{\Psi_V^{(h)}}{\lra} \rH^0\bigl(X(w), \omega^{\dagger,k_V+2}_w\bigr)\ho_K\C_p\lra 0.
$$

c) For a wide open disk $V$ as at b) above let us denote by $\chi_V^{\rm univ}$ the following composition:
$$G_K\stackrel{\chi}{\lra}\Z_p^\times\stackrel{k_V}{\lra}B_V^\times\lra (B_V\hat{\otimes}\C_p)^\times,
$$
where $\chi$ is the cyclotomic character of $K$. We call $\chi_V^{\rm univ}$ the universal cyclotomic character
attached to $V$.
Then the semilinear action of $G_K$ on the module $S_V:=M_V^{(h)}\bigl(\chi^{-1}(\chi_V^{\rm univ})^{-1}\bigr)$ is trivial.
Moreover $S_V$ is a
finite, projective $(B_V\ho\C_p)$-module with trivial semilinear $G_K$-action. Obviously $M_V^{(h)}=S_V(\chi\cdot\chi_V^{\rm univ})$ as semilinear $G_K$-modules.

d) For each $V$ as at b) and c) above there is a non-zero element $0\ne \beta\in B_V$ such that
the localized exact sequence
$$
0\lra \bigl(S_V(\chi\cdot \chi_U^{\rm univ})\bigr)_{\beta}\lra \Bigl(\rH^1\bigl(\Gamma, D_V\bigr)^{(h)}\hat{\otimes} \C_p(1)\Bigr)_{\beta}\lra
\Bigl(\rH^0\bigl(X(w),\omega^{\dagger,k_V}_w\bigr)^{(h)}\hat{\otimes}\C_p\Bigr)_\beta\lra 0
$$
is naturally and uniquely split as as a sequence of $G_K$-modules.
\end{theorem}

Before proving the theorem let us point out some of its consequences.

\begin{corollary}
\label{cor:consequences}
Assume we have $U$, $k_0$, $h$ as at the beginning of section \ref{sec:mainresult}.

a) There exists a finite set of weights $Z'\subset U(\C_p)$ such that for every
$k\in U(K)-Z'$ we have a natural
isomorphism as $\C_p$-vector spaces equivariant for the semilinear $G_K$-action
$$
\Psi_k^{\rm ES}\colon \rH^1\bigl(\Gamma, D_k\bigr)^{(h)}\otimes_K\C_p(1)\cong S_k(k+1)\oplus \rH^0\bigl(X(w), \omega^{\dagger,k+2}_w\bigr)^{(h)}\ho_K\C_p.
$$
Here $S_k$ is a finite $\C_p$-vector space with trivial, semilinear action of $G_K$.

b) The set $Z'$ at a) above contains the integers $\kappa\in U\cap \Z$ such that $0\le \kappa \le h-1$.
\end{corollary}

\bigskip
\noindent
{\bf Proof of theorem \ref{thm:muh}.} Let $k\in U(K)\cap \Z$ such that $k\ge 0$.

We will first recall Faltings' version of the classical Eichler-Shimura isomorphism, see \cite{faltingsmodular}.
Let us recall the ind-continuous sheaf $\cV_k\cong \Symm^k(\cT)\otimes_{\Z_p} K$ on $X(N,p)^{\rm ket}$,
which can also be seen as an ind-continuous  sheaf on $\fX(N,p)$. Let $k\ge 0$ be an integer.
The main result of \cite{faltingsmodular} is that there is a $\C_p$-linear, $G_K$-equivariant isomorphism (the Eichler-Shimura isomorphism)
$$
\Phi_k\colon \rH^1\bigl(X(N,p)_{\Kbar}^{\rm ket}, \cV_k\bigr)\otimes_K\C_p(1)\cong \rH^0\bigl(X(N,p), \omega^{k+2}\bigr)\otimes_K\C_p\oplus \rH^1\bigl(X(N,p),
\omega^{-k}\bigr)\otimes_K\C_p(k+1).
$$
We have a natural isomorphism $\rH^1\bigl(X(N,p)_{\Kbar}^{\rm ket}, \cV_k(1)\bigr)\cong \rH^1\bigl(\Gamma, V_k(1)\bigr)$ compatible with all structure and therefore
we have a natural diagram
$$
\begin{array}{cccccccccc}
\rH^1\bigl(\Gamma, D_U\bigr)\hat{\otimes}_K\C_p(1)&\stackrel{\Psi_U}{\lra}&\rH^0\bigl(X(w), \omega^{\dagger,k_U+2}_w\bigr)\hat{\otimes}\C_p\\
\downarrow&&\downarrow\\
\rH^1\bigl(\Gamma, D_k\bigr)\otimes_K\C_p(1)&\stackrel{\Psi_k}{\lra}&\rH^0\bigl(X(w),\omega^{\dagger,k+2}_w\bigr)\otimes_K\C_p\\
\downarrow&&\uparrow\\
\rH^1\bigl(\Gamma, V_k(1)\bigr)\otimes_K\C_p&\stackrel{p_2\circ\Phi_k}{\lra}&\rH^0\bigl(X(N,p), \omega^{k+2}\bigr)\otimes_K\C_p
\end{array}
$$
where the left vertical maps are induced by the specializations $D_U\lra D_k\lra V_k$ (see (\ref{rhok}) in \S\ref{sec:locallyanalytic}), same as the top right map. The lower right map is
restriction (let us recall that if $m>0$ is an integer $\omega^m|_{X(w)}\cong \omega_w^{\dagger, m}$).

{\em Claim 1} The above diagram is commutative. In fact we know that the upper rectangle is commutative so it would be enough to show that the lower rectangle is
also commutative.

\begin{proof}
For this we have to briefly recall (in a slightly different formulation) the proof of Faltings' result, namely the definition of the map $p_2\circ \Phi_k$. We first
notice that arguing as in the proof of proposition \ref{prop:XfXcompcoh} we have a natural isomorphism (as $\cX(N,p)$ is proper and semistable)
$$
\rH^1\bigl(X(N,p)_{\Kbar}^{\rm ket}, \cV_k(1)\bigr)\otimes_K\C_p\lra \rH^1\bigl(\fX(N,p), \Symm^k(\cT)\hat{\otimes}\hO_{\fX(N,p)}\otimes_{\cO_K}K(1)\bigr).
$$
As in definition \ref{def:omegak}  we denote by
$$
\omega_{\fX(N,p)}:=v_{\fX(N,p)}^\ast\bigl(\omega\bigr)\hat{\otimes}_{\hO_{\fX(N,p)}^{\rm un}}
\hO_{\fX(N,p)}.
$$
We then have the Hodge-Tate sequence of sheaves of $\hO_{\fX(N,p)}$-modules on $\fX(N,p)$
$$
0\lra \omega^{-1}_{\fX(N,p)}(1)\lra \cT\hat{\otimes}\hO_{\fX(N,p)}\lra \omega_{\fX(N,p)}\lra 0.
$$
This sequence is not exact but it becomes exact if we invert $p$, or if we tensor with $K$. We obtain
$$
0\lra (\omega^{-1}_{\fX(N,p)}\otimes K)(1)\lra \cT\hat{\otimes}\hO_{\fX(N,p)}\otimes_{\cO_K} K\lra (\omega_{\fX(N,p)}\otimes K)\lra 0.
$$
One shows by induction that, for every $m\ge 1$ we have a surjective map
of $(\hO_{\fX(N,p)}\otimes K)$-modules:
$$\Symm^k(\cT)\hat{\otimes}\hO_{\fX(N,p)}\otimes_{\cO_K} K(1)\lra \omega^k_{\fX(N,p)}\otimes K(1),
$$
which induces the morphism
$$
\rH^1\bigl(\fX(N,p), \Symm^k(\cT)\hat{\otimes}\hO_{\fX(N,p)}\otimes_{\cO_K} K(1)\bigr)\lra \rH^1\bigl(\fX(N,p), \omega^k_{\fX(N,p)}\otimes K(1)\bigr).
$$
Similarly to the calculations in the proof of theorem \ref{thm:cohomologyomega}, we calculate $\rH^1\bigl(\fX(N,p), \omega^k_{\fX(N,p)}\otimes K(1)\bigr)$ using the
spectral sequence
$$
\rH^a\Bigl(\cX(N,p)^{\rm ket}, R^bv_{\fX(N,p), \ast}\bigl(\omega^k_{\fX(N,p)}\otimes K(1)\bigr)\Bigr)\Longrightarrow \rH^{a+b}\bigl(\fX(N,p),
\omega^k_{\fX(N,p)}\otimes K(1)\bigr).
$$
For $a+b=1$ we have an edge morphism
$$
\rH^1\bigl(\fX(N,p), \omega_{\fX(N,p)}\otimes K(1)\bigr)\lra \rH^0\Bigl(\cX(N,p)^{\rm ket}, R^1v_{\fX(N,p),\ast}\bigl(\omega^k_{\fX(N,p)}\otimes K(1)\bigr)\Bigr).
$$
As in the proof of theorem \ref{thm:cohomologyomega} we have:
$$
R^1v_{\fX(N,p),\ast}\bigl(\omega^k_{\fX(N,p)}\otimes K(1)\bigr)\cong \omega^k\hat{\otimes}R^1v_{\fX(N,p),\ast}\bigl(\hO_{\fX(N,p)}\otimes K(1)\bigr)
\cong \omega^{k+2}\hat{\otimes}\C_p.
$$
Therefore we obtain a natural composition
$$
\rH^1\bigl(\fX(N,p), \cV_k\hat{\otimes}\hO_{\fX(N,p)}(1)\bigr)\lra \rH^1\bigl(\fX(N,p), \omega^k_{\fX(N,p)}\otimes K(1)\bigr)\lra \rH^0\bigl(X(N,p),
\omega^{k+2}\bigr)\otimes_K\C_p,
$$
which is the map $p_2\circ\Phi_k$ appearing in \cite{faltingsmodular}.

Let us now see what happens on $\fX(w)$. We denote by $\cT_w:=\nu^\ast(\cT)$. We have a natural map  $\nu^\ast\bigl(\cD_k\bigr)\lra \Symm^k(\cT_w)=
\nu^\ast\bigl(\cV_k\bigr)$ (associated to (\ref{rhok}) in \S\ref{sec:locallyanalytic}) and the composite with $\Symm^m(\cT)\hat{\otimes}\hO_{\fX(N,p)} \otimes_{\cO_K} K\lra \omega_{\fX(N,p)}^k\otimes K$ is the morphism $
\delta^\vee_k(w)\colon  \nu^\ast(\cD_k) \lra \omega_{\fX(N,p)}^k\otimes K$ of (\ref{deltakw}) of section \S\ref{sec:dkappa}. Moreover we have the following natural
commutative diagram of sites and continuous functors, inducing morphisms of topoi:
$$
\begin{array}{cccccccc}
\cX(N,p)^{\rm ket}&\stackrel{v_{\fX(N,p)}}{\lra}&\fX(N,p)\\
\downarrow\mu&&\downarrow\nu\\
\cX(w)^{\rm ket}&\stackrel{v_{\fX(w)}}{\lra}&\fX(w)
\end{array}
$$
Here $\mu,\nu$ are induced by the natural morphism of log formal schemes $\cX(w)\lra \cX(N,p)$.

Let us recall from theorem \ref{thm:cohomologyomega} that we have isomorphisms:
$$
\rH^1\bigl(\fX(w), \omega^{\dagger,k}_{\fX(w)}(1)\bigr)\cong \rH^0\bigl(\cX(w)^{\rm ket}, R^1v_{\fX(w),\ast}(\omega^{\dagger,k}_{\fX(w)}(1))\bigr)\cong
\rH^0\bigl(X(w), \omega^{\dagger,k+2}_w\bigr)\otimes_K\C_p.
$$
We also have the following sequence of isomorphisms of sheaves on $\fX(w)$:
$$
R^1v_{\fX(w),\ast}\Bigl(\nu^\ast\bigl(\omega^{k+2}_{\fX(N,p)}\otimes K(1)\bigr)\Bigr)\cong
R^1v_{\fX(w),\ast}\bigl(\omega^{\dagger,k+2}_{\fX(w)}(1)\bigr)\cong
$$
$$
\cong\omega_w^{\dagger,k+2}\otimes\C_p\cong \mu^\ast\bigl(\omega^{k+2}\otimes \C_p\bigr)\cong
\mu^\ast\Bigl(R^1v_{\fX(N,p),\ast}\bigl(\omega^{k+2}_{\fX(N,p)}\otimes K(1)\bigr)\Bigr).
$$
Finally putting together what we have done so far we have the following commutative diagram:
$$
\begin{array}{cccccccccc}
\rH^1\bigl(\fX(N,p), \cV_k\hat{\otimes}\hO_{\fX(N,p)}(1)\bigr)&\lra&\rH^1\bigl(\fX(w), \Symm^k(\cT_w)\hat{\otimes}\hO_{\fX(w)}\otimes_{\cO_K} K(1)\bigr)\\
\downarrow&&\downarrow\\
\rH^1\bigl(\fX(N,p), \omega^k_{\fX(N,p)}\otimes K(1)\bigr)&\lra&\rH^1\bigl(\fX(w),\omega_{\fX(w)}^{\dagger,k}(1)\bigr)\\
\downarrow&&\downarrow\cong\\
\rH^0\bigl(X(N,p), \omega^{k+2}\otimes K\bigr)\otimes_K\C_p&\stackrel{\varphi}{\lra}&\rH^0\bigl(X(w),\omega^{\dagger,k+2}_w\bigr)\otimes_K\C_p
\end{array}
$$
where $\varphi$ is the restriction map. This proves the claim 1.\end{proof}

\noindent {\bf Remark 1} Let us suppose now that $h\ge 0$, $h\in \Q$ is a slope such that both $\rH^1\bigl(\Gamma, D_U\bigr)$ and $\rH^0\bigl(X(w),
\omega^{\dagger,k_U+2}_w\bigr)$ have slope decompositions. Let $k\in U(K)\cap \Z$, $k\ge 0$. Then the diagram
$$
\begin{array}{cccccccccc}
\rH^1\bigl(\Gamma, D_U\bigr)^{(h)}\hat{\otimes}_K\C_p(1)&\stackrel{\Psi_U^{(h)}}{\lra}&\rH^0\bigl(X(w), \omega^{\dagger,k_U+2}_w\bigr)^{(h)}\hat{\otimes}\C_p\\
\downarrow&&\downarrow\\
\rH^1\bigl(\Gamma, D_k\bigr)^{(h)}\otimes_K\C_p(1)&\stackrel{\Psi_k^{(h)}}{\lra}&\rH^0\bigl(X(w),\omega^{\dagger,k+2}_w\bigr)^{(h)}\otimes_K\C_p\\
\downarrow&&\uparrow\varphi\\
\rH^1\bigl(\Gamma, V_k(1)\bigr)^{(h)}\otimes_K\C_p&\stackrel{p_2\circ\Phi_k}{\lra}&\rH^0\bigl(X(N,p), \omega^{k+2}\bigr)^{(h)}\otimes_K\C_p
\end{array}
$$
is commutative.

\bigskip
\noindent
{\em Claim 2} Let $U\subset \cW$ be a wide open disk, $k_U:\Z_p^\times\lra B_U^\times$ the universal
weight, $w>0, w\in\Q$ adapted to $k_U$  and $k\in U(K)$. Let $t_k\in B_U$ be a
rigid analytic function on $U$ which vanishes with order $1$ at $k$ and nowhere else on $U$.
The specialization maps $D_U\lra D_k$ and $\omega^{\dagger, k_U}_w\lra \omega^{\dagger,k}_w$
induce the following exact sequences
$$
\rH^1\bigl(\Gamma, D_U\bigr)\stackrel{t_k}{\lra}\rH^1\bigl(\Gamma, D_U\bigr)\lra \rH^1\bigl(\Gamma, D_k\bigr)\lra 0
$$
and
$$
0\lra \rH^0\bigl(X(w),\omega^{\dagger,k_U}_w\bigr)\stackrel{t_k}{\lra}\rH^0\bigl(X(w),\omega^{\dagger,k_U}_w\bigr)\lra \rH^0\bigl(X(w),
\omega^{\dagger,k}_w\bigr)\lra 0.
$$

\begin{proof} In fact the specialization maps are part of the following exact sequences:
$$
0\lra D_U\stackrel{t_k}{\lra}D_U\lra D_k\lra 0\mbox{ and } 0\lra \omega^{\dagger,k_U}_w\stackrel{t_k}{\lra}
\omega^{\dagger,k_U}_w\lra\omega^{\dagger,k}_w\lra 0.
$$
It follows that we have an exact sequence of $B_U$-modules
$$
\rH^1\bigl(\Gamma, D_U\bigr)\stackrel{t_k}{\lra}\rH^1\bigl(\Gamma, D_U\bigr)\lra \rH^1\bigl(\Gamma, D_k\bigr)\lra \rH^2\bigl(\Gamma, D_U\bigr).
$$
Now let us recall that $\Gamma=\Gamma_1(N)\cap \Gamma_0(p)$ is a torsion free group, therefore it is the fundamental group of the complement $Y(N,p)_{/\C}$ of a
non-void, finite set of points in a compact Riemann surface $X(N,p)_{/\C}$ and so it has cohomological dimension $1$. It follows that $\rH^2\bigl(\Gamma,
D_U\bigr)=0$ and the first claim follows.

Let us also remark that we have an exact sequence of $B_U$-modules
$$
0\lra \rH^0\bigl(X(w), \omega^{\dagger,k_U}\bigr)\stackrel{t_k}{\lra}\rH^0\bigl(X(w), \omega^{\dagger,k_U}\bigr)\lra \rH^0\bigl(X(w), \omega^{\dagger,k}\bigr)\lra
\rH^1\bigl(X(w), \omega^{\dagger,k_U}\bigr).
$$
As $X(w)$ is an affinoid subdomain and $\omega^{\dagger,k_U}_w$ is a sheaf of $B_U$-Banach modules by the Appendix of \cite{andreatta_iovita_pilloni} it follows
that $\rH^1\bigl(X(w), \omega^{\dagger,k_U}\bigr)=0$.
\end{proof}

\noindent {\bf Remark 2} Let $U\subset\cW$ be a wide open disk defined over $K$, $k_U:\Z_p^\times \lra B_U^\times$ the universal character, $w>0,w\in
\Q$ adapted to $k_U$ and $h\ge 0,h\in\Q$ a slope. We suppose that both $\rH^1\bigl(\Gamma, D_U\bigr)$ and $\rH^0\bigl(X(w), \omega^{\dagger,k_U}_w\bigr)$ have slope
decompositions. Let $k\in U(K)\cap \Z$, $k\ge 0$ and let us recall the commutative diagram of the previous remark.
$$
\begin{array}{cccccccccc}
\rH^1\bigl(\Gamma, D_U\bigr)^{(h)}\hat{\otimes}_K\C_p(1)&\stackrel{\Psi_U^{(h)}}{\lra}&\rH^0\bigl(X(w), \omega^{\dagger,k_U+2}_w\bigr)^{(h)}\hat{\otimes}\C_p\\
\downarrow&&\downarrow\\
\rH^1\bigl(\Gamma, D_k\bigr)^{(h)}\otimes_K\C_p(1)&\stackrel{\Psi_k^{(h)}}{\lra}&\rH^0\bigl(X(w),\omega^{\dagger,k+2}_w\bigr)^{(h)}
\otimes_K\C_p\\
\downarrow\psi&&\uparrow\varphi\\
\rH^1\bigl(\Gamma, V_k(1)\bigr)^{(h)}\otimes_K\C_p&\stackrel{p_2\circ\Phi_k}{\lra}&\rH^0\bigl(X(N,p), \omega^{k+2}\bigr)^{(h)}\otimes_K\C_p
\end{array}
$$

First let us remark that the surjective map $p_2\circ \Phi_k$ induces a surjective $\C_p$-linear map denoted by the same symbols $\rH^1\bigl(\Gamma,
V_k(1)\bigr)^{(h)}\otimes_K\C_p\stackrel{p_2\circ\Phi_k}{\lra}\rH^0\bigl(X(N,p), \omega^{k+2}\bigr)^{(h)}\otimes_K\C_p$.

 Then, there are two cases:

i) $k+1\ge h$. Then the classicity theorems both for overconvergent modular symbols and for
overconvergent modular forms imply that $\psi$ and $\varphi$ are isomorphisms.
It follows that in this case $\Psi_k^{(h)}$ is surjective.

ii) $k+1<h$. In this case the commutativity of the lower rectangle implies that the image of $\Psi_k$ is contained in the image of the classical forms inside
$\rH^0\bigl(X(w), \omega^{\dagger,k+2}_w\bigr)^{(h)}\hat{\otimes}\C_p$. The relationship between $h$ and $k$ implies that in this case $\Psi_k^{(h)}$ is not
surjective in general.

\bigskip
\noindent Now we prove a) of theorem \ref{thm:muh}. We assume that $U,w,h$ are as at the beginning of section \ref{sec:mainresult}. Let $k\in U(K)\cap \Z$ be such
that $h<k+1$ (there are infinitely many such $k$'s).

Let us recall that both $\rH^1\bigl(\Gamma, D_U\bigr)^{(h)}\hat{\otimes}_K\C_p(1)$ and $\rH^0\bigl(X(w),\omega_w^{\dagger,k_U}\bigr)^{(h)} \hat{\otimes}_K\C_p$ are
finite free $B:=(B_U\hat{\otimes}_K\C_p)$-modules of ranks $n$ and $m$ respectively. By choosing a basis, we can write $\Phi_U^{(h)}$ as a matrix
$\Phi_U^{(h)}=\bigl(a_{ij}\bigr)_{1\le i\le n,1\le j\le m}$ with $a_{ij}\in B$ for all $i,j$. By Claim 2 we have:
$$
\rH^1\bigl(\Gamma, D_k\bigr)^{(h)}\otimes_K\C_p\cong \Bigl(\rH^1\bigl(\Gamma, D_U\bigr)^{(h)}\hat{\otimes}_K\C_p\Bigr)/t_k\Bigl(\rH^1\bigl(\Gamma,
D_U\bigr)^{(h)}\hat{\otimes}_K\C_p\Bigr),
$$
and
$$
\rH^0\bigl(X(w),\omega_w^{\dagger,k+2}\bigr)^{(h)}\hat{\otimes}_K\C_p\cong
\Bigl(\rH^0\bigl(X(w),\omega_w^{\dagger,k_U+2}\bigr)^{(h)}\hat{\otimes}_K\C_p\Bigr)/t_k\Bigl(\rH^0\bigl(X(w),
\omega_w^{\dagger,k_U+2}\bigr)^{(h)}\hat{\otimes}_K\C_p\Bigr),
$$
where let us recall that we have denoted by $t_k$ an element of $B$ which vanishes with order $1$ at $k$ and nowhere else in $U$.

Moreover $\Psi_k^{(h)}=\bigl(a_{ij}(k)\bigr)_{i,j}$. The second remark above implies that
$n\ge m$ and the matrix $\Psi_k^{(h)}$ has rank exactly $m$, i.e. there is an $m\times m$-minor of
$\Psi_U^{(h)}$, $Q$, whose determinant has the property
$\det(Q)(k)\ne 0$. Therefore $b:=\det(Q)\ne 0$, $b\in B$
has the property that $b{\rm Coker}\bigl(\Psi_U^{(h)}\bigr)=0$.

\bigskip
\noindent
Let us now prove b) of theorem \ref{thm:muh}. Let $Z$ denote the set of zeroes of $b$ and
let $V\subset U$ be a connected affinoid subdomain defined over $K$ such that
$V(\C_p)\cap Z=\phi$ and such that $V(K)$ contains an integer $k>h-1$.
Then $b|_{V\times_K\C_p}\in (B_V\ho_K\C_p)^\times$,
therefore the following sequence
$$
0\lra M_V^{(h)}\lra \rH^1\bigl(\Gamma, D_V\bigr)^{(h)}\hat{\otimes}_K\C_p(1)\stackrel{\Psi_V^{(h)}}{\lra}
\rH^0\bigl(X(w),\omega^{\dagger,k_V+2}_w\bigr)^{(h)}\hat{\otimes}_K\C_p\lra 0
$$
is exact, where we have denoted by $M_V^{(h)}$ the kernel of $\Psi_V^{(h)}$. As $\rH^0\bigl(X(w),\omega^{\dagger,k_V+2}_w\bigr)^{(h)}\hat{\otimes}_K\C_p$ is a free
$(B_V\ho_K\C_p)$-module of finite rank the above exact sequence is split (as sequence of $(B_V\ho_K\C_p)$-modules ignoring for the moment the $G_K$-action).
Therefore $M_V^{(h)}$ is a finite projective $(B_V\ho_K\C_p)$-module
and because $(B_V\ho_K\C_p)$is a PID, $M_V^{(h)}$ is a finite and free $(B_V\ho_K\C_p)$-module of finite rank.

{\bf Remark 3} In fact we also have a localized exact sequence of $B_b$-modules
($0\ne b\in B$ is the element chosen at a) above)
$$
0\lra \bigl(M_U^{(h)}\bigr)_b\lra \Bigl(\rH^1\bigl(\Gamma, D_U\bigr)^{(h)}\hat{\otimes}_K\C_p(1)\Bigr)_b\lra
\Bigl(\rH^0\bigl(X(w),\omega^{\dagger,k_U+2}_w\bigr)^{(h)}\hat{\otimes}_K\C_p\Bigr)_b\lra 0.
$$
As in general $b$ is not invariant under $G_K$, the above exact sequence is not $G_K$-equivariant.

\bigskip
Now we prove c) of theorem \ref{thm:muh}.
Let $V$ be as in the theorem and we denote by  $S_V:=M_V^{(h)}\bigl(\chi^{-1}(\cdot \chi_U^{\rm univ})^{-1}\bigr)$.
It is a finite free
$(B_V\ho_K\C_p)$-module of rank say $q=n-m$ endowed with a continuous, semilinear action of $G_K$. Let us briefly recall
the so called {\em Sen's theory in families}, see \cite{sen1}, \cite{sen2} and also the section \S 2 of \cite{kisin}.

We assume (to simplify the exposition) that $K$ contains a non trivial $p$-th root of $1$.
Let $R$ be an affinoid $K$-algebra, $M$ a finite free $R\ho_K\C_p$-module of rank $q$ with a continuous, semilinear
action of $G_K$. The action of $G_K$ on $R\ho_K\C_p$ is via its natural action on $\C_p$.

Let $K'\subset \C_p$ be a finite extension, we denote by $H_{K'}:={\rm Ker}\bigl(\chi:G_K\lra (1+p\Z_p)\bigr)$ and
$\Gamma_{K'}:=G_{K'}/H_{K'}$. Also $K'_{\infty}:=\Kbar^{H_{K'}}$ and $\widehat{K}'_{\infty}:=\C_p^{H_{K'}}$.

We denote by
$$
\widehat{W}_{K'_{\infty}}\bigl(M\bigr):=M^{H_{K'}}.
$$
Then, if $K'$ is large enough (but still a finite extension of $K$) then
$\widehat{W}_{K'_{\infty}}\bigl(M\bigr)$ is a free $\widehat{K}'_{\infty}\ho_K R$-module of
rank $q$ and the natural map $\C_p\ho_{\widehat{K}'_{\infty}}\widehat{W}_{K'_{\infty}}\bigl(M\bigr)\lra
M$ is an isomorphism.

There is a finite extension $K'$ of $K$ (possibly larger then at the previous step)
such that $\widehat{W}_{K'_{\infty}}\bigl(M\bigr)$ has a basis $\{e_1,e_2,...,e_q\}$
over $\widehat{K}'_{\infty}\ho_K R$ such that the $K'\otimes_K R$-submodule
$W_\ast$ generated by this basis in $\widehat{W}_{K'_{\infty}}\bigl(M\bigr)$ is stable by
$\Gamma_{K'}$. If we denote by $\gamma$ a topological generator of $\Gamma_{K'}$, we define
the linear endomorphism $\phi\in {\rm End}_{K'\otimes_KR}\bigl(W_\ast\bigr)$ by
$$
\phi:=\frac{\log(\gamma^{p^r})}{\log(\chi(\gamma^{p^r})}\mbox{ for some }r>>0.
$$
We extend $\phi$ by linearity to $\widehat{W}_{K'_{\infty}}\bigl(M\bigr)$, where it is independent
of all the choices and whose characteristic polynomial has coefficients in $R$. It is called
{\bf the Sen operator} associated to $M$. Its importance consists in that its formation commutes with base change and
for $r$ large enough
the action of $\gamma^{p^r}$ on $W_\ast$ is determined by
$$
\gamma^{p^r}|_{W_\ast}=\exp\bigl(\log(\chi(\gamma^{p^r}))\phi\bigr).
$$

\bigskip
\noindent
We will apply the theory above as follows. Let $Z\subset U(\C_p)$ be the finite set of
zeroes of the function $b\in B$ above and let $V\subset U$ be a wide open disk
 defined over $K$ which contains an integer $k>h-1$ and such that $V(\C_p)\cap Z=\phi$.
Let $B_V:=\Lambda_V\otimes_{\cO_K}K$, where as usual $\Lambda_V$  denotes the algebra of bounded
rigid functions on $V$. Let
$$
S_V:=(M_V^{(h)}\bigl(\chi^{-1}\cdot(\chi_U^{\rm univ})^{-1}\bigr).
$$
Let us recall that $S_V$ is a free $(B_V\ho_K\C_p)$-module of rank $q$
with continuous, semilinear action of $G_K$.

Let $\phi_V$ denote the Sen operator attached to $S_V$ and let $K'$ be a finite, Galois
extension of $K$ in $\C_p$ such that:

i) $\widehat{W}_{K'_{\infty}}\bigl(S_V\bigr)$ is a free $(A_V\ho_{K}\widehat{K}'_{\infty})$-module
of rank $q$,

ii) There is a basis $\{e_1,e_2,   ,e_q\}$ of $\widehat{W}_{K'_{\infty}}\bigl(S_V\bigr)$ over
$(A_V\ho_{K}\widehat{K}'_{\infty})$ such that $W_\ast:=(K'\otimes_KB_V)e_1+...+(K'\otimes_KB_V)e_q$
is stable under $\Gamma_{K'}$

and

iii) The action of $\gamma$ on this basis is given by:
$$
\gamma(e_i)=\exp\bigl(\log(\chi(\gamma))\phi\bigr)(e_i)\mbox{ for every }.1\le i\le q,
$$
 where $\gamma$ is a topological generator
of $\Gamma_{K'}$.

Let us write the matrix of $\phi_V$ in the basis $\{e_1,e_2,\ldots,e_q\}$ as $\bigl(\alpha_{ij}\bigr)_{1\le i,j\le q}
\in M_{q\times q}\bigl(\widehat{K}'_{\infty}\ho_K B_V\bigr).$

Let now $k\in V(K)$ be an integer such that $k>h-1$ (there are infinitely many such weights).
We have an exact sequence of $(B_V\ho_K\C_p)$-modules, with $G_K$ and Hecke actions
$$
0\lra S_V\lra \rH^1\bigl(\Gamma, D_V\bigr)^{(h)}\ho_K\C_p(1)\bigl(\chi^{-1}(\chi_V^{\rm univ})^{-1}\bigr)\lra
$$
$$
\lra \rH^0\bigl(X(w), \omega_w^{\dagger,k_V+2}\bigr)^{(h)}\ho_K\C_p\bigl(\chi^{-1}(\chi_V^{\rm univ})^{-1}\bigr)\lra 0.
$$
We now specialize the sequence at the weight $k$, i.e. tensor over $B_V$ with $K$, for the map $B_V\lra K$ sending
$\alpha\rightarrow \alpha(k)$. As usual we denote by $t_k$ a generator of the kernel of the above map which does not vanish anywhere
else in $V$. Because in the above exact sequence all modules are free $(B_V\ho\C_p)$-modules, specialization gives an exact
sequence. Comparing with Faltings' result above we obtain the following commutative diagram with exact rows
$$
\begin{array}{cccccccccc}
\rH^1\bigl(\Gamma, D_k(1)\bigr)^{(h)}\otimes_K\C_p(-k-1)&\stackrel{\Psi_k^{(h)}}{\lra} &\rH^0\bigl(X(w), \omega^{k+2}_w\bigr)^{(h)}\ho\C_p(-k-1)&\lra&0\\
\downarrow\cong&&\uparrow\cong\\
\rH^1\bigl(\Gamma, V_k(1)\bigr)^{(h)}\otimes_K\C_p(-k-1)& \stackrel{(p_2\circ\Phi_k^{(h)})}\lra&\rH^0\bigl(X(N,p),
\omega^{k+2}\bigr)^{(h)}\otimes_K\C_p(-k-1)&\lra&0
\end{array}
$$
Therefore we have a natural isomorphism, $G_K$ and Hecke equivariant
$$
S_k=S_V/t_kS_V={\rm Ker}\bigl(\Psi_k^{(h)}\bigr)\cong {\rm Ker}\bigl((p_2\circ\Phi_k)\bigr)=\rH^1\big(X(N,p), \omega^{-k}\bigr)^{(h)}\otimes \C_p.
$$
The exact sequence $\displaystyle 0\lra S_V\stackrel{t_k}{\lra}S_V\lra S_k\lra 0$ induces the exact sequence
$$
0\lra \widehat{W}_{K'_{\infty}}\bigl(S_V\bigr)\stackrel{t_k}{\lra}\widehat{W}_{K'_{\infty}}\bigl(S_V\bigr)\lra \widehat{W}_{K'_{\infty}}\bigl(S_k\bigr)\lra
\rH^1\bigl(H_{K'},S_V\bigr).
$$
The theory of almost \'etale extensions implies that $\rH^1\bigl(H_{K'},S_V\bigr)=0$ and therefore if we denote by $\phi_k$ the Sen operator attached to $S_k$, then
the image of $\{e_1,e_2,\ldots,e_q\}$ is a basis of $\widehat{W}_{K'_{\infty}}\bigl(S_k\bigr)$ in which $\phi_k$ has matrix $\bigl(\alpha_{ij}(k)\bigr)_{1\le i,j\le
q}$. But $\widehat{W}_{K'_{\infty}}\bigl(S_k\bigr)\cong \widehat{K}'_{\infty}\ho_K \rH^1\bigl(X(N,p),\omega^{-k}\bigr)$, therefore $\phi_k=0$. It follows that
$\alpha_{ij}(k)=0$ for infinitely many $k\in V(K)$, therefore $\alpha_{ij}=0$ for all $1\le i,j\le q$. It follows that $\phi_V(e_i)=0$ which implies that
$\gamma(e_i)=e_i$ for all $1\le i\le q$. Therefore the free $K'\otimes_KB_V$-module of rank $q$, $W_\ast$, is equal to $(S_V)^{G_{K'}}$, i.e. $S_V$ is a trivial
$G_{K'}$-module. We supposed that $K'/K$ was a finite Galois extension, therefore by \'etale descent $S_V$ is trivial as $G_K$-module. It follows that
$M_V^{(h)}\cong S_V\bigl(\chi\cdot\chi_V^{\rm univ}\bigr)$ as $G_K$-modules, where $S_V$ is a free $(B_V\ho_K\C_p)$-module of rank $q$ with trivial $G_K$-action.

\bigskip
Finally let us prove d) of theorem \ref{thm:muh}. We denote by $$\cH:={\rm Hom}_{(B_V\ho\C_p)}\Bigl(\rH^0\bigl(X(w), \omega_w^{\dagger,k_V+2}\bigr)\ho_K\C_p,
S_V\bigl(\chi\cdot\chi_V^{\rm univ}\bigr)\Bigr).$$ Then $\cH$ is a free $(B_V\ho \C_p)$-module of finite rank with continuous, semilinear action of $G_K$. Moreover,
the extension class of the exact sequence at b) corresponds to a cohomology class in $\rH^1\bigl(G_K, \cH\bigr)$. If we denote by $\phi$ the Sen operator of $\cH$,
a result of \cite{sen1} implies that $\det(\phi)\in B_V$ annihilates this cohomology group. Moreover $\det(\phi)\ne 0$, so if we localize the sequence at this
element it will split naturally as a short exact sequence of $G_K$-modules. This finally ends the proof of theorem \ref{thm:muh}.

\bigskip
\noindent {\bf Proof of corollary \ref{cor:consequences}.} Let us assume the hypothesis of the corollary, i.e. we have $U,h$ satisfying the assumption there. Let
$Z\subset U(\C_p)$ be the finite set defined in theorem \ref{thm:muh} b) and let first $k\in  U(K)-Z$. Then there exists a wide open disk $V\subset U$, defined
over $K$ such that $V(\C_p)\cap Z=\phi$ and $k\in V(K)$. By theorem \ref{thm:muh} c) we have an exact sequence of $(B_V\ho_K\C_p)$-modules with continuous
semilinear $G_K$-action
$$
0\lra S_V\bigl(\chi\cdot\chi_V^{\rm univ}\bigr)\lra \rH^1\bigl(\Gamma, D_U^{(h)}\bigr)\ho_K\C_p(1)\lra \rH^0\bigl(X(w), \omega_w^{\dagger, k_V+2}
\bigr)^{(h)}\ho_K\C_p\lra 0.
$$
As $k\in V(K)$ we may specialize this sequence and we obtain an exact sequence of $\C_p$-vector spaces with continuous, semilinear action of $G_K$
$$
0\lra S_k(k+1)\lra \rH^1\bigl(\Gamma, D_k\bigr)^{(h)}\ho_K\C_p(1)\lra \rH^0\bigl(X(w), \omega^{\dagger,k+2}_w\bigr) \ho_K\C_p\lra 0.
$$
Now $k\in V(K)\subset U(K)\subset \cW^\ast(K)$ so $k$ is an accessible weight associated to a pair $(s, i)$. Then it follows that if $s\ne -1$, the character
$\chi^{s+1}$ is a character of infinite order, therefore by the main result of \cite{tate}, the above sequence is naturally and uniquely split as a sequence of
$G_K$-modules. Therefore we choose $Z':=Z\cup \{k\in U(K)-Z\quad |\quad k=(s,i), s\ne -1\}.$ Then $Z'$ is finite and the corollary follows.

\bigskip
\noindent
Finally, in this article we do not give a precise geometric interpretation of the $B_V$-module $S_V$ of
theorem \ref{thm:muh} but we prove the following lemma. With $U, h, Z'$ and $V$ as in theorem \ref{thm:muh},
we denote by
$$
\mathbb{S}_V:=H^0\bigl(X(w), \omega_w^{\dagger, k_V}\otimes \Omega^1_{X(w)/K}\bigr)\subset H^0\bigl(X(w), \omega_w^{\dagger, k_V+2}\bigr),
$$
the $A_V$-module of families of overconvergent cuspforms over $V$. It is a Hecke-submodule of
the $B_V$-module of families of overconvergent modular forms over $V$,
$\mathbb{M}_V:=H^0\bigl(X(w), \omega_w^{\dagger, k_V+2}\bigr)$ and we have:

\begin{lemma}
\label{lemma:heckesv}
Let  $\ell$ be a positive prime integer such that $(\ell, Np)=1$. Then the characteristic polynomials of $T_{\ell}$ acting on $S_V^{(h)}$ and on $\mathbb{S}_V^{(h)}$ are equal.
\end{lemma}

\begin{proof}
Let us first consider an integer weight $k\in V$ such that $k>h-1$ and let us recall (\cite{faltingsmodular}) that  the natural
Poincar\'e pairing between $H^1\bigl(\Gamma, V_k(1)\bigr)$ and $H_c^1\bigl(\Gamma, V_k(1)\bigr)$ induces
Serre-duality between $H^1\bigl(X(N,p), \omega^{-k}\bigr)$ and $H^0\bigl(X(N,p), \omega^k\otimes
\Omega^1_{X(N,p)/K}\bigr)$. Therefore we can identify $H^1\bigl(X(N,p), \omega^{-k}\bigr)$ with
the $K$-dual of $H^0\bigl(X(N,p), \omega^k\otimes\Omega^1_{X(N,p)/K}\bigr)$, and so the characteristic polynomials
of $T_{\ell}$ acting on  $H^1\bigl(X(N,p), \omega^{-k}\bigr)$ and $H^0\bigl(X(N,p), \omega^k\otimes
\Omega^1_{X(N,p)/K}\bigr)$ are equal.

Now let $P_{\ell,i}(T)\in B_U[T]$, for $i=1,2$ be the characteristic polynomials of $T_{\ell}$ acting on
$S_V^{(h)}$ and respectively $\mathbb{S}_V^{(h)}$. If $k\in V$ is an integer weight such that $k>h-1$,
then the characteristic polynimials of $T_{\ell}$ acting on:
$$
S_V^{(h)}/t_kS_V^{(h)}\cong  H^1\bigl(X(N,p), \omega^{-k}\bigr)^{(h)}\mbox{ and on }\mathbb{S}_V^{(h)}/t_k\mathbb{S}_V^{(h)}\cong
H^0\bigl(X(N,p), \omega^k\otimes\Omega^1_{X(N,p)/K}\bigr)^{(h)}
$$
are $P_{\ell,1}(k)$ and $P_{\ell,2}(k)$. By the above argument $P_{\ell,1}(T)(k)=P_{\ell,2}(T)(k)$ for infinitely many $k\in V$
therefore $P_{\ell,1}(T)=P_{\ell,2}(T)$. Here by $P_{\ell,1}(T)(k)$ we mean the polynomial obtained by evaluating the coefficients
of $P_{\ell,1}(T)$ at $k$.
\end{proof}

\subsection{On the global Galois representations attached to overconvergent eigenforms}
\label{sec:globalrep}

In this section we give a geometric interpretation of the $G_{\Q}={\rm Gal}(\qbar/\Q)$-representation attached
to a {\it generic} overconvergent cuspidal eigenform.

Let us start by  fixing a slope $h\ge 0$, $h\in \Q$ and a wide open disk $U\subset \cW^\ast$ as
in the statement of theorem \ref{thm:muh}.

Let us recall (see section \S 5) that we may think of $\Gamma$ as the fundamental group of $Y(N,p)_{\C}$ for a
choice of a geometric generic point, therefore we have canonical isomorphisms of topological $B_U$-modules
$$
H^1\bigl(\Gamma, D_U\bigr)\cong H^1\bigl(X(N,p)^{\rm ket}_{\C}, \cD_U\bigr)\cong H^1\bigl(X(N,p)^{\rm ket}_{\qbar},
 \cD\bigr).
$$
Let us remark that the last $B_U$-module has a natural, continuous, $B_U$-linear action of $G_{\Q}$ with the property that
its restriction to $G_K$ (seen as an open subgroup of a decomposition group of $G_{\Q}$ at $p$) is
what we denoted in section \S 5 by $H^1\big(X(N,p)_{\Kbar}^{\rm ket}, \cD\bigr)$. Moreover,
as the $G_{\Q}$-action commutes with the action of the Hecke operators, in particular with the action of $U_p$, it induces a
$G_{\Q}$-module structure on the finite free $B_U$-module $H^1(\Gamma, D_U)^{(h)}$.
We also denote by $H^1_p\bigl(\Gamma, D_U\bigr)$ the image of the natural map
$H^1_c\bigl(\Gamma, D_U\bigr)\lra H^1\bigl(\Gamma, D_U\bigr)$. Then all these cohomology groups have natural
interpretations as \'etale cohomology groups with compact support, respectively \'etale parabolic cohomology, they are naturally
$G_\Q$ and Hecke modules. Moreover, $U$ will now be chosen  such that both $H^1\bigl(\Gamma, D_U\bigr)$ and
$H^1_p\bigl(\Gamma, D_U\bigr)$ have slope $\le h$-decompositions.

We have

\begin{theorem}
\label{thm:global1}
a) For every positive prime  integer $\ell$ with $(\ell, Np)=1$ the $G_{\Q}$-representations
$H^1\bigl(\Gamma, D_U(1)\bigr)$ and $H_p^1(\Gamma, D_U(1))$
are unramified at $\ell$.

b) Let us fix  $\ell$ as at a) above and denote by $\varphi_\ell$ a geometric Frobenius at
$\ell$ and by $T_\ell$ the Hecke operator both acting on $H_p^1\bigl(\Gamma, D_U(1)\bigr)^{(h)}$. Then the characteristic polynomials of
$\varphi_\ell$ and $T_\ell$ are equal.
\end{theorem}

\begin{proof}
a) is clear as $X(N,p)_{\Q_\ell}$ has a smooth proper model over $\Spec(\Z_\ell)$.
For b), let us denote by $P_i(T)\in B_U[T], i=1,2$ the characteristic polynomials of $\varphi_\ell$ and $T_\ell$
respectively. For every $k\in U\cap \Z$ with $k>h+1$ we have natural isomorphisms, equivariant for the $G_{\Q}$ and Hecke
actions
$$
H_p^1\bigl(\Gamma, D_U(1)\bigr)^{(h)}/t_kH_p^1\bigl(\Gamma, D_U(1)\bigr)^{(h)}\cong H_p^1\bigl(\Gamma, D_k(1)\bigr)^{(h)}\cong H_p^1(\Gamma, V_k(1)\bigr)^{(h)}.
$$
Moreover the characteristic polynomials of $\varphi_{\ell}$ and $T_{\ell}$ on the last group are $P_1(T)(k)$ and $P_2(T)(k)$
and by theorem 4.9 of \cite{deligne} they are equal: $P_1(T)(k)=P_2(T)(k)$. As there are infintely many
weights $k$ as above in $U$, it follows that $P_1(T)=P_2(T)$.
\end{proof}

Let now $Z'\subset U(\C_p)$ be the finite set of weights of corollary \ref{cor:consequences}.

 \begin{corollary}
\label{cor:repoverconvergent}

 Let $k\in U(\C_p)-Z'$ and let $f$ be an overconvergent cuspidal eigenform of weight $k+2$ and slope smaller or equal to
$h$. Let $K_f$ denote the finite extension of $K$ generated by all the eigenvalues of $f$. Then
$H^1\big(\Gamma, D_k(1)\bigr)^{(h)}_f$ is a $K_f$-vector space of dimension $2$ and it is the $G_{\Q}$-representation
attached to $f$ by the theory of pseudo-representations.
\end{corollary}

Before staring the proof of this corollary, let us explain its notations: we denote by $\mathbb{T}$ the $K$ subalgebra
of the $K$-endomorphism algebra  of $H^1\bigl(\Gamma, D_k(1)\bigr)^{(h)}$ generated by
the images of $T_{\ell}$ for all positive prime integers $\ell$ such that $(\ell, Np)=1$. The
overconvergent cuspidal eigenform $f$ determines a surjective $K$-algebra homomorphism
$\mathbb{T}\lra K_f$ sending $T_\ell$ to its $f$ eigenvalue. Then we denote by
$H^1\bigl(\Gamma, D_k(1)\bigr)^{(h)}_f:=H^1\bigl(\Gamma, D_k(1)\bigr)^{(h)}\otimes_{\mathbb{T}}K_f$.

\begin{proof}
As $f$ is an overconvergent cuspform we have $W_f:=H^1\bigl(\Gamma, D_k(1)\bigr)^{(h)}_f\cong H^1_p\bigl(\Gamma, D_k(1)\bigr)^{(h)}_f$. If we denote by $\varphi_\ell$ a geometric Frobenius acting on $W_f$, theorem \ref{thm:global1}
implies (by specialization at weight $k$) that the characteristic polynomial of $\varphi_\ell$ is equal to the characteristic
polynomial of $T_\ell$ on the same module.
Now we consider  $W_f\otimes_K \C_p$ and forget the global Galois action. Corollary 6.2 and Lemma 6.3 imply that
$W_f\otimes\C_p$ is a free two dimensional $K_f\otimes_K \C_p$-module which means that $W_f$ is two dimensional over
$K_f$. This ends the argument.
\end{proof}

\end{document}